\pdfoutput=1
\documentclass[oneside]{scrreprt}

\usepackage{amscd}
\usepackage{amsfonts}
\usepackage{amsmath}
\usepackage{amssymb}
\usepackage{amstext}
\usepackage{amsthm}
\usepackage{mathtools}
\usepackage{thmtools, thm-restate}

\usepackage[sc]{mathpazo}
\usepackage[utf8]{inputenc}

\usepackage{bbm}
\usepackage{array}
\usepackage{leftidx}
\usepackage{enumerate}
\usepackage[all, cmtip, arrow, matrix, curve]{xy}
\usepackage{enumerate}
\usepackage[nottoc]{tocbibind}
\usepackage{graphicx}
\usepackage{subfig}
\usepackage[margin=10pt,font=small,labelfont=bf,labelsep=endash]{caption}
\usepackage[section]{placeins} 
\usepackage{longtable} 
\usepackage{csquotes}

\usepackage{textcomp} 
\newcommand{\onehalf}{\frac{1}{2}}

\DeclareGraphicsExtensions{.pdf,.png,.jpg}

\newcolumntype{M}{>{$}c<{$}}

\newtheoremstyle{lecture}
  {9pt}
  {9pt}
  {}
  {}
  {\bfseries}
  {}
  { }
  {}
\theoremstyle{lecture}

\newtheorem{theorem}{Theorem}[section]
\newtheorem{lemma}[theorem]{Lemma}
\newtheorem*{lemma*}{Lemma}
\newtheorem{proposition}[theorem]{Proposition}
\newtheorem{corollary}[theorem]{Corollary}
\newtheorem{remark}[theorem]{Remark}
\newtheorem*{remark*}{Remark}
\newtheorem{definition}[theorem]{Definition}

\newcommand{\R}{\mathbb{R}}
\newcommand{\C}{\mathbb{C}}

\newcommand{\Z}{\mathbb{Z}}
\newcommand{\Proj}{\mathbb{P}}
\newcommand{\RProj}{\mathbb{RP}}

\newcommand{\cl}{\mathrm{cl}}

\newcommand{\SU}{\mathrm{SU}}

\newcommand{\SOR}[1]{\mathrm{SO}(#1,\mathbb{R})}

\newcommand{\su}{\mathfrak{su}}

\newcommand{\id}{\mathrm{id}}
\newcommand{\Id}{\mathrm{Id}}

\newcommand{\Gr}{\text{Gr}}

\newcommand{\dunion}{\dot{\cup}}

\newcommand{\Triple}[3]{\left(#1,\,#2,\,#3\right)}

\newcommand{\Pair}[2]{\left(#2,\,#1\right)}

\newcommand{\Bideg}[2]{\left(#1,\,#2\right)}
\newcommand{\Fix}{\text{Fix}\,}
\newcommand{\Diag}{\text{Diag}\,}

\newcommand{\sig}{\text{sig}\,}

\newcommand{\Deck}{\text{Deck}\,}
\newcommand\restr[2]{
  #1|_{#2}
  }
\newcommand\Restr[2]{%
  \left.\kern-\nulldelimiterspace 
  #1 
  \right|_{#2} 
  }

\renewcommand{\Re}[1]{\text{Re}\,#1}
\renewcommand{\Im}[1]{\text{Im}\,#1}
\renewcommand{\mod}[0]{\;\text{mod}\;}
\renewcommand{\setminus}{\smallsetminus}

\newcommand{\LatGen}[2]{\left<#1,#2\right>_{\mathbb{Z}}}
\newcommand{\hdim}{\text{dim}_H}
\newcommand{\pdim}{\text{dim}_P}
\newcommand{\Maps}[2]{{#2}^{#1}}
\newcommand{\closure}{\text{cl}\,}

\usepackage{xparse}

\DeclareDocumentCommand \I { G{} G{} } {%
  \ifstrempty{#2}{%
    \ifstrempty{#1}{%
      \mathbb{E} %
    }{%
      \mathbb{E}_{#1} %
    }%
  }{%
    \mathbb{E}_{#1,#2} %
  }%
}

\usepackage{cite}

\usepackage{etex}
\usepackage[margin=10pt,font=small,labelfont=bf,labelsep=endash]{caption}

\usepackage[mathscr]{eucal} 
\usepackage{booktabs}
\usepackage{placeins}

\usepackage{tikz}
\usepackage{tikz-cd}
\usetikzlibrary{decorations.pathreplacing,decorations.markings,arrows,calc}

\tikzset{
  on each segment/.style={
    decorate,
    decoration={
      show path construction,
      moveto code={},
      lineto code={
        \path [#1]
        (\tikzinputsegmentfirst) -- (\tikzinputsegmentlast);
      },
      curveto code={
        \path [#1] (\tikzinputsegmentfirst)
        .. controls
        (\tikzinputsegmentsupporta) and (\tikzinputsegmentsupportb)
        ..
        (\tikzinputsegmentlast);
      },
      closepath code={
        \path [#1]
        (\tikzinputsegmentfirst) -- (\tikzinputsegmentlast);
      },
    },
  },
  mid arrow/.style={postaction={decorate,decoration={
        markings,
        mark=at position .5 with {\arrow[scale=1.8,#1]{stealth}}
      }}},
}

\addtokomafont{disposition}{\rmfamily}

\begin{document}

\subject{Dissertation}
\title{On Time-Reversal Equivariant\\ Hermitian Operator Valued Maps\\
  from a 2D-Torus Phase Space}
\author{Moritz Schulte}
\date{\today}
\maketitle

\cleardoublepage
\tableofcontents

\addchap{Acknowledgements}

First of all, I would like to express my sincerest gratitude to my
supervisor Professor Huckleberry. His help and guidance have been
invaluable to me throughout the project. The mathematical discussions
we shared were crucial to my progress, in particular when he helped
with several of the proofs by suggesting key ideas. The work enabled
me to grasp what mathematical research is like.

I would also like to thank my parents, my family and my friends for
supporting me throughout my studies and research. 

And finally, there is somebody whose role throughout the critical
phase of my PhD project I would like to underline: Thank you,
Christina, for your unwavering support, especially when some of the
proofs had given me a hard time. Your part in the completion of this
paper can hardly be overestimated.

This research was conducted as part of the the Sonderforschungsbereich
TR12 (SFB TR12) ``Symmetries and Universality in Mesoscopic Systems''
and was partially financed by the Deutsche Forschungsgemeinschaft
(DFG) for which I am very thankful.


\chapter{Introduction}

In this paper we are primarily concerned with the equivariant homotopy
classification of maps from a torus phase space $X$ to
$\mathcal{H}_n^*$, the space of $n \times n$, non-singular hermitian
operators. The equivariance is with respect to a time-reversal
involution on $X$ and an involution on $\mathcal{H}_n^*$ defining a
certain symmetry class. This is in contrast to related works in this
area (e\,g. \cite{Kennedy}), which deal with the classification of
equivariant maps defined on \emph{momentum space}, where the
time-reversal involution on momentum space is given by $k \mapsto
-k$. Here we consider the model problem where the phase space is a
$2$-dimensional symplectic torus $(X,\omega)$ and the time-reversal
involution $T$ is antisymplectic, i.\,e.
\begin{align*}
  T^*(\omega) = -\omega.
\end{align*}
In this case it is known (see e.\,g. \cite{Seppala}, \cite{Tehrani})
that up to equivariant symplectic diffeomorphisms there are exactly
three classes of antisymplectic involutions on $X$. Representatives
for these three classes can be given by
\begin{align*}
  [z] \mapsto [\overline{z}] \;\;\text{,}\;\; [z] \mapsto [i\overline{z}] \;\;\text{and}\;\; [z] \mapsto [\overline{z} + \text{\textonehalf}],
\end{align*}
where $X$ is regarded as the square torus $\C/\Lambda$ defined by the
lattice $\Lambda = \LatGen{1}{i}$. In this paper we only deal with the
first two involutions, as these are the only ones, which have
fixed points. In particular, in local coordinates they are of the form
\begin{align*}
  (q,p) \mapsto (q,-p).
\end{align*}
We call the first one \emph{type I} involution and the second one
\emph{type II} involution. On $\mathcal{H}_n^*$ resp. $\mathcal{H}_n$
we define the involution, which we also denote by $T$, to be
\begin{align*}
  H \mapsto \overline{H} = \ltrans{H}.
\end{align*}
We believe, though, that the methods described in this paper, are also
applicable, when $X$ is equipped with a time-reversal involution
coming from the third class of (free) antisymplectic involutions on
$X$ and when physically relevant image spaces other than
$(\mathcal{H}_n,T)$ are considered.

Letting $G$ be the cyclic group $C_2$ of order two and $X$ and
$\mathcal{H}_n$ be equipped as above with $G$-actions, our main
results in chapter~\ref{ChapterClassification} give a complete
description of the spaces
\begin{align*}
  [X,\mathcal{H}_n^*]_G
\end{align*}
of $G$-equivariant homotopy classes of maps $X \to \mathcal{H}_n^*$,
where $X$ is equipped with the type I or the type II involution. These
descriptions are in terms of numerical topological invariants. It
turns out that the topological invariant given by the \emph{total
  degree} together with certain \emph{fixed point degrees} defines a
\emph{complete} topological invariant of equivariant homotopy. Recall
that an invariant is called \emph{complete} if the equality of the
invariants associated to two objects (equivariant maps in our case)
implies the equivalence of the objects (the maps are equivariantly
homotopic). It is important to emphasize that not all combinations of
total degree and fixed point degrees are realizable; a certain
elementary number theoretic condition is necessary and sufficient for
the existence of a map with given invariants.

In chapter~\ref{ChapterJumps} we study the situation where a certain
controlled degeneration is allowed. In formal terms this is a curve
\begin{align*}
  f\colon [-1,1] \times X \to \mathcal{H}_n
\end{align*}
where the image $\Im(f_t)$ is contained in $\mathcal{H}_n^*$ for $t
\not= 0$ and where the degeneration $f_0(x)$ being singular is only
allowed at isolated points in $X$. The interesting case is where
$f_{-1}$ and $f_{+1}$ are in distinct $G$-homotopy classes,
i.\,e. where we have a sort of \emph{jump curve}. In this case we show
exactly which jumps are possible and produce maps with given jumps.

This research was carried out in the context of the interdisciplinary
project SFB TR12. Thus we have attempted to present a work that is
mostly self-contained; for the convenience of the reader, the appendix
(Chapter~\ref{ChapterAppendix}) includes some well-known background
material.

\section{Outline}

The equivariant homotopy classification of maps $X \to
\mathcal{H}_n^*$ is contained in
chapter~\ref{ChapterClassification}. We begin by identifying the
connected components of $\mathcal{H}_n^*$. These are the subspaces
$\mathcal{H}_{(p,q)}$ of the matrices with $p$ positive and $q$
negative eigenvalues ($p + q = n$). It follows that this
\emph{eigenvalue signature} $(p,q)$ is a homotopy invariant of maps $X
\to \mathcal{H}_n^*$ and the classification of maps to
$\mathcal{H}_n^*$ reduces to the classification of maps to
$\mathcal{H}_{(p,q)}$, where $(p,q)$ takes on all possible eigenvalue
signatures. We then begin by examining the case where $X$ is equipped
with the type I involution and $n=2$. An ad-hoc computation shows that
$\mathcal{H}_{(2,0)}$ resp. $\mathcal{H}_{(0,2)}$ are equivariantly
contractible to plus resp. minus the identity matrix, whereas
$\mathcal{H}_{(1,1)}$ contains a $U(2)$-orbit diffeomorphic to $S^2$
as equivariant strong deformation retract. The induced $T$-action on
the sphere is given by a reflection along the $x,y$-plane.  This
motivates the need to classify equivariant maps $X \to S^2$. Also note
that the $2$-sphere can be equivariantly identified with the
projective space $\Proj_1$, on which the involution is the standard
real structure given by complex conjugation. Since $X$ and $S^2$ are
both $2$-dimensional manifolds, we obtain the (Brouwer) degree, which
we call \emph{total degree}, as a homotopy invariant of maps $X \to
S^2$. The equivariance requires that fixed point sets are mapped to
fixed point sets. The fixed point set in the torus -- for the type I
involution -- consists of two disjoint circles -- $C_0$ and $C_1$ --
while the fixed point set in the sphere consists of only the equator
$E$. This defines two additional invariants of equivariant homotopy:
The \emph{fixed point degrees}, which are the degrees of the restrictions
of the maps to the respective fixed point circles in $X$, regarded as
maps to the equator. This gives rise the \emph{degree triple map}
$\mathcal{T}$, which sends an equivariant map $f\colon X \to S^2$ to
the degree triple $\Triple{d_0}{d}{d_1}$, where $d$ is the total
degree of $f$ and $d_0$, $d_1$ denote the fixed point degrees of $f$. A
convenient property of the type I involution is the fact that a
fundamental region for this involution is given by a \emph{cylinder},
which we call $Z$. The boundary circles of the cylinder are the
fixed point circles in $X$. Maps $(Z,C_0 \cup C_1) \to (S^2,E)$ can be
uniquely equivariantly extended to the full torus $X$. It follows that
the equivariant homotopy classification of maps $X \to S^2$ reduces to
the (non-equivariant) homotopy classification of maps $(Z, C_0 \cup
C_1) \to (S^2,E)$. We then regard maps from the cylinder $Z = I \times
S^1$ to $S^2$ as \emph{curves} in the free loop space
$\mathcal{L}S^2$. After this translation of the problem and a certain
normalization procedure on the boundary circles, we are dealing with
the problem of computing homotopy classes of curves in
$\mathcal{L}S^2$ (with fixed endpoints). The key point now is to
define the \emph{degree map}, which associates to curves in
$\mathcal{L}S^2$ the degree of its equivariant extension to the full
torus. The degree map factors through the space of
homotopy class of curves. Furthermore it satisfies a certain
compatibility condition with respect to the concatenation of
curves. In particular, the degree map becomes a group homomorphism,
when it is restricted to a based fundamental group. This can be used for
proving that the restriction of the degree map to based fundamental
groups is injective. The injectivity of this map is then the basis for
the proof of the statement that maps with the same degree triple are
equivariantly homotopic. To summarize the above: The degree triple map
$\mathcal{T}$ is well-defined on the set of equivariant homotopy
classes of maps $X \to S^2$ and defines a bijection to a subset of
$\Z^3$. We call a degree triple, which is contained in the image of
the degree triple map, \emph{realizable}.  As a preperation for
describing the image of the degree triple map we express the
concatenation of curves in $\mathcal{L}S^2$ in terms of a binary
operation (``concatenation'') on the space of degree triples:
\begin{align*}
  \Triple{d_0}{d}{d_1} \bullet \Triple{d_1}{d'}{d_2} = \Triple{d_0}{d+d'}{d_2}.
\end{align*}
This concatenation operation is only defined for degree triples which
are compatible in the sense that the fixed point degrees in the middle
coincide. Using this formalism we observe that the image of the degree
triple map $\Im(\mathcal{T})$ is closed under the concatenation
operation. Furthermore we show that certain basic triples,
e.\,g. $\Triple{0}{1}{0}$, are not contained in the image of the
degree triple map. These are the key tools for completely describing
the image $\Im(\mathcal{T})$ using the formalism of degree triple
concatenations. It turns out that the number theoretic condition
\begin{align}
  \label{DegreeTripleCondition}
  d \equiv d_0 + d_1 \mod 2
\end{align}
is sufficient and necessary for a degree triple $\Triple{d_0}{d}{d_1}$
to be in the image of the degree triple map. This completely solves
the problem for the type I involution in the case $n=2$.

In the next step we reduce the classification problem for the type II
involution to the type I classification. The key observation here is
that maps $X \to S^2$, where $X$ is equipped with the type II
involution, can be normalized on a subspace $A \subset X$ such that
they push down to the quotient $X/A$, which happens to be
equivariantly diffeomorphic to $S^2$. The action on the $2$-sphere is
again given by a reflection along the $x,y$-plane. After another
normalization procedure, concerning the images of the two poles of
$S^2$, we prove a one-to-one correspondence between equivariant maps
$S^2 \to S^2$ and equivariant maps from a torus equipped with the type
I involution to $S^2$ having the one fixed point degree be zero. This
allows us to use the results for the type I involution.  In particular
it implies that we are dealing with degree \emph{pairs} instead of
degree triples. The degree pair map associates to an equivariant map
$f\colon X \to S^2$ its degree pair $\Pair{d_C}{d}$ where $d_C$ is the
fixed point degree and $d$ the total degree of $f$. Because of this
correspondence between maps defined on a type II torus and maps
defined on a type I torus, the condition (\ref{DegreeTripleCondition})
for degree triples to be in the image of the degree triple map induces
a corresponding condition for a pair $\Pair{d_C}{d}$ to be in the
image of the degree pair map:
\begin{align}
  \label{DegreePairCondition}
  d \equiv d_C \mod 2.
\end{align}
This completes the case $n=2$ for both involutions.

For the general case $n > 2$ we begin with an examination of the
connected components $\mathcal{H}_{(p,q)}$ of $\mathcal{H}_n^*$. As in
the case $n=2$ we prove that the components $\mathcal{H}_{(n,0)}$ and
$\mathcal{H}_{(0,n)}$ are equivariantly contractible to a
point. Fundamental for the following considerations is the result that
the components $\mathcal{H}_{(p,q)}$ with $0 < p,q < n$ have a
$U(n)$-orbit as equivariant strong deformation retract, where $U(n)$
acts on $\mathcal{H}_{(p,q)}$ by conjugation. This orbit is
equivariantly diffeomorphic to the complex Grassmann manifold
$\Gr_p(\C^n)$ on which the involution $T$ acts as
\begin{align*}
  V \mapsto \overline{V}.
\end{align*}
Thus, the problem of classifying equivariant maps $X \to
\mathcal{H}_{(p,q)}$ is equivalent to classifying equivariant maps $X
\to \Gr_p(\C^n)$. Now we fix the standard flag in $\C^n$ and consider
the -- with respect to this flag -- unique (complex) one-dimensional
Schubert variety $\mathcal{S}$ in $\Gr_p(\C^n)$. In this situation we
prove, using basic Hausdorff dimension theoretical results, that it is
possible to iteratively remove certain parts of the Grassmann manifold
$\Gr_p(\C^n)$ so that in the end we obtain an equivariant strong
deformation retract to the Schubert variety $\mathcal{S}$. The latter
is equivariantly biholomorphic to the complex projective space
$\Proj_1$ equipped with the standard real structure. This allows us to
use the results for the $n=2$ case. In terms of the invariants, the
only difference between the case $n = 2$ and $n > 2$ is that in the
former case the fixed point degrees are integers from the set $\Z$
while in the latter case they are only from the set $\{0,1\}$. This
stems from the fact that in the case $n>2$ the fundamental group of
the $T$-fixed points in the Grassmann manifold is not infinite cyclic
anymore, but cyclic of order two. In this setup, the conditions for
degree triples resp. pairs to be realizable are exactly
(\ref{DegreeTripleCondition}) and (\ref{DegreePairCondition}). This
completes the classification of equivariant maps $X \to
\mathcal{H}_n^*$.

In chapter~\ref{ChapterJumps} we construct curves
\begin{align*}
  H\colon [-1,1] \times X \to \mathcal{H}_n
\end{align*}
of equivariant maps with the property that the image $\Im(H_t)$ is
contained in $\mathcal{H}_n^*$ for $t\not= 0$ and $H_{-1}$ and
$H_{+1}$ are not equivariantly homotopic as maps to
$\mathcal{H}_n^*$. This is only possible if we allow a certain
degeneration at $t=0$, which means that there exists a non-empty
\emph{singular set} $S(H_0)$ consisting of the points $x \in X$ such
that the matrix $H_0(x)$ is singular. To make the construction
non-trivial we require the singular set $S(H_0)$ to be discrete.  As
in the previous chapter we first consider the type I involution in the
special case $n=2$. As a first result we obtain that jump curves from
$H_{-1}$ to $H_{+1}$ with discrete singular set can only exist if the
eigenvalue signatures of $H_{-1}$ and $H_{+1}$ coincide. Hence we fix
an eigenvalue signature $(p,q)$. The construction of \emph{jump
  curves} is based on the decomposition of degree triples as a
concatenation of simpler triples. First we show how to construct jumps
for certain ``model maps'' for basic degree triples. Subsequently we
extend this method such that we can construct jumps from any
equivariant homotopy class to any other -- as long as the eigenvalue
signature remains unchanged.

For the type II involution in the case $n=2$, we can employ the same
correspondence that has been used during the homotopy classification
to turn jump curves for the type I involution into corresponding jump
curves for the type II involution. This completes the construction of
jump curves in the case $n=2$.

For the general case $n > 2$ we construct jumps curves using the
embedding of the Schubert variety $\mathcal{S} \cong \Proj_1$ into
$\Gr_p(\C^n)$. In other words: The jump curves we construct in the
higher dimensional situation really come from jump curves in the case
$n=2$.


\chapter{Equivariant Homotopy Classification}
\label{ChapterClassification}

As indicated in the introduction, we may regard the torus $X$ as being
defined by the standard square lattice $\Lambda = \LatGen{1}{i}$. As a
complex manifold it carries a canonical orientation. Furthermore, we
regard $X$ as being equipped with the real structure $T([z]) =
[\overline{z}]$ (type I) or with the real structure $T([z]) =
[i\overline{z}]$ (type II). Recall that $G$ denotes the cyclic group
$C_2$ of order two. The $G$-action on $\mathcal{H}_n$ is assumed to be
given by matrix transposition or -- equivalently -- complex
conjugation:
\begin{align*}
  T\colon \mathcal{H}_n &\to \mathcal{H}_n\\
  H &\mapsto \ltrans{H} = \overline{H}.
\end{align*}
The goal in this chapter is to completely describe the sets
$[X,\mathcal{H}_n^*]_G$ of $G$-\parbox{0pt}{}equivariant homotopy classes of maps
$X \to \mathcal{H}_n^*$ ($n \geq 2$). Recall the basic definition of
equivariant homotopy:
\begin{definition}
  Let $X$ and $Y$ be two $G$-spaces and $f_0, f_1\colon X \to Y$ two
  $G$-maps. Note that $I \times X$ ($I$ is the unit interval $[0,1]$)
  can be regarded as a $G$-space by defining $T(t,x) = (t,T(X))$. Then
  $f_0$ and $f_1$ are said to be \emph{$G$-homotopic} (or
  \emph{equivariantly homotopic}\footnote{We use the terms
    ``$G$-homotopy'' and ``equivariant homotopy'' interchangeably.})
  if there exists a $G$-map $f\colon I \times X \to Y$ such that
  $f(0,\cdot) = f_0$ and $f(1,\cdot) = f_1$.
\end{definition}

The first invariant of maps $X \to \mathcal{H}_n^*$ we consider is the
\emph{eigenvalue signature}. We begin with the following
\begin{definition}
  Let $H$ be a non-singular $n \times n$ matrix and
  \begin{align*}
    \lambda^+_1 \geq \ldots \geq \lambda^+_p > 0 > \lambda^-_1 \ldots \geq \lambda^-_q
  \end{align*}
  the eigenvalues of $H$ (repetitions allowed). Then we call $(p,q)$
  the \emph{eigenvalue signature} (or simply \emph{signature}) of $H$
  and set $\sig H = (p,q)$.
\end{definition}
The connected components of the space $\mathcal{H}_n^*$ are the $n+1$
open subspaces $\mathcal{H}_{(p,q)}$, where
\begin{align*}
  \mathcal{H}_{(p,q)} = \left\{H \in \mathcal{H}_n^*\colon \sig H = (p,q)\right\}.
\end{align*}
Thus we can write
\begin{align*}
  [X,\mathcal{H}_n^*]_G = \dot{\bigcup_{p+q=n}}\;[X,\mathcal{H}_{(p,q)}]_G.
\end{align*}
In particular this means that the signature $\sig(H)$ for maps
$H\colon X \to \mathcal{H}_n^*$ is well-defined and a topological
invariant. It turns out that the components $\mathcal{H}_{(n,0)}$ and
$\mathcal{H}_{(0,n)}$ are equivariantly contractible but the
components $\mathcal{H}_{(p,q)}$ with $0<p,q<n$ are topologically
interesting as they have Grassmann manifolds as equivariant strong
deformation retract. The classification of maps to a component
$\mathcal{H}_{(p,q)}$ for a fixed (non-definite) eigenvalue signature
$(p,q)$ is in this way reduced to the problem of classifying
equivariant maps to complex Grassmann manifolds $\Gr_p(\C^n)$ on which
the $G$-action is given by standard complex conjugation.

As a preparation for proving a general statement about the set
$[X,\mathcal{H}_n^*]_G$, we first consider the special case $n=2$
(theorem~\ref{HamiltonianClassificationRank2},
p.~\pageref{HamiltonianClassificationRank2}). As mentioned earlier,
the only non-trivial cases occur for maps whose images are contained
in the components $\mathcal{H}_{(p,q)}$ of mixed signature. In the
case $n=2$ this boils down to $\mathcal{H}_{(1,1)}$.  This space has
the $2$-sphere, equipped with the involution given by a reflection
along the $x,y$-plane, as equivariant strong deformation
retract.\footnote{The 2-sphere is regarded as being embedded as the
  unit sphere in $\R^3$} Thus, the problem of describing
$[X,\mathcal{H}_{(1,1)}]_G$ is reduced to classifying equivariant maps
$X \to S^2$. We do this for the type~I and the type~II involution
seperately, but it turns out that the type II case can be reduced to
the type I case (see theorem~\ref{Classification1} and
theorem~\ref{Classification2}). The result is that for both
involutions the complete invariants of maps $X \to S^2$ consist of two
pieces of data: First the \emph{total degree} $d$ (which is also an
invariant of non-equivariant homotopy) and second the so-called
\emph{fixed point degrees} -- two integers (named $d_0, d_1$) for the
type I involution and one integer (named $d_C$) for the type~II
involution. Note that the notion of fixed point degree only makes
sense for \emph{equivariant} maps, as they are required to map the
fixed point set in $X$ into the fixed point set in $S^2$. Furthermore
we obtain the statement that not all combinations of total degree and
fixed point degrees are realizable by $G$-maps. As mentioned in the
introduction, the conditions which are sufficient and necessary are
$d \equiv d_0 + d_1 \mod 2$ for the type I involution and
$d \equiv d_C \mod 2$ for the type II involution.

After having proven a classification statement for $n=2$, we observe
that this case is really the fundamental case to which the general
case ($n$ arbitrary) can be reduced to by means of a retraction
argument. In each Grassmannian $\Gr_p(\C^n)$ there exists a Schubert
variety $\mathcal{S} \cong \Proj_1$; after fixing a full flag in
$\C^n$, this the unique one-dimensional (over $\C$) Schubert variety
which generates the second homology groups
$H_2(\Gr_p(\C^n),\Z)$. After iteratively cutting out certain parts of
a general Grassmannian $\Gr_p(\C^n)$ we obtain an equivariant
retraction to this embedded Schubert variety. It can be equivariantly
identified with $\Proj_1$ on which the involution is given by the
standard real structure:
\begin{align*}
  [z_0:z_1] \mapsto \left[\overline{z_0}:\overline{z_1}\right]
\end{align*}
This space can be equivariantly identified with $S^2$. Hence, we can
use the previous results for the case $n=2$.

\section{Maps to $\mathcal{H}_2$}
\label{SectionN=2}

In this section we provide a complete description of the set
$[X,\mathcal{H}_2^*]_G$, where $X$ is equipped with a real structure
$T$, which is either the type I or the type II involution. This case
is of particular importance, since all the other cases (i.\,e.
$n > 2$) can be reduced to this one. Note that a general equivariant
map $H\colon X \to \mathcal{H}_2$ is of the form
\begin{align}
  \label{GeneralRank2Map}
  H =
  \begin{pmatrix}
    a_1 & b \\
    \overline{b} & a_2
  \end{pmatrix},
\end{align}
where $a_1$, $a_2$ are functions $X \to \R$ and $b$ is a function $X
\to \C$ satisfying
\begin{align}
  \label{N=2MatrixEquivarianceConditionsA}
  &a_j \circ T = a_j \;\text{for $j=1,2$}\\
  \label{N=2MatrixEquivarianceConditionsB}
  &b \circ T = \overline{b}
\end{align}
A map into $\mathcal{H}_n^*$ has the additional property that
$\det H = a_1 a_2 - |b|^2$ is nowhere zero.  As we have mentioned
previously, the space $\mathcal{H}_2^*$ decomposes into the disjoint
union
\begin{align*}
  \mathcal{H}_{(2,0)} \,\dunion\, \mathcal{H}_{(1,1)} \,\dunion\, \mathcal{H}_{(0,2)},
\end{align*}
where $(2,0)$, $(1,1)$ and $(0,2)$ are the possible eigenvalue
signatures for maps $X \to \mathcal{H}_n^*$. First we note:
\begin{remark}
  \label{Rank2DefiniteComponentsContractible}
  The components $\mathcal{H}_{(2,0)}$ and $\mathcal{H}_{(0,2)}$ have
  $\{\pm \I{2}\}$ as equivariant strong deformation
  retract.\footnote{Here, $\I{k}$ is the $k \times k$ identity
    matrix.}
\end{remark}
\begin{proof}
  It suffices to consider the component $\mathcal{H}_{(2,0)}$, as the
  proof for the other is exactly the same. We define the following
  strong deformation retract:
  \begin{align*}
    \rho\colon I \times \mathcal{H}_{(2,0)} &\to \mathcal{H}_{(2,0)}\\
    (t,H) &\mapsto (1-t)H + t\I{2}.
  \end{align*}
  We have to check that $\rho$ really maps into
  $\mathcal{H}_{(2,0)}$. For this, let $H$ be any positive-definite
  hermitian matrix. Then $(1-t)H$ is also positive-definite for $t \in
  (0,1)$. If $\det((1-t)H + t\I{2}) = 0$, this means that the negative
  number $-t$ is an eigenvalue of $(1-t)H$, which is a
  contradiction. Equivariance of $\rho_t$ follows from the fact that
  the equivariance conditions (\ref{N=2MatrixEquivarianceConditionsA})
  and (\ref{N=2MatrixEquivarianceConditionsB}) are compatible with
  scaling by real numbers.
\end{proof}

At this point we begin using the space $i\mathfrak{su}_2$, which is
the vector space of traceless hermitian operators:
\begin{align*}
  i\mathfrak{su}_2 = \left\{
    \begin{pmatrix}
      a & \phantom{-}b \\
      \overline{b} & -a
    \end{pmatrix}\colon a \in \R \;\text{and}\; b \in \C\right\}.
\end{align*}
It can be identified with $\R^3$ via the linear isomorphism
\begin{align}
  \label{EmbeddingU2OrbitAsUnitSphere}
  \Psi\colon \begin{pmatrix}
    a & b \\
    \overline{b} & -a
  \end{pmatrix} \mapsto
  \begin{pmatrix}
    a \\
    \Re(b) \\
    \Im(b)
  \end{pmatrix}.
\end{align}
We regard $\R^3$ as being equipped with the involution, which
reflects along the $x,y$-plane:
\begin{align*}
  \begin{pmatrix}
    x \\
    y \\
    z
  \end{pmatrix} \mapsto
  \begin{pmatrix}
    \phantom{-}x \\
    \phantom{-}y \\
    -z
  \end{pmatrix}.
\end{align*}
With respect to this action, the diffeomorphism $i\mathfrak{su}_2
\xrightarrow{\;\sim\;} \R^3$ is equivariant. Regarding the topology of
the component $\mathcal{H}_{(1,1)}$ we state:
\begin{remark}
  \label{Rank2TraceZeroReduction}
  The component $\mathcal{H}_{(1,1)}$ has
  $i\mathfrak{su}_2\setminus\{0\}$ as equivariant, strong deformation
  retract.
\end{remark}
\begin{proof}
  A general matrix in $i\mathfrak{su}_2$ is of the form
  \begin{align*}
    \begin{pmatrix}
      a & \phantom{-}b \\
      \overline{b} & -a
    \end{pmatrix}
  \end{align*}
  with $a$ real and $b$ complex. Its eigenvalues are
  \begin{align*}
    \pm\sqrt{a^2 + |b|^2}.
  \end{align*}
  This implies that $i\mathfrak{su}_2\smallsetminus\{0\}$ is contained
  in $\mathcal{H}_{(1,1)}$. Now we prove that $\mathcal{H}_{(1,1)}$
  has the space $i\mathfrak{su}_2\setminus\{0\}$ as equivariant,
  strong deformation retract. A computation shows that the two
  eigenvalue functions $\lambda_1,\lambda_2$ of the general map $H$
  (\ref{GeneralRank2Map}) are
  \begin{align*}
    \lambda_{1,2} = \frac{a_1+a_2}{2} \pm \sqrt{\left(\frac{a_1 + a_2}{2}\right)^2 + |b|^2 - a_1 a_2}.
  \end{align*}
  It is a straightforward computation to show that $\det H > 0$
  implies that $H$ is positive or negative definite. From this we can
  conclude that the eigenvalue signature being $(1,1)$ implies $\det H
  < 0$, that is $|b|^2 - a_1 a_2 > 0$. Now define
  \begin{align*}
    c = \frac{a_1 + a_2}{2}
  \end{align*}
  and consider the strong deformation retract
  \begin{align*}
    \rho\colon I \times \mathcal{H}_{(1,1)} &\to \mathcal{H}_{(1,1)}\\
    (t,H) &\mapsto (1-t)H - tc\I{2}.
  \end{align*}
  For this to be well-defined we need $\det(\rho(t,H)) <
  0$ for all $t$ and $H$:
  \begin{align*}
    \det(\rho(t,H)) &= (a_1 - tc)(a_2 - tc) - |b|^2\\
                    &= a_1 a_2 - |b|^2 - tc(a_1 + a_2) + (tc)^2\\
                    &< - tc(a_1 + a_2) + (tc)^2\\
                    &= -\frac{(a_1 + a_2)^2}{2} + t\frac{(a_1 + a_2)^2}{4}\\
                    &= \left(\frac{a_1 + a_2}{2}\right)^2(t-2) < 0.
  \end{align*}
  Furthermore, $\rho$ stabilizes $i\mathfrak{su}_2\setminus\{0\}$ and
  $\rho(1,\mathcal{H}_{(1,1)}) = i\mathfrak{su}_2\setminus\{0\}$. As
  in the previous remark, equivariance of $\rho_t$ follows since the
  equivariance conditions (\ref{N=2MatrixEquivarianceConditionsA}) and
  (\ref{N=2MatrixEquivarianceConditionsB}) are compatible with scaling
  by real numbers.
\end{proof}

Clearly, the isomorphism $\Psi$ maps the zero matrix to the
origin. The space $\R^3\smallsetminus\{0\}$ has the unit $2$-sphere as
equivariant strong deformation retract. Under $\Psi$ this corresponds
to the $U(2)$ orbit of $\I{1}{1} = \Diag(1,-1)$ in
$i\mathfrak{su}_2$. This orbit consists of the matrices
\begin{align*}
  \left\{
    \begin{pmatrix}
      |a|^2 - |b|^2 & 2a\overline{c} \\
      2\overline{a}c & |b|^2 - |a|^2
    \end{pmatrix}\colon \;\text{where}\;
    \begin{pmatrix}
      a & b \\
      c & d
    \end{pmatrix}\;\text{is unitary.}\right\}
\end{align*}

This proves the following
\begin{proposition}
  \label{Rank2MixedSignatureReduction}
  The space $i\mathfrak{su}_2\setminus\{0\}$ has the $U(2)$-orbit of
  $\I{1}{1}$ as equivariant strong deformation retract. The above
  isomorphispm $\Psi$ equivariantly identifies the $U(2)$-orbit with
  the unit sphere in $\R^3$ on which the involution acts as a
  reflection along the $x,y$-plane. \qed
\end{proposition}

This reduces the classification problem to the classification of
equivariant maps $X \to S^2$. In order to discuss mapping degrees of
maps to $S^2$ we need to fix an orientation on the sphere and on its
equator $E$. We choose the orientation on $S^2$ to be defined by an
outer normal vector field on the sphere. Furthermore we define the
orientation on the equator $E$ to be such that the loop
\begin{align}
  \label{DegreeOneMapOnEquator}
  S^1 &\to E \subset S^2 \subset \R^3\\
  z &\mapsto
  \begin{pmatrix}
    \cos(\arg z)\\
    \sin(\arg z)\\
    0
  \end{pmatrix}\nonumber
\end{align}
has positive degree one\footnote{These choices are arbitrary. Choosing
  the opposite orientation on the sphere or on the equator only
  changes the signs appearing in concrete examples, but not the
  general discussion.}. Later, when discussing the case $n > 2$, it
will become important to equivariantly identify $S^2$ with the
projective space $\Proj_1$. The action on $\Proj_1$ is given by
standard complex conjugation
\begin{align*}
  [z_0:z_1] \mapsto \left[\overline{z_0}:\overline{z_1}\right].
\end{align*}
\label{P1S2OrientationDiscussion}Note that $\Proj_1$ is canonically
oriented as a complex manifold. Let $\psi$ be an equivariant,
orientation-preserving diffeomorphism $S^2 \to \Proj_1$. This
identifies the compactified real line $\smash{\widehat{\R}}$ with the equator
$E$. We can assume that this identification of $S^2$ with $\Proj_1$ is
such that a loop of degree $+1$ into the compactified real line
$\smash{\widehat{\R}} \subset \Proj_1$ is given by
\begin{align*}
  S^1 &\to \RProj_1\\
  z &\mapsto
  \begin{cases}
    \infty & \;\text{if $\arg(z) \in 2\pi\Z$}\\
    \tan\left(\frac{\arg z - \pi}{2}\right) & \;\text{else}\\
  \end{cases}.
\end{align*}
That is, the loop starts at $\infty$, then traverses the negative real
numbers until it reaches zero, then it traverses the positive real
numbers until it reaches $\infty$ again.

As indicated in the introduction, the relevant topological invariants
of maps $X \to S^2$ consist of two pieces of data:
\begin{enumerate}
\item The \emph{total degree}, which is the usual Brouwer degree, and
\item The \emph{fixed point degrees}, which are the Brouwer degrees of
  the restrictions of the maps to the components of the fixed point sets
  of the respective involution on $X$.
\end{enumerate}
For the type I involution this fixed point set consists of two
disjoint circles. Hence we obtain \emph{degree triples} of the form
$\Triple{d_0}{d}{d_1}$ as invariants, where the $d_j$ are the fixed
point degrees and $d$ is the total degree. For the type II involution
the fixed point set is a single circle, which is why the invariants we
work with are \emph{degree pairs} of the form $\Pair{d_C}{d}$ ($d$
being the total degree and $d_C$ the fixed point degree). We call a
degree triple resp. a degree pair \emph{realizable} if it occurs as
invariant of a $G$-maps. The exact definitions will be given in
definition~\ref{DefinitionTriple} and
definition~\ref{DefinitionDegreePair}.

The main result in this section is the following
\begin{restatable*}{theorem}{HamiltonianClassificationRankTwo}
  \label{HamiltonianClassificationRank2}
  Let $X$ be a torus equipped with either the type I or the type II
  involution. Then:
  \begin{enumerate}[(i)]
  \item The sets $[X,\mathcal{H}_{(2,0)}]_G$ and
    $[X,\mathcal{H}_{(0,2)}]_G$ are trivial (i.\,e. one-point sets).
  \item Two $G$-maps $X \to \mathcal{H}_{(1,1)}$ are $G$-homotopic iff
    their degree triples (type I) resp. their degree pairs (type II)
    agree.
  \item The realizable degree triples $\Triple{d_0}{d}{d_1}$ (type I)
    resp. degree pairs $\Pair{d_C}{d}$ (type II) are exactly those
    which satisfy
    \begin{align*}
      d \equiv d_0 + d_1 \mod 2 \;\;\text{resp.}\;\;d \equiv d_C \mod 2.
    \end{align*}
  \end{enumerate}
\end{restatable*}
The proof will be given in
section~\ref{SectionHamiltonianClassificationRank2}
(p.~\pageref{HamiltonianClassificationRankTwoProof}).

\subsection{Maps to $S^2$}
\label{SectionMapsTo2Sphere}

After the reduction described above, we can regard $G$-maps
$X \to \mathcal{H}_{(1,1)}$ as having their image contained in the
oriented $2$-sphere $S^2 \subset \R^3$. Throughout this section we use
a fixed orientation-preserving identification of the equator $E$ with
$S^1$. Furthermore, we let $p_0 \in E$ be the fixed base point in the
equator, which corresponds -- under this identification -- to
$1 \in S^1$. This point will be used later for bringing the maps into
a convenient “normal form”. We will make frequent use of the following
theorem:
\begin{theorem}
  \label{HopfTheorem}
  Let $M$ be a closed, $n$-dimensional manifold. Then maps $M \to S^n$
  are homotopic if and only if their degrees agree.
\end{theorem}
\begin{proof}
  See e.\,g. \cite[p.~50]{Milnor}.
\end{proof}

An important tool for the classification will be the study of
homotopies on the fixed point sets. For this we start with a general
definition:
\begin{definition}
  \label{DefinitionFixpointDegree} (Fixed point degree) Let $f\colon M
  \to Y$ be a $G$-map between $G$-manifolds. Assume that the fixed point
  set $Y^G$ is a closed, oriented manifold $K$ of dimension $m$ and
  that the fixed point set $M^G$ is the disjoint union
  \begin{align*}
    M^G = \bigcup_{j=0,\ldots,k} K_j
  \end{align*}
  of closed, oriented (connected) manifolds $K_j$ each of dimension
  $m$. Then we define the \emph{fixed point degrees} to be the
  $k$-tuple $\left(d_0,\ldots,d_k\right)$ where $d_j$ is the degree of
  the map $\restr{f}{K_j}\colon K_j \to K$.
\end{definition}

\begin{remark}
  \label{FixpointDegreeInvariant}
  The fixed point degrees are invariants of $G$-homotopies.
\end{remark}
\begin{proof}
  Let $M$ and $Y$ be two $G$-manifolds as in
  definition~\ref{DefinitionFixpointDegree} and let $f,f'$ be two
  $G$-homotopic maps $M \to Y$. Denote by $(d_0,\ldots,d_k)$
  resp. $(d'_0,\ldots,d'_k)$ their fixed point degrees. Assume that
  $d_j \neq d'_j$ for some $j$. By assumption, there exists a
  $G$-homotopy $H\colon I \times X \to Y$ from $f$ to $f'$. This
  homotopy can be restricted to $I \times K_j$, which can be regarded
  as a homotopy $h\colon I \times K_j \to K$. The degree $d_j$ is the
  degree of the map $h_0$ and $d'_j$ is the degree of the map
  $h_1$. But the degree is a homotopy invariant, hence $d_j \neq d'_j$
  yields a contradiction.
\end{proof}

From now on we handle the involution types I and II seperately.

\subsubsection{Type I}
\label{SectionMapsToSphereClassI}

In this section the real structure $T$ on $X$ is the type I
involution:
\begin{align*}
  T\colon X &\to X\\
  [z] &\mapsto [\overline{z}].
\end{align*}
Furthermore, we regard the torus $X$ as being equipped with the
structure of a $G$-CW complex (see e.g. \cite[p.~16]{May} or
\cite[p.~1]{ShahArticle}). Such a $G$-CW structure is depicted in
figure~\ref{FigureClassICWDecomposition}. We continue with the
\emph{cylinder reduction} method.

\paragraph{Cylinder Reduction}

In order to describe the cylinder reduction we need the notion of a
fundamental region:
\begin{definition}
  \label{DefinitionFundamentalRegion}
  A \emph{fundamental region} for the $G$-action on $X$ is a connected
  subset $R \subset X$ such that each $G$-orbit in $X$ intersects $R$
  in a single point.
\end{definition}

\begin{figure}%
  \centering
  \subfloat[$G$-CW structure of the torus $X$.]{%
    \label{FigureClassICWDecomposition}%
    \centering%
    \begin{tikzpicture}[scale=0.8]
      \draw
      (0,0) coordinate (A) node[below left] {$e^0_1$}
      (4,0) coordinate (B) node[below right] {$e^0_1$}
      (4,4) coordinate (C) node[above right] {$e^0_1$}
      (0,4) coordinate (D) node[above left] {$e^0_1$}
      (0,2) coordinate (E1) node[left] {$e^0_2$}
      (4,2) coordinate (E2) node[right] {$e^0_2$};

      \draw[postaction={on each segment={mid arrow=black}}]
      (A) -- node [left] {$e^1_{2,1}$} (E1)
      (B) -- node [right] {$e^1_{2,1}$} (E2)
      (E1) -- node [left] {$e^1_{2,2}$} (D)
      (E2) -- node [right] {$e^1_{2,2}$} (C);

      \draw[color=red,postaction={on each segment={mid arrow=red}}]
      (A) -- node [below] {$e^1_1$} (B)
      (D) -- node [above] {$e^1_1$} (C);
      \draw[color=red]
      (E1) -- node [above] {$e^1_3$} (E2);
      
      \fill [radius=2pt] (0,0) circle
      [] (4,0) circle
      [] (4,2) circle
      [] (0,2) circle
      [] (0,4) circle
      [] (4,4) circle;

      \useasboundingbox (-2,5.5) rectangle (5.5,-2.5);
    \end{tikzpicture}}
  \subfloat[Fundamental region for the type I involution on $X$.]{%
    \label{FigureClassIFundamentalRegion}%
    \centering%
    \begin{tikzpicture}[scale=0.8]
      \draw
      (0,0) coordinate (A) node[below left] {$e^0_1$}
      (4,0) coordinate (B) node[below right] {$e^0_1$}
      (0,2) coordinate (E1) node[above left] {$e^0_2$}
      (4,2) coordinate (E2) node[above right] {$e^0_2$};

      \draw [postaction={on each segment={mid arrow=black}}]
      (A) -- node [left] {$e^1_{2,1}$} (E1)
      (B) -- node [right] {$e^1_{2,1}$} (E2);
      \draw[color=red]
      (A) -- node [below] {$e^1_1$} (B)
      (E1) -- node [above] {$e^1_3$} (E2);

      \fill [radius=2pt] (0,0) circle
      [] (4,0) circle
      [] (4,2) circle
      [] (0,2) circle;

      \draw[color=red,thick,decorate,decoration={brace,mirror,amplitude=10pt}] (0,-0.8) -- (4,-0.8) node[shift={(-1.4,-0.5)},below] {$C_0$};
      \draw[color=red,thick,decorate,decoration={brace,amplitude=10pt}] (0,2.8) -- (4,2.8) node[shift={(-1.4,+0.5)},above] {$C_1$};

      \useasboundingbox (-1.5,5) rectangle (5.5,-2.5);
    \end{tikzpicture}}
  \caption{Type I $G$-CW setup of the torus $X$.}
  \label{3figs}
\end{figure}
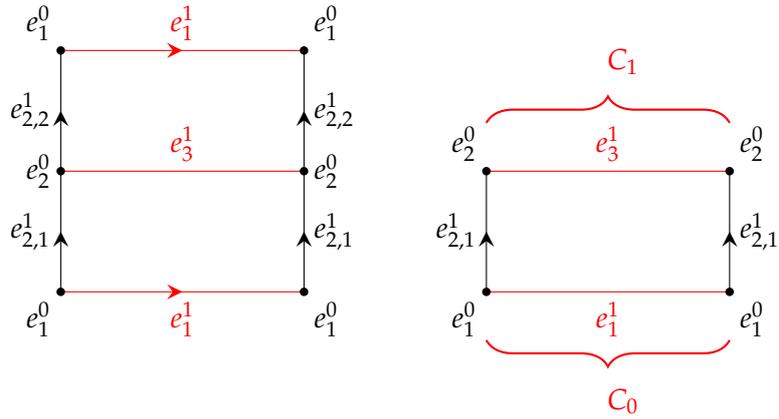

The type I involution is particularly convenient to work with, since
for this action there exists a sub-$G$-CW complex that is a
fundamental region for the $T$-action (see
figure~\ref{FigureClassIFundamentalRegion}). Geometrically this
fundamental region is a closed cylinder (including its boundary
circles), which we denote by $Z$. The fixed point set $X^G$ is the
disjoint union of two circles $C_0$ and $C_1$.  The cylinder $Z$ is
bounded by these circles. We identify $Z$ with $I \times S^1$. In this
setup, the boundary circles are $C_0 = \{0\} \times S^1$ and $C_1 =
\{1\} \times S^1$. For convenience we also set $C = C_0 \cup C_1$.  In
this situation we note the following general lemma:
\begin{lemma}
  \label{Class1EquivariantExtension}
  Let $Y$ be a $G$-space, $Y^G$ the $T$-fixed point set in $Y$ and
  $$
  f\colon (Z,C) \to (Y,Y^G)
  $$
  a map. Then there exists a unique equivariant extension of $f$ to a
  $G$-map $X \to Y$ which we denote by $\smash{\hat{f}}$. The same applies
  to homotopies of such maps.
\end{lemma}
\begin{proof}
  We only prove the statement for a given homotopy
  \begin{align*}
    H\colon I \times (Z,C) \to (Y,Y^G).
  \end{align*}
  Denote the complementary cylinder by $Z'$ (note that $Z \cap Z' = C
  = C_0 \cup C_1$).  Then we can construct the equivariant extension
  $\smash{\widehat{H}}$ of $H$ by defining
  \begin{align*}
    \widehat{H}\colon I \times X &\to Y\\
    (t,x) &\mapsto
    \begin{cases}
      H(t,x) & \;\text{if $x \in Z$}\\
      T(H(t,T(x))) & \;\text{if $x \in Z'$}\\
    \end{cases}
  \end{align*}
  Since $H(t,\cdot)$ is a map $(Z,C) \to (Y,Y^G)$, this definition of
  $\smash{\widehat{H}}$ is well-defined on the boundary circles $C_0$ and
  $C_1$ and globally continuous.
\end{proof}

For the concrete situation $(Y,Y^G) = (S^2,E)$ this implies:
\begin{remark}
  \label{CylinderReduction} (Cylinder reduction) Every $G$-map
  $f\colon X \to S^2$ is completely determined by its restriction
  $\restr{f}{Z}$ to the cylinder $Z$. Discussing $G$-homotopies of
  maps $X \to S^2$ is equivalent to discussing homotopies of maps
  $(Z, C) \to (S^2,E)$.
\end{remark}

\label{CjS1Identification}
For the following we need to fix an identification of $C_0$ and $C_1$
with $S^1$.\footnote{Here, $S^1$ is regarded as being equipped
  with the standard orientation.} Let us assume that the
identification of the circles $C_j$ with $S^1$ is such that the loops
\begin{align*}
  \gamma_j\colon I &\to X\\
  t &\mapsto \left[t + \frac{j}{2}i\right] \;\text{ for $j=0,1$}
\end{align*}
both define loops of degree one. Recall that loops in $S^1$ are up to
homotopy determined by their degree. Based on this observation we
define a normal form for equivariant maps $X \to S^2$:
\begin{definition}
  \label{DefinitionBoundaryNormalization}
  (Type I normalization) Let $f\colon X \to S^2$ be a $G$-map (resp. a
  map $(Z,C) \to (S^2,E)$). Then $f$ is \emph{type I normalized} if
  the restrictions
  \begin{align*}
   \Restr{f}{C_j}\colon C_j \to E
  \end{align*}
  are, using the fixed identifications of the circles $C_j$ and $E$
  with $S^1$, of the form $z \mapsto z^{k_j}$ for some integers $k_j$.
\end{definition}

\begin{remark}
  \label{BoundaryNormalization}
  Let $f\colon (Z,C) \to (S^2,E)$ be a map (or let $f$ be a $G$-map $X
  \to S^2$). Via the homotopy extension property (resp. its
  equivariant version, see corollary~\ref{G-HEP}) of the pair $(Z, C)$
  (resp. $(X,C)$), there exists a $G$-homotopy from $f$ to a map $f'$,
  which is type I normalized. This means that, using the respective
  identifications of the $C_j$ and $E$ with $S^1$, the map
  $\restr{f'}{C_j}$ is of the form $z \mapsto z^{k_j}$ where $k_j$
  denotes the respective fixed point degree on the circle $C_j$. \qed
\end{remark}

Given a $G$-map $f\colon X \to S^2$ we have already defined its fixed point
degree $\Bideg{d_0}{d_1}$ (see
definition~\ref{DefinitionFixpointDegree}). Given a map $f\colon (Z,
C) \to (S^2,E)$ we define its fixed point degree to be the fixed point degree
of its equivariant extension $X \to S^2$.

\paragraph{The Free Loop Space}

By the cylinder reduction (remark~\ref{CylinderReduction}) it suffices
to classify maps $(Z,C) \to (S^2,E)$ up to homotopy. The cylinder $Z$
is just $S^1$ times some closed interval (e.\,g. the closed unit
interval $I$). Recall that $\mathcal{L}S^2$, the free loop space of
the space $S^2$, is the space of all maps $S^1 \to S^2$:
\begin{align*}
  \mathcal{L}S^2 = \mathcal{M}(S^1,S^2) = \{f\colon S^1 \to S^2\}.
\end{align*}
First we make a general remark about the topology on the free loop
space:
\begin{remark}
  The space $S^1$ is Hausdorff and locally compact. Hence, in the
  terminology of e.\,g. \cite{ArticleEscardoHeckmann}, it is
  \emph{exponentiable}, which means that for any other spaces $Y$ and
  $A$, the natural bijection
  \begin{align*}
    \mathcal{M}(A,\mathcal{M}(S^1,Y)) = \left(\Maps{S^1}{Y}\right)^A
    \cong \Maps{A \times S^1}{Y} = \mathcal{M}(A \times S^1,Y)
  \end{align*}
  is compatible with the topology in the sense that it preserves
  continuity. See \cite{ArticleEscardoHeckmann}. The mapping spaces
  are regarded as being equipped with the compact-open topology. \qed
\end{remark}

As an immediate consequence of the above we note:
\begin{remark}
  A map $f\colon (Z,C) \to (S^2,E)$ can be regarded as a curve in
  $\mathcal{L}S^2$ starting and ending at loops whose image are
  contained in $E$. After type I normalization the start and end curve
  are distinguished and indexed by $\mathbb{Z}$. \qed
\end{remark}

We do not make any distinction in the notation; we silently identify
maps $(Z,C) \to (S^2,E)$ with curves in $\mathcal{L}S^2$.

\begin{definition}
  Let $\alpha$ be a curve in $\mathcal{L}S^2$ from $p_1$ to $p_2$ and
  $\beta$ a curve from $p_2$ to $p_3$. Then we denote by
  $\beta \ast \alpha$ the concatenation of the curves $\alpha$ and
  $\beta$, which is a curve from $p_1$ to $p_3$. Given a curve
  $\alpha$, we denote by $\alpha^{-1}$ the curve with time reversed.
\end{definition}

\begin{figure}[h]
  \centering
  \scriptsize
  \subfloat[The map $\alpha\colon Z \to S^2$.]{%
    \centering
    \def\svgwidth{45mm}%
    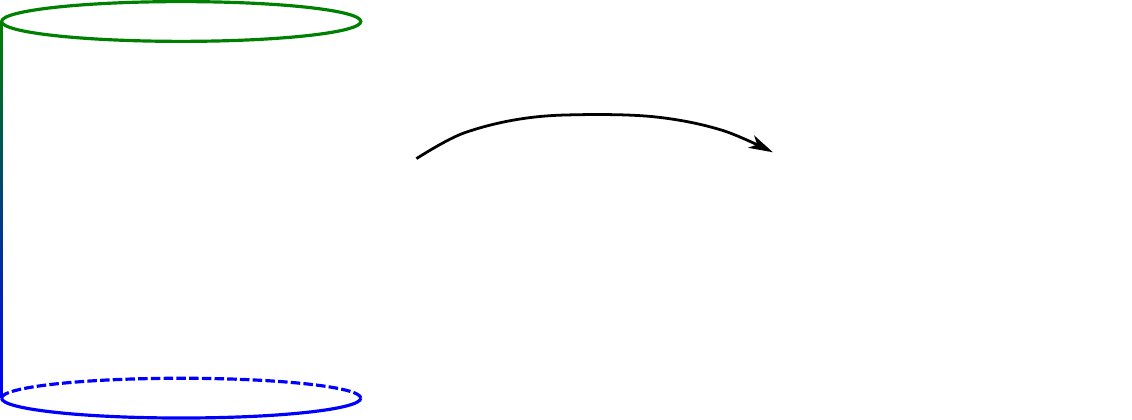}\quad
  \subfloat[The map $\beta\colon Z \to S^2$.]{%
    \centering
    \def\svgwidth{45mm}%
    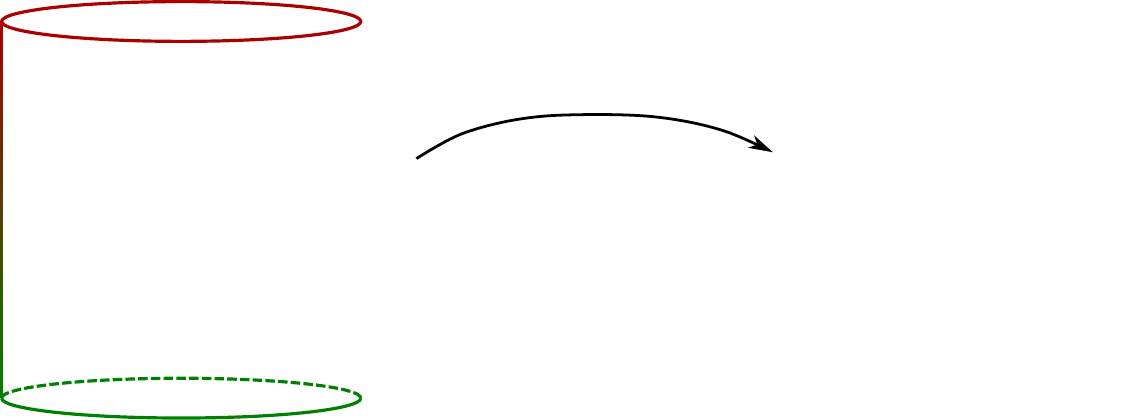}\quad
  \subfloat[The map $\beta \ast \alpha\colon Z \to S^2$.]{%
    \centering
    \def\svgwidth{45mm}%
    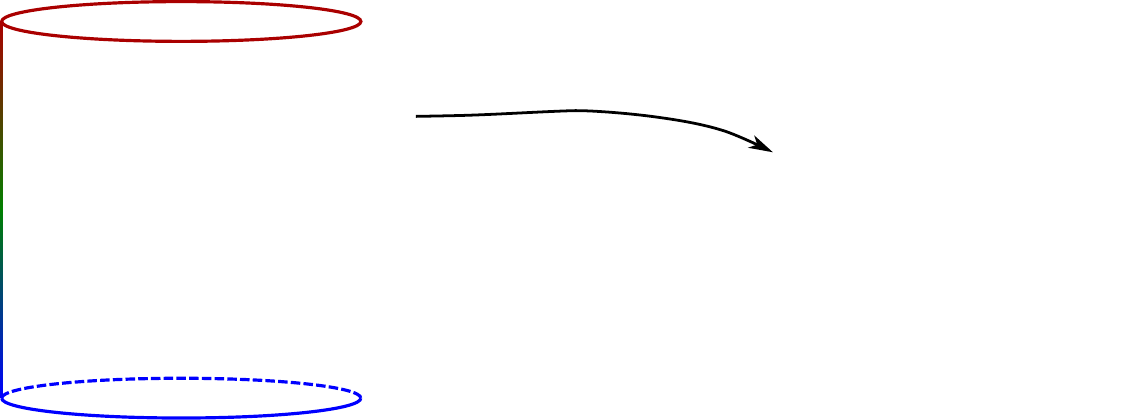}
  \caption{Concatenation of maps from the cylinder.}
  \label{fig:ConcatCylinderMaps}
\end{figure}

Using the isomorphism 
\begin{align}
  \label{FLSIso}
  \mathcal{M}(I,\mathcal{L}S^2) \cong \mathcal{M}(Z,S^2),
\end{align}
concatenation of curves induces a corresponding operation on the space
of (type I normalized) maps $(Z,C) \to (S^2,E)$. This is depicted in
figure~\ref{fig:ConcatCylinderMaps}. We need to be careful to
distinguish the inverse of a map from its inverse regarded as a
curve. But this should always be clear from the context.

Of course, the homotopy classification of curves in $\mathcal{L}S^2$
depends on the topology of $\mathcal{L}S^2$. Our first remark towards
understanding the topology of the latter is:
\begin{remark}
  The free loop space $\mathcal{L}S^2$ is path-connected.
\end{remark}
\begin{proof}
  Let $\alpha, \beta \in \mathcal{L}S^2$ be two loops in $S^2$. Since
  $\pi_1(S^2) = 0$, any loop can be deformed to the constant loop at
  some chosen point in $S^2$. By transitivity, this implies that
  $\alpha$ and $\beta$ can be deformed into each other. In other
  words, there is a path from $\alpha$ to $\beta$, regarded as
  elements of the free loop space.
\end{proof}

Now we consider the set $\pi(\mathcal{L}S^2;\gamma_0,\gamma_1)$ of
homotopy classes of curves in $\mathcal{L}S^2$ with fixed endpoints
$\gamma_0$ and $\gamma_1$. Using lemma~\ref{LemmaCurvesNullHomotopic}
we can conclude that two curves $f, g\colon I \to \mathcal{L}S^2$ from
$\gamma_0$ to $\gamma_1$ are homotopic if and only if $g^{-1} \ast f$
is null-homotopic in $\pi(\mathcal{L}S^2,\gamma_0)$. This underlines
that we need to understand the fundamental group
$\pi_1(\mathcal{L}S^2)$. In the next step we compute the fundamental
group $\pi_1(\mathcal{L}S^2)$. For this we make use of the following
fibration (see e.\,g. \cite{AgaudeArticle})
\begin{align}
  \label{LoopFibration}
  \Omega S^2 \hookrightarrow \mathcal{L}S^2 \to S^2
\end{align}
where $\Omega S^2$ denotes the space of \emph{based} loops in $S^2$ at
some fixed basepoint, e.\,g. $p_0 \in E$. The map $\mathcal{L}S^2 \to
S^2$ is the evaluation map, which takes a loop in $\mathcal{L}S^2$,
i.\,e. a map $S^1 \to S^2$, and evaluates it at $1$. Now we can state:
\begin{lemma}
  \label{FundamentalGroupOfFLS} 
  The fundamental group $\pi_1(\mathcal{L}S^2)$ of the free loop space
  of $S^2$ is infinite cyclic.
\end{lemma}
\begin{proof}
  We consider the fibration (\ref{LoopFibration}) and develop it into
  a long exact homotopy sequence. The relevant part of this sequence
  is:
  \begin{align}
    \ldots \rightarrow \mathbb{Z} \cong \pi_2(S^2) \rightarrow \pi_1(\Omega S^2) \rightarrow \pi_1(\mathcal{L}S^2) \rightarrow \pi_1(S^2) \cong 0
  \end{align}
  It is known that
  \begin{align*}
    \pi_1(\Omega S^2) \cong \pi_2(S^2) \cong \mathbb{Z}.
  \end{align*}
  Therefore, by exactness of the above sequence, it follows that
  $\pi_1(\mathcal{L}S^2)$ must be a quotient of $\mathbb{Z}$,
  i.\,e. either the trivial group, $\mathbb{Z}$ or one of the finite
  groups $\mathbb{Z}_n$. We prove that it is $\mathbb{Z}$ with the
  help of a geometric argument: Let $\gamma_1$ be the curve in $S^2$
  which parametrizes the equator $E$ with degree one and choose
  $\gamma_1$ as the basepoint for the fundamental group and then
  construct an infinite family of maps $X \to S^2$ with the same fixed
  point degree $(1,1)$ but with different global degrees -- the latter
  implies by Hopf's theorem that they cannot be homotopic. In
  particular they cannot be $G$-homotopic and hence the curves
  associated to their restrictions to $Z$ cannot be homotopic.

  The construction goes as follows: We construct a map $\varphi$ from
  the cylinder $Z = I \times S^1$ to $S^2$ such that
  $\varphi_0 \equiv \varphi_1$ and $\varphi_0$ and $\varphi_1$ are
  degree 1 maps $S^1 \to E$ as follows: Assume that for $t=0$ we have
  an embedding of $S^1$ as the equator $E \subset S^2$. Then we fix
  two antipodal points on the equator and begin rotating the equator
  on the Riemann sphere while keeping the two antipodal points fixed
  for all $t$. We can rotate this $S^1$-copy on the sphere as fast as
  we want, we just need to make sure that at $t=1$ we again match up
  with the equator. By making the correct number (i.\,e. an even
  number) of rotations we obtain a map $\varphi$ from $Z$ to $S^2$
  whose equivariant extension to the full torus $X$ has the desired
  fixed point degree $(1,1)$. By increasing the number of rotations,
  we can generate homotopically distinct $G$-maps from the torus into
  the sphere. Thus we obtain an infinite family of maps with fixed
  point degree $\Bideg{1}{1}$ and with distinct total
  degrees. Therefore, the fundamental group cannot be finite, but must
  be infinite cyclic. \qedhere
\end{proof}

\paragraph{The Degree Map}

In the following we introduce the \emph{total degree} into the
discussion. The fixed point degree together with this total degree will
turn out to be a complete invariant of equivariant homotopy. The
degree map will be the fundamental tool for the proof. We would like
to associate to a curve in $\mathcal{L}S^2$ the degree of its
equivariant extension. This only makes sense for a certain subset of
all curves. Thus we define:
\begin{definition}
  For each $j \in \mathbb{Z}$ we let $\mathscr{D}_j \in
  \mathcal{L}S^2$ be the normalized loop in
  $E \subset S^2$ of degree $j$. We set
  \begin{align*}
    \mathcal{L}_{\text{std}} = \left\{\mathscr{D}_j\colon j \in \mathbb{Z}\right\} \subset \mathcal{L}S^2.
  \end{align*}
  Let $\mathscr{P}$ be the set of curves in $\mathcal{L}S^2$ starting
  and ending at points in $\mathcal{L}_{\text{std}}$:
  \begin{align*}
    \mathscr{P} = \left\{\alpha\colon I \to \mathcal{L}S^2\mid \alpha(0), \alpha(1) \in \mathcal{L}_{\text{std}}\right\}.
  \end{align*}
\end{definition}
The set $\mathscr{P}$ is a suitable domain of definition for the
degree map, which we now define:
\begin{definition}
  \label{DefDegreeMap}
  The \emph{degree map} on $\mathscr{P}$ is the map
  \begin{align*}
    \text{deg}\colon \mathscr{P} &\to \mathbb{Z}\\
    \alpha &\mapsto \deg \widehat{\alpha},
  \end{align*}
  where $\smash{\widehat{\alpha}}$ denotes the equivariant extension
  of $\alpha$. That is, $\alpha$, which is a curve in the free loop
  space $\mathcal{L}S^2$, is to be regarded as a map
  $(Z,C) \to (S^2,E)$ and $\smash{\widehat{\alpha}}\colon X \to S^2$
  is its equivariant extension to the full torus (see
  lemma~\ref{Class1EquivariantExtension}).
\end{definition}
An immediate property of the degree map is:
\begin{remark}
  \label{DegOfConstantCurveZero}
  Let $\alpha \in \mathscr{P}$ be the constant curve at some $\gamma
  \in \mathcal{L}_{\text{std}}$. Then $\deg \alpha = 0$.
\end{remark}
\begin{proof}
  The equivariant extension $\smash{\widehat{\alpha}}$ has its image contained
  in $E \subset S^2$. Hence it is not surjective as a map to $S^2$,
  which implies $\deg \smash{\widehat{\alpha}} = 0$.
\end{proof}
For curves $\alpha$, $\beta$ we introduce the symbols $\alpha \simeq
\beta$ to express that $\alpha$ and $\beta$ are homotopic as curves
(with fixed endpoints).
\begin{remark}
  \label{DegreeMapInvariantUnderHomotopies}
  When $\alpha, \beta \in \mathscr{P}$ are two curves with common
  endpoints, then $\alpha \simeq \beta$ implies $\deg \alpha = \deg
  \beta$. In other words: the degree map is also well-defined
  on $\mathscr{P}/{\simeq}$, that is $\mathscr{P}$ modulo homotopy.
\end{remark}
\begin{proof}
  Homotopies between two curves can be regarded as homotopies between
  maps $Z \to S^2$ rel $C$. Such homotopies can also be equivariantly
  extended, giving homotopies between the respective equivariant
  extensions of the two maps $X \to S^2$. Therefore, their total
  degree must agree.
\end{proof}
We have just seen that the degree map factors through
$\mathscr{P}/{\simeq}$.  The next goal is to prove an injectivity
property of the degree map. Crucial for its proof is that the degree
map is compatible with the aforementioned concatenation operation of
loops:
\begin{lemma}
  \label{DegreeCompatibleWithConcatenation}
  Let $\gamma_1, \gamma_2, \gamma_3$ be in $\mathcal{L}_{\text{std}}$,
  furthermore let $\alpha \in \mathscr{P}$ be a path in
  $\mathcal{L}S^2$ from $\gamma_1$ to $\gamma_2$ and $\beta$ be a path
  in $\mathcal{L}S^2$ from $\gamma_2$ to $\gamma_3$. Then
  $\beta \ast \alpha$ is a path from $\gamma_1$ to $\gamma_3$ and
  \begin{align*}
    \deg (\beta \ast \alpha) = \deg \beta + \deg \alpha.
  \end{align*}
\end{lemma}

\begin{proof}
  Let $\alpha$ and $\beta$ be paths in $\mathcal{L}S^2$ with the
  properties mentioned above.  To prove the lemma, we have to show
  that
  \begin{align*}
    \deg \widehat{\beta \ast \alpha} = \deg \widehat{\beta} + \deg
    \widehat{\alpha}.
  \end{align*}
  The idea we employ here is that -- after smooth approximation of the
  maps -- we can count preimage points of a regular fiber to compute
  the total degrees. We regard the curves $\alpha$ and $\beta$ as maps
  from the cylinder $Z = I \times S^1$ to $S^2$. We can assume that
  $\alpha$ and $\beta$ are smooth on $Z$; the concatenation
  $\beta \ast \alpha$ does not have to be smooth at the glueing
  curves. But after a smooth approximation of $\beta \ast \alpha$ near
  this glueing curve, we can assume that the equivariant extensions
  $\smash{\widehat{\alpha}}$, $\smash{\widehat{\beta}}$ and
  $\smash{\widehat{\beta \ast \alpha}}$ are smooth maps from the torus
  into $S^2$.
  Now let $q \in S^2$ be a regular point of
  $\smash{\widehat{\beta \ast \alpha}}$, away from the neighborhood of
  the boundary circles where we have smoothed the map
  $\beta \ast \alpha$. We now consider its fiber over $q$. This
  consists of all points $p_1,\ldots,p_k$ with
  $\smash{\widehat{\alpha}}(p_j) = q$ and all points
  $p'_1,\ldots,p'_\ell$ with $\smash{\widehat{\beta}}(p'_j) =
  q$.
  Summing the points $p_j$ (resp. $p'_j$) respecting the orientation
  signs yields $\deg\smash{\widehat{\alpha}}$
  (resp. $\deg\smash{\widehat{\beta}}$). Therefore, the points
  $p_1,\ldots,p_k,p'_1,\ldots,p'_\ell$ with their respective
  orientation signs attached add up to
  $\deg\smash{\widehat{\alpha}} + \deg\smash{\widehat{\beta}}$.
\end{proof}

Note that lemma~\ref{DegreeCompatibleWithConcatenation} also implies
that
\begin{align*}
  \deg\left(\alpha^{-1}\right) = -\deg\left(\alpha\right),
\end{align*}
where $\alpha^{-1}$ is the curve $\alpha\colon I \to \mathcal{L}S^2$
with reversed timed. The space $\mathscr{P}$ does not carry a group
structure in the usual sense, but of course the based fundamental
groups $\pi_1(\mathcal{L}S^2,\gamma)$ do carry a group structure. When
$\gamma$ is in $\mathcal{L}_{\text{std}}$, then $\pi_1(\mathcal{L}S^2,\gamma)
\subset \mathscr{P}/{\simeq}$ and we can note the following
\begin{remark}
  The degree map induces a group homomorphism on fundamental groups:
  \begin{align*}
    \text{deg}\colon \pi_1(\mathcal{L}S^2, \mathscr{D}_j) \to \mathbb{Z}
  \end{align*}
  for every $j \in \mathbb{Z}$.
\end{remark}
\begin{proof}
  Let $\mathscr{D}_j$ be in $\mathcal{L}_{\text{std}}$. By
  remark~\ref{DegreeMapInvariantUnderHomotopies} we know that the
  degree map is well-defined as a map on the fundamental group
  $\pi_1(\mathcal{L}S^2,\mathscr{D}_j)$. By
  remark~\ref{DegOfConstantCurveZero} the neutral element in
  $\pi_1(\mathcal{L}S^2,\mathscr{D}_j)$, that is the homotopy class of
  the constant curve at $\mathscr{D}_j$, gets mapped to $0 \in
  \mathbb{Z}$. The group structure in the fundamental group is given
  be the concatenation of loops at the
  basepoint. Lemma~\ref{DegreeCompatibleWithConcatenation} shows that
  the degree map is compatible with the group structures. Hence, the
  degree map is a group homomorphism on fundamental groups.
\end{proof}

Now that we know that restrictions of the degree map to fundamental
groups are group homomorphisms it is useful to compute its kernel. For
this we need a lemma, which deals with the special case of a fixed point
degree of the form $\Bideg{d_0}{d_0}$. To formulate the statement, we
introduce an equivalence relation $\sim$ on the cylinder $Z$, which
identifies the circles $C_0, C_1$: $(0,z) \sim (1,z)$.  The resulting
space $\smash{\widetilde{Z}} = Z/{\sim}$ is a
2-torus.\label{CylinderBoundaryCirclesIdentified} Now we can state:

\begin{lemma}
  \label{MapExtensionDegreeEquality}
  A map $f\colon (Z,C) \to (S^2,E)$ with $\restr{f}{C_0} \equiv
  \restr{f}{C_1}$ has even degree. More concretely: The map $f$
  induces a map $\smash{\tilde{f}}\colon \smash{\widetilde{Z}} \to S^2$ and
  \begin{align*}
    \deg \hat{f} = 2 \deg \tilde{f}.
  \end{align*}
  In particular, the equivariant extension $\smash{\hat{f}}$ has degree
  zero iff the induced map $\smash{\tilde{f}}$ has degree zero.
\end{lemma}

\begin{proof}
  Note that on the LHS we have the degree of the equivariant extension
  $\smash{\hat{f}}\colon X \to S^2$ of $f$ while on the RHS we have (two
  times) the degree of the induced map $\smash{\tilde{f}}\colon
  \smash{\widetilde{Z}} \to S^2$.

  We prove the statement by counting preimage points. The torus
  $\smash{\widetilde{Z}}$ carries a differentiable structure.  The
  induced map $\smash{\tilde{f}}\colon \smash{\widetilde{Z}} \to S^2$
  does not have to be smooth near $C \subset
  \smash{\widetilde{Z}}$.
  But after a smooth approximation near $C$ we can assume that it is
  globally smooth.  By the same reason the equivariant extension
  $\smash{\hat{f}}$ defined on $X$ is globally smooth.
  Let $q$ be a regular value of the smoothed map $f$, away from the
  circle $C$. This implies that $q$ is also a regular value of
  $\smash{\tilde{f}}$ and $\smash{\hat{f}}$.  Let $d$ be the degree of
  $\smash{\tilde{f}}$.  The fiber of $q$ under $\smash{\tilde{f}}$ can thus be
  written as
  \begin{align*}
    \tilde{f}^{-1}(q) = \{p_1,\ldots,p_k\}.
  \end{align*}
  Denoting the respective orientation signs at each preimage point
  $p_j$ with $\sigma_j$ we can express the degree $d$ as
  \begin{align*}
    d = \sum_{1\leq j \leq k} \sigma_j.
  \end{align*}
  To simplify the proof we assume that $q$ has been chosen such that
  $T(q)$ is also a regular value of $\smash{\hat{f}}$ and that $q$
  (and therefore also $T(q)$) has no preimage points on the boundary
  circles.

  Recall that when regarding the cylinder $Z$ as being embedded in $X$
  with boundary circles $C_0$ and $C_1$, the equivariant extension
  $\smash{\hat{f}}$ is the map
  \begin{align*}
    \hat{f}\colon X &\to S^2\\
    x &\mapsto
    \begin{cases}
      f(x) & \;\text{if $x \in Z$}\\
      T \circ f \circ T(x) & \;\text{if $x \in Z'$}\\
    \end{cases}
  \end{align*}
  where $Z'$ denotes the complementary cylinder in $X$ whose
  intersection with $Z$ consists of the circles $C_0 \cup C_1$. It
  remains to compare the fibers $\smash{\hat{f}^{-1}}(q)$ and
  $\smash{\tilde{f}^{-1}}(q)$. Since $q$ is assumed to have no preimage
  points in the boundary circles, we can write the first fiber as
  \begin{align*}
    \hat{f}^{-1}(q) = \Restr{\hat{f}}{Z}^{-1}(q) \cup \Restr{\hat{f}}{Z'}^{-1}(q).
  \end{align*}
  By definition of $\smash{\hat{f}}$, the preimage points in the cylinder
  $Z$ are exactly those of $\smash{\tilde{f}}$, with the same orientation
  signs. We now prove
  \begin{align*}
    \Restr{\hat{f}}{Z'}^{-1}(q) = T\left(\Restr{\hat{f}}{Z}^{-1}(T(q))\right).
  \end{align*}
  For this, let $p$ be in $Z'$ with
  $\restr{\smash{\hat{f}}}{Z'}(p) = q$.  This implies
  $T \circ \restr{\smash{\hat{f}}}{Z'}(p) = T(q)$, which in turn
  implies by equivariance
  \begin{align*}
   \Restr{\hat{f}}{Z}(T(p)) = T(q).
  \end{align*}
  In other words, after defining $\smash{\tilde{p}} = T(p) \in Z$,
  we have: $p = T(\smash{\tilde{p}})$ with $\smash{\tilde{p}} \in Z$
  and
  \begin{align*}
    \Restr{\smash{\tilde{f}}}{Z}(\smash{\tilde{p}}) =
    \Restr{\smash{\tilde{f}}}{Z}(T(p)) =
    T \circ \Restr{\smash{\tilde{f}}}{Z'}(p) = T(q).
  \end{align*}
  On the other hand, let $\smash{\tilde{p}}$ be in
  $T\left(\restr{\smash{\hat{f}}}{Z}^{-1}(T(q))\right)$. This
  means that $\smash{\tilde{p}} = p$ with
  $\restr{\smash{\hat{f}}}{Z}(p) = T(q)$. Set
  $p = T(\smash{\tilde{p}})$. Then:
  \begin{align*}
    \Restr{\hat{f}}{Z'}(\hat{p}) = \Restr{\hat{f}}{Z'}(T(p)) = T \circ \Restr{\hat{f}}{Z}(p) =  q,
  \end{align*}
  which means $\smash{\tilde{p}} \in \restr{\smash{\hat{f}}}{Z'}^{-1}(q)$.

  Since $T$ bijectively sends $Z$ to $Z'$ it follows that the number
  of preimage points on $Z$ and $Z'$ coincide. Also, since
  $\smash{\hat{f}}$ restricted to (the interior of) $Z'$ is the
  composition of $f$ with two orientation-reversing diffeomorphisms,
  the orientation sign at each preimage point $p$ in $Z$ is the same
  as the orientation sign of the corresponding preimage point $T(p)$
  in $Z'$. This means that summing (respecting the orientation) over
  all the preimage points in the fiber $\smash{\hat{f}}^{-1}(q)$
  yields exactly $2 \deg \smash{\tilde{f}}$.
\end{proof}

Now we can prove the desired injectivity property of the degree map:
\begin{proposition}
  \label{DegreeMapInjective}
  For any $\gamma \in \mathcal{L}_{\text{std}}$, the degree homomorphism
  \begin{align*}
    \deg\colon \pi_1(\mathcal{L}S^2,\gamma) \to \mathbb{Z}
  \end{align*}
  is injective.
\end{proposition}
\begin{proof}
  We prove that the degree homomorphism has trivial kernel. By
  path-connected- ness of $\mathcal{L}S^2$ all the based fundamental
  groups of $\mathcal{L}S^2$ are isomorphic. Hence it suffices to
  prove the statement for one fixed loop $\gamma$. Therefore we reduce
  the problem to a simpler case and take $\gamma$ to be the constant
  loop at the base point $p_0$ in $E$.

  Let $\alpha$ be a type I normalized map of the cylinder $Z$ into
  $S^2$ with fixed point degree $\Bideg{0}{0}$ and $\deg \alpha = 0$. As before,
  $\alpha$ can be regarded as curve in $\mathcal{L}S^2$ and $[\alpha]$
  defines an element in the fundamental group
  $\pi_1(\mathcal{L}S^2,\gamma)$. Let $\widetilde{Z} = Z/{\sim}$ be the
  torus which results from the cylinder $Z$ by identifying the circles
  $C_0$ and $C_1$ (see
  p.\pageref{CylinderBoundaryCirclesIdentified}). Since, by
  construction, $\restr{\alpha}{C_0} \equiv \restr{\alpha}{C_1} \equiv
  \gamma$, $\alpha$ induces a map $\widetilde{Z} \to S^2$ which, by
  lemma~\ref{MapExtensionDegreeEquality}, also has degree zero.

  Denote the generator of the fundamental group of the torus
  $\widetilde{Z}$ which corresponds to $C_0$ resp. $C_1$ by $C$.
  Let $C'$ denote the other generator of the fundamental group such
  that the intersection number of $C$ with $C'$ is $+1$. The
  restriction of $f\colon \widetilde{Z} \to S^2$ to $C$ is already
  constant. Using the homotopy extension property together with the
  simply-connectedness of $S^2$ we can assume that $f$ is also
  constant along $C'$. Therefore we can collapse these two generators
  and $f$ induces a map
  $\widetilde{Z}/(C \cup C') \cong S^2 \to S^2 \cong S^2$. By
  remark~\ref{RemarkMapOnQuotientDegreeUnchanged} the map
  $f\colon \widetilde{Z}/(C \cup C') \to S^2$ also has degree zero. By
  Hopf's theorem, this map is null-homotopic (even as a map of pointed
  spaces, as can be shown by another application of the homotopy
  extension property). This corresponds to a null-homotopy of the
  curve $\alpha$ with fixed basis curve $\gamma$, hence
  $[\alpha] = 0$, which finishes the proof.
\end{proof}

\paragraph{The Degree Triple}

In the following we combine the total degree and the fixed point
degrees. For this we make the following
\begin{definition}
  \label{DefinitionTriple}
  (Degree triple map) The \emph{degree triple map} is the map
  \begin{align*}
    \mathcal{T}\colon \mathcal{M}_G(X,S^2) &\to \Z^3\\
    f &\mapsto \Triple{d_0}{d}{d_1},
  \end{align*}
  where $\Bideg{d_0}{d_1}$ is the fixed point degree of $f$ and $d$ is
  the total degree of $f$. For a map $f\colon X \to S^2$ we define its
  \emph{degree triple} (or simply \emph{triple}) to be
  $\mathcal{T}(f)$. We call a given triple \emph{realizable} if it is
  contained in the image $\Im(\mathcal{T})$.  For a map $f\colon (Z,C)
  \to (S^2,E)$ we define its \emph{degree triple} to be the degree
  triple of the equivariant extension $\smash{\hat{f}}\colon X \to S^2$.
\end{definition}

As an immediate consequence we obtain:
\begin{remark}
  \label{DegreeTripleInvariant}
  The degree triple of a map is a $G$-homotopy invariant. \qed
\end{remark}
Thus, the degree triple map factors through $[X,S^2]_G$. In the
following we analyze the image $\Im(\mathcal{T})$.

\begin{remark}
  \label{RemarkTripleGlobalDegree}
  Let $f$ be a $G$-map $X \to S^2$ with fixed point degrees
  $\Bideg{d_0}{d_0}$. Then the total degree $\deg f$ must be even. In
  other words, the degree triples $\Triple{d_0}{2k+1}{d_0}$ are not
  contained in the image $\Im(\mathcal{T})$.\end{remark}
\begin{proof}
  Assume that the map $f$ is type I normalized. The restriction
  $\restr{f}{Z}$ to the cylinder then satisfies the assumptions of
  lemma~\ref{MapExtensionDegreeEquality}, hence we can conclude that
  its equivariant extension $\smash{\hat{f}}$ has even degree. But the
  equivariant extension of $\restr{f}{Z}$ is $f$ itself, hence $f$
  must have even degree.
\end{proof}
In particular, the triple $\Triple{0}{1}{0}$ is not contained in
$\Im(\mathcal{T})$. This gives rise to the $\mod 2$ condition
\begin{align}
  \label{DiscussionParityCondition}
  d \equiv d_0 + d_1 \mod 2,
\end{align}
which was already mentioned in the outline and must hold for any
realizable degree triple. The details will be explained in
paragraph~\ref{ParagraphClassificationType1MainResult}
(p.~\pageref{ParagraphClassificationType1MainResult}). In order to
deduce the above ``parity condition''
(\ref{DiscussionParityCondition}) for general degree triples we
introduce an algebraic structure on the set of realizable
triples. This is the following binary operation for triples:
\begin{definition}
  Two triples $\Triple{d_0}{d}{d_1}$ and $\Triple{d_0'}{d'}{d_1'}$ are
  called \emph{compatible} if $d_1 = d_0'$.  Given two compatible
  triples $\Triple{d_0}{d}{d_1}$ and $\Triple{d_1}{d'}{d_2}$, then we
  define $\Triple{d_0}{d}{d_1} \bullet \Triple{d_1}{d'}{d_2}$ to be
  the triple
  \begin{align*}
    \Triple{d_0}{d + d'}{d_2}.
  \end{align*}
\end{definition}
We call this binary operation \emph{concatenation} of triples. Its
usefulness is illustrated by the following remark:
\begin{remark}
  \label{ImageOfTripleMapClosedUnderConcatenation}
  When the triples $\Triple{d_0}{d}{d_1}$ and $\Triple{d_1}{d'}{d_2}$
  are in $\Im(\mathcal{T})$, then so is
  \begin{align*}
    \Triple{d_0}{d}{d_1} \bullet \Triple{d_1}{d'}{d_2}.
  \end{align*}
\end{remark}
\begin{proof}
  Assume that the triples $\Triple{d_0}{d}{d_1}$
  resp. $\Triple{d_1}{d'}{d_2}$ are realized by the paths $\alpha$
  resp. $\beta$ in $\mathscr{P}$ satisfying $\alpha(1) = \beta(0)$. By
  definition,
  \begin{align*}
    \Triple{d_0}{d}{d_1} \bullet \Triple{d_1}{d'}{d_2} = \Triple{d_0}{d+d'}{d_2}\;.
  \end{align*}
  We need to show the existence of a path in the loop space
  $\mathcal{L}S^2$ from $\alpha(0)$ to $\beta(1)$ with total degree
  $d + d'$ of its associated equivariant extension. Recall that $d$
  (resp. $d'$) is the degree of the equivariant extension
  $\widehat{\alpha}$ (resp. $\widehat{\beta}$). The concatenation
  $\alpha \ast \beta$ is a path in loop space from $d_0$ to $d_2$.
  Its total degree is defined to be the degree of the equivariant
  extension $\widehat{\alpha \ast \beta}$. By
  lemma~\ref{DegreeCompatibleWithConcatenation} this is the same as
  $\deg\widehat{\alpha} + \deg\widehat{\beta} = d + d'$.
\end{proof}

In other words: the image $\Im(\mathcal{T})$ is closed under the
concatenation operation of degree triples. In the following lemma we
encapsulate several fundamental properties of the image
$\Im(\mathcal{T})$.
\begin{proposition}
  \label{TripleBuildingBlocks}
  Let $d_0, d_1$ be integers. Then:
  \begin{enumerate}[(i)]
  \item If the triple $\Triple{d_0}{d}{d_1}$ is in $\Im(\mathcal{T})$, then so is
    $\Triple{d_0}{-d}{d_1}$.
  \item For any fixed point degree $\Bideg{d_0}{d_1}$, the triple
    $\Triple{d_0}{d_0 - d_1}{d_1}$ is in $\Im(\mathcal{T})$. In
    particular, the triple $\Triple{1}{1}{0}$ is in
    $\Im(\mathcal{T})$.
  \item For any $k$, the triple $\Triple{0}{2k}{0}$ is in $\Im(\mathcal{T})$.
  \item If $\Triple{d_0}{d}{d_1}$ is in $\Im(\mathcal{T})$, then so is
    $\Triple{d_1}{d}{d_0}$.
  \end{enumerate}
\end{proposition}
\begin{proof}
  Regarding (i): Assume that $f\colon X \to S^2$ is a map with
  $\mathcal{T}(f) = \Triple{d_0}{d}{d_1}$. Note that $T\colon S^2 \to
  S^2$ is an orientation-reversing diffeomorphism of $S^2$ which keeps
  the equator $E$ fixed. Therefore:
  \begin{align*}
    \mathcal{T}(T \circ f) = \Triple{d_0}{-d}{d_1}.
  \end{align*}
  In other words: the triple $\Triple{d_0}{-d}{d_1}$ is realizable.
 
  Regarding (ii): We need to construct a map $X \to S^2$ with fixed
  point degrees $\Bideg{d_0}{d_1}$ and total degree $d_0 - d_1$. For
  this, let $D$ be the closed unit disk in the complex plane:
  \begin{align*}
    D = \{z \in \mathbb{C}\colon |z| \leq 1\}.
  \end{align*}
  Let $\iota_D$ be the embedding of $D$ onto one of the two hemispheres.:
  \begin{align}
    \label{IotaHemisphereEmbedding}
    \iota\colon D &\hookrightarrow S^2\\
    r e^{i\varphi} &\mapsto
    \begin{pmatrix}
      r\cos \varphi \\
      r\sin \varphi \\
      \pm\sqrt{1-r^2}
    \end{pmatrix}.
  \end{align}
  We pick the sign defining either the lower or the upper hemisphere
  such that $\iota_D$ is orientation preserving.  Now we define the
  map
  \begin{align*}
    f\colon Z &\to D\\
    (t, \varphi) &\mapsto \begin{cases}
      (1-2t) e^{i d_0 \varphi}  & \text{ for $0 \leq t \leq \frac{1}{2}$}\\
      (2t-1) e^{i d_1 \varphi} & \text{ for $\frac{1}{2} \leq t \leq 1$}.
    \end{cases}
  \end{align*}
  To ease the notation, set
  \begin{align*}
    Z_1 = \left(0,\frac{1}{2}\right) \times S^1 \;\text{ and }\; Z_2 &= \left(\frac{1}{2},1\right) \times S^1.
  \end{align*}
  Then we consider the composition map
  \begin{align*}
    F = \iota_D \circ f\colon Z \to S^2.
  \end{align*}
  The restriction $\restr{F}{C_j}$ to the boundary circles is of the
  form:
  \begin{align*}
    \Restr{F}{C_j}\colon S^1 &\to E \subset S^2\\
    \varphi &\mapsto
    \begin{pmatrix}
      \cos(d_j \varphi) \\
      \sin(d_j \varphi) \\
      0
    \end{pmatrix}.
  \end{align*}
  Hence, by construction, $F$ has fixed point degrees
  $\Bideg{d_0}{d_1}$ (compare with (\ref{DegreeOneMapOnEquator}),
  p.~\pageref{DegreeOneMapOnEquator}). It remains to show that the
  equivariant extension $\smash{\hat{F}}$ has total degree
  $d_0 - d_1$. At least away from the boundary circles,
  $\smash{\hat{F}}$ is a smooth map. Thus, let $q \in S^2$ be a
  regular value of $\smash{\hat{F}}$, say, on the hemisphere
  $\iota_D(D)$, but not one of the poles (they are not regular values
  for $f$). The point $q$ corresponds to some point
  $\smash{\tilde{q}} = \smash{\tilde{t}e^{i\tilde{\varphi}}}$ in the
  closed unit disk $D$ with $\smash{\tilde{t}} \in (0,1)$. By
  construction, the fiber $\smash{\hat{F}^{-1}}(q)$ will be the same as
  the fiber $F^{-1}(\tilde{q})$, which contains exactly $d_0 + d_1$
  points in the cylinder $Z$:
  \begin{align*}
    q_{j,k} = \left(\frac{1+(-1)^{j+1}\tilde{t}}{2}, d_j\left(\tilde{\varphi} + \frac{2k}{d_j}\pi\right)\right)\; \text{for $j=0,1$ and $k=0,1,\ldots,d_j$}.
  \end{align*}
  The points $q_{0,k}$, $k=0,\ldots,d_1-1$, are of the form
  \begin{align*}
    q_{0,k} = \left(\frac{1-\tilde{t}}{2}, d_1\left(\tilde{\varphi} + \frac{2k}{d_1}\pi\right)\right)
  \end{align*}
  which means that they are contained in the cylinder half $Z_1$. The
  points $q_{1,k}$, $k=0,\ldots,d_2-1$ are of the form
  \begin{align*}
    q_{2,k} = \left(\frac{1+\tilde{t}}{2}, d_2\left(\tilde{\varphi} + \frac{2k}{d_2}\pi\right)\right)
  \end{align*}
  and therefore they are contained in $Z_2$. A computation shows
  that
  the orientations signs at the points $q_{1,k} \in Z_1$ are the
  opposite of those at the points $q_{2,k} \in Z_2$. Therefore,
  summing over the points in the generic fiber
  $\smash{\hat{F}^{-1}}(q)$ yields $\pm (d_0 - d_1)$ as total degree
  of $\smash{\hat{F}}$. Since $\iota_D$ has been chosen
  orientation-preserving, the total degree is $d_0 - d_1$.

  Regarding (iii): First note that by (i) and (ii) the triple
  $\Triple{0}{1}{1}$ is in $\Im(\mathcal{T})$. But then the triple
  $\Triple{0}{1}{1} \bullet \Triple{1}{1}{0} = \Triple{0}{2}{0}$ is
  also in $\Im(\mathcal{T})$. Concatenation of this triple with itself
  $k$ times yields the desired triple $\Triple{0}{2k}{0}$.

  Regarding (iv): Let $\tau\colon X \to X$ be the translation inside
  the torus, which swaps the boundary circles (e.\,g. $[z] \mapsto [z
  + \text{\textonehalf}i]$ in our model). Then, composing a map of type
  $\Triple{d_0}{d}{d_1}$ with $\tau$ yields a map of type
  $\Triple{d_1}{d}{d_0}$.
  This finishes the proof.
\end{proof}

\label{RemarkAboutNormalForms}
The proof of proposition~\ref{TripleBuildingBlocks} (ii) is of
particular importance -- it contains a constructive method for
producing equivariant maps $X \to S^2$ with triple
$\Triple{d_0}{d_0-d_1}{d_1}$. These triples can be regarded as basic
building blocks for equivariant maps $X \to S^2$, since all other
triples can be built from triples of this form by concatenation.

\begin{definition}
  We call the maps constructed in proposition~\ref{TripleBuildingBlocks}
  (ii) \emph{maps in normal form} for the triple
  $\Triple{d_0}{d_0-d_1}{d_1}$.
\end{definition}

Note that the normal form map for a triple
$\Triple{d_0}{d_1-d_0}{d_1}$ is by definition the normal form map for
$\Triple{d_0}{d_0-d_1}{d_1}$ composed with $T\colon S^2 \to S^2$.

\paragraph{Main Result}
\label{ParagraphClassificationType1MainResult}

In this paragraph we state and prove the main classification result
for the type I involution.  The statement is:
\begin{theorem}
  \label{Classification1}
  The $G$-homotopy class of a map $f \in \mathcal{M}_G(X,S^2)$ is
  uniquely determined by its degree triple $\mathcal{T}(f)$. The image
  $\Im(\mathcal{T})$ of the degree triple map
  $\mathcal{T}\colon \mathcal{M}_G(X,S^2) \to \Z^3$ consists of those
  triples $\Triple{d_0}{d}{d_1}$ satisfying
  \begin{align*}
    d \equiv d_0 + d_1 \mod 2.
  \end{align*}
\end{theorem}

We need one last lemma in order to prove this theorem:
\begin{lemma}
  \label{HomotopyOnCylinder}
  Let $f$ and $g$ be two type I normalized maps $(Z, C) \to (S^2,E)$
  with the same triple $\Triple{d_0}{d}{d_1}$. Then $f$ and $g$ are
  homotopic rel $C$.
\end{lemma}
\begin{proof}
  Both maps can be regarded as curves $f,g\colon I \to
  \mathcal{L}S^2$. Since they are assumed to be type I normalized,
  they start at the same curve $\gamma_0 \in \mathcal{L}_{\text{std}}$
  and end at the same curve $\gamma_1 \in
  \mathcal{L}_{\text{std}}$. We need to prove the existence of a
  homotopy between the curves $f$ and $g$ in $\mathcal{L}S^2$.

  By lemma~\ref{LemmaCurvesNullHomotopic} it suffices to show that the
  loop $g^{-1} \ast f$ based at $\gamma_0$ is null-homotopic.  To
  prove this we use the degree map restricted to the fundamental group
  based at $\gamma_0$:
  \begin{align*}
    \deg\colon \pi_1(\mathcal{L}S^2,\gamma_0) &\to \mathbb{Z}\\
    [\gamma] &\mapsto \deg \hat{\gamma}.
  \end{align*}
  From lemma~\ref{DegreeCompatibleWithConcatenation} it follows that
  \begin{align*}
    \deg (g^{-1} \ast f) = \deg f - \deg g = \deg \hat{f} - \deg \hat{g} = \deg f - \deg g = d - d = 0.
  \end{align*}
  Because of remark~\ref{DegreeMapInvariantUnderHomotopies} this can
  also be stated on the level of homotopy classes:
  \begin{align*}
    \deg [g^{-1} \ast f] = \deg [f] - \deg [g] = 0.
  \end{align*}
  Therefore, using the injectivity of the degree map
  (proposition~\ref{DegreeMapInjective}) we can conclude that $[g^{-1}
  \ast f] = 0$ in $\pi_1(\mathcal{L}S^2,\gamma_1)$. In other
  words:
  \begin{align*}
    g^{-1} \ast f \simeq c_{\gamma_0},
  \end{align*}
  where $c_{\gamma_0}$ denotes the constant curve in $\mathcal{L}S^2$ at
  $\gamma_0$. Now, lemma~\ref{LemmaCurvesNullHomotopic} implies that
  $f$ and $g$ are homotopic by means of a homotopy $h\colon I
  \times I \to \mathcal{L}S^2$. This homotopy induces a homotopy
  \begin{align*}
    H\colon I \times (Z,C) &\to (S^2,E)\\
    (t,(s,z)) &\mapsto h(t,s,z)
  \end{align*}
  between the maps $f$ and $g$.
\end{proof}

Finally, we can prove theorem~\ref{Classification1}:
\begin{proof}
  It is clear by remark~\ref{DegreeTripleInvariant} that the degree
  triple $\Triple{d_0}{d}{d_1}$ is a $G$-homotopy invariant. Now we
  show that it is a \emph{complete} invariant in the sense that
  \begin{align*}
    \mathcal{T}(f) = \mathcal{T}(g) \;\Rightarrow\; f \simeq_G g,
  \end{align*}
  where $f \simeq_G g$ means that the maps $f$ and $g$ are
  $G$-homotopic.  For this, let $f$ and $g$ be two $G$-maps with the
  same triple $\Triple{d_0}{d}{d_1}$.
  By remark~\ref{BoundaryNormalization} we can assume that $f$ and $g$
  are type~I normalized. Now we use remark~\ref{CylinderReduction} and
  reduce the construction of a $G$-homotopy from $f$ to $g$ to the
  construction of a homotopy from $\restr{f}{Z}$ to $\restr{g}{Z}$ rel
  $C$. The assumptions of lemma~\ref{HomotopyOnCylinder} are
  satisfied, hence this lemma provides us with a homotopy $H\colon I
  \times (Z,C) \to (S^2,E)$ rel $C$. Such a homotopy can be
  equivariantly extended to all of $X$, establishing a $G$-homotopy
  between $f$ and $g$ as maps $X \to S^2$.

  It remains to compute the image $\Im(\mathcal{T})$. To simplify the
  notation, set
  \begin{align*}
    p =
    \begin{cases}
      0 & \;\text{if $d_0 + d_1$ is even}\\
      1 & \;\text{if $d_0 + d_1$ is odd}.
    \end{cases}
  \end{align*}
  First we show that for any $d$ of the form $2k + p$ the triple
  $\Triple{d_0}{d}{d_1}$ is contained in the image $\Im(\mathcal{T})$,
  then we show that the assumption of the triple $\Triple{d_0}{2k + 1
    + p}{d_1}$ being contained in $\Im(\mathcal{T})$ leads to a
  contradiction. To ease the notation, let $\sigma_j$ be the sign of
  $d_j$, i.\,e. $d_j = \sigma_j|d_j|$, for $j=0,1$. By definition of
  $p$ we know that $2k + p - (d_0 + d_1)$ is an even number. Hence the
  triple
  \begin{align*}
    t = \Triple{0}{2k+p-(d_0+d_1)}{0}
  \end{align*}
  is contained in $\Im(\mathcal{T})$ by
  proposition~\ref{TripleBuildingBlocks} (ii). We define the following
  triples:
  \begin{align*}
    t_1 =& \Triple{\sigma_0}{\sigma_0}{0}\\
    t_2 =&\Triple{2\sigma_0}{\sigma_0}{\sigma_0}\\
    & \ldots\\
    t_{|d_0|} =&\Triple{|d_0|\sigma_0}{\sigma_0}{|d_0|\sigma_0 - \sigma_0}\\
    u_1 =& \Triple{0}{\sigma_1}{\sigma_1}\\
    u_2 =& \Triple{\sigma_1}{\sigma_1}{2\sigma_1}\\
    & \ldots\\
    u_{|d_1|} =& \Triple{|d_1|\sigma_1 - \sigma_1}{\sigma_1}{|d_1|\sigma_1},
  \end{align*}
  which are contained in $\Im(\mathcal{T})$ by
  proposition~\ref{TripleBuildingBlocks} (ii). From this we can form the
  triple
  \begin{align*}
    \tilde{t} &= t_{|d_0|} \bullet \ldots \bullet t_2 \bullet t_1 \bullet t \bullet u_1 \bullet u_2 \bullet \ldots \bullet u_{|d_1|}\\
    &= \Triple{d_0}{2k+p}{d_1}
  \end{align*}
  Therefore, the triple $\tilde{t}$ is contained in
  $\Im(\mathcal{T})$. On the other hand, assume that the triple $t =
  \Triple{d_0}{2k + 1 + p}{d_1}$ is contained in $\Im(\mathcal{T})$
  for some $k$. As above, we define the following realizable triples
  \begin{align*}
    t_1 =& \Triple{0}{-\sigma_1}{\sigma_1}\\
    t_2 =&\Triple{\sigma_1}{-\sigma_1}{2\sigma_1}\\
    & \ldots\\
    t_{|d_0|} =&\Triple{|d_0|\sigma_1 - \sigma_1}{-\sigma_1}{|d_0|\sigma_1}\\
    u_1 =& \Triple{\sigma_2}{-\sigma_2}{0}\\
    u_2 =& \Triple{2\sigma_2}{-\sigma_2}{\sigma_2}\\
    & \ldots\\
    u_{|d_1|} =& \Triple{|d_1|\sigma_2}{-\sigma_2}{|d_1|\sigma_2 - \sigma_2}.
  \end{align*}
  This allows us to form the following triple
  \begin{align*}
    \tilde{t} &= t_1 \bullet t_2 \bullet \ldots \bullet t_{|d_0|} \bullet t \bullet u_{|d_1|} \bullet \ldots \bullet u_2 \bullet u_1 \\
    &= \Triple{0}{2k + 1 + p - (d_0 + d_1)}{0},
  \end{align*}
  which then also has to be in $\Im(\mathcal{T})$. But, by definition,
  $p - (d_0 + d_1)$ is an even number, hence $2k + 1 + p - (d_0 +
  d_1)$ is odd, contrary to
  remark~\ref{RemarkTripleGlobalDegree}. Hence, the assumption that
  $\Triple{d_0}{2k+1+p}{d_1}$ is in the image $\Im(\mathcal{T})$
  cannot hold.
\end{proof}

A consequence of the above is that we can identify equivariant
homotopy classes with their associated degree triples. At this point
we underline that the proof of theorem~\ref{Classification1} and
lemma~\ref{HomotopyOnCylinder} shows that when two type I normalized
$G$-maps $f, g\colon X \to S^2$ have the same degree triple, then they
are not only equivariantly homotopic, but they are equivariantly
homotopic rel $C = C_0 \cup C_1$. This will be of particular
importance in the next section.

\subsubsection{Type  II}

In this section, the torus $X$ -- still defined in terms of the
standard square lattice -- is equipped with the type II involution
\begin{align*}
  T\colon X &\to X\\
  [z] &\mapsto [i\overline{z}].
\end{align*}
As before, we can regard the torus $X$ as a $G$-CW complex. The $G$-CW
structure we use for the type II involution is depicted in
figure~\ref{MSInvolutionFigureCW}.
\label{ParagraphGeometryOfClass2}
The first major difference when compared with type I is the fact that
the fixed point set in the torus consists of a single circle $C$, not
two.  In the universal cover $\mathbb{C}$ this circle can be described
as the diagonal line $x + ix$. Furthermore, the complement of the
circle $C$ in the torus is still connected. The choice of
``fundamental region'' in this case is not completely obvious. Indeed,
we need a new definition -- the notion of a fundamental region in the
sense of definition~\ref{DefinitionFundamentalRegion} is not suitable
for the type II involution:
\begin{definition}
  \label{PseudofundamentalRegion}
  A connected subset $R \subset X$ is called a \emph{pseudofundamental
    region} for the $G$-action on $X$ if there exists a subset $R'
  \subset \partial R$ in the boundary of $R$ such that $R\setminus R'$
  is a fundamental region.
\end{definition}

The idea behind the proof of the classification for type II is to identify a
pseudofundamental region $R \subset X$, then bring the maps into a
certain normal form such that we can collapse certain parts of
$\partial R$ and the maps push down to this quotient. In the quotient,
the geometry of the $T$-action is easier to understand and the image
of $R$ in the quotient will be fundamental region for the induced
$T$-action on the quotient.

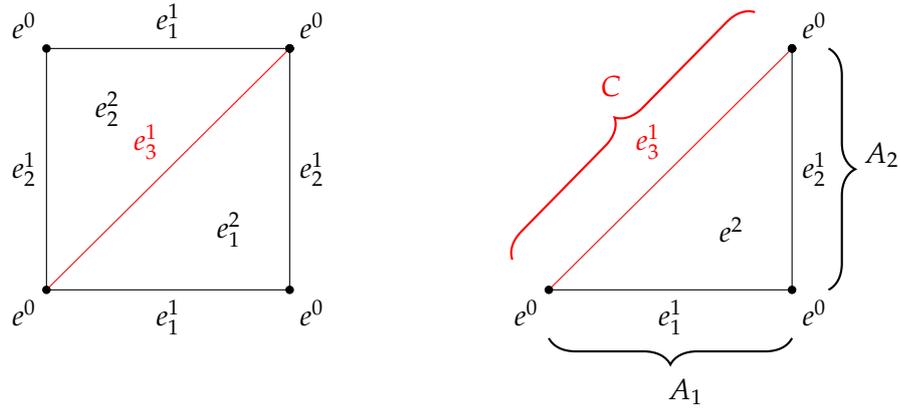
\begin{figure}[h]
  \centering
  \subfloat[$G$-CW Decomposition of $X$.]{%
    \label{MSInvolutionFigureCW}%
    \centering%
    \begin{tikzpicture}[scale=0.8]
      \draw
      (4,4) coordinate (a) node[above right] {$e^0$}
      (0,0) coordinate (b) node[below left] {$e^0$}
      (4,0) coordinate (c) node[below right] {$e^0$}
      (0,4) coordinate (d) node[above left] {$e^0$}
      (b) -- node [below] {$e^1_1$} (c)
      (a) -- node [right] {$e^1_2$} (c);
      \draw[color=red] (b) -- node [above left] {$e^1_3$} (a);
      \draw
      (b) -- node[left] {$e^1_2$} (d)
      (d) -- node[above] {$e^1_1$} (a);

      \node[align=center] at (3,1) {$e^2_1$};
      \node[align=center] at (1,3) {$e^2_2$};

      \fill [radius=2pt] (0,0) circle
      [] (4,4) circle
      [] (4,4) circle
      [] (4,0) circle
      [] (0,4) circle;

      \useasboundingbox (-2,5.5) rectangle (5.5,-2.5);
    \end{tikzpicture}}\qquad
  \subfloat[Pseudofundamental region $R \subset X$.]{%
    \label{MSInvolutionFigureCWDelta}%
    \centering%
    \begin{tikzpicture}[scale=0.8]
      \draw
      (4,4) coordinate (a) node[above right] {$e^0$}
      (0,0) coordinate (b) node[below left] {$e^0$}
      (4,0) coordinate (c) node[below right] {$e^0$}
      (b) -- node [below] {$e^1_1$} (c)
      (a) -- node [right] {$e^1_2$} (c);
      \draw[color=red] (b) -- node [above left] {$e^1_3$} (a);

      \node[align=center] at (3,1) {$e^2$};

      \draw[thick,decorate,decoration={brace,mirror,amplitude=10pt}] (0,-0.8) -- (4,-0.8) node[shift={(-1.4,-0.4)},below] {$A_1$};
      \draw[thick,decorate,decoration={brace,mirror,amplitude=10pt}] [xshift=6mm](4,0) -- (4,4) node[shift={(0.7,-1.1)},below] {$A_2$};
      \draw[color=red,thick,decorate,decoration={brace,amplitude=10pt}] (-0.6,0.5) -- (3.4,4.6) node[shift={(-1.9,-0.7)},below] {$C$};

      \fill [radius=2pt] (0,0) circle
      [] (4,4) circle
      [] (4,4) circle
      [] (4,0) circle;

      \useasboundingbox (-1.8,5) rectangle (5.8,-2.5);
    \end{tikzpicture}}
  \caption{Type II $G$-CW Setup of the Torus $X$.}.
\end{figure}

The pseudofundamental region we use for the type II $T$-action on $X$ is
\begin{align*}
  R = \left\{[x+iy]\colon 0 \leq x \leq 1 \text{ and } 0 \leq y \leq x\right\}
\end{align*}

See figure~\ref{MSInvolutionFigureCWDelta} for a depiction of this
fundamental region. Note that $T$-equivariance of a map requires the
values of this map along the circles $A_1$ and $A_2$ to be compatible,
because $T$ maps $A_1$ to $A_2$ and vice versa. To make this precise
(and to ease the notation) we make the following definition.
\begin{definition}
  \label{Class2CompatibilityCondition}
  For $f\colon X \to S^2$ resp. $f\colon R \to S^2$ we let
  $f_{A_j}\colon I \to S^2$ ($j=1,2$) be the map
  \begin{align*}
    f_{A_j}\colon t \mapsto 
    \begin{cases}
      f\left([t]\right) & \;\text{for $j=1$}\\
      f\left([1 + it]\right) & \;\text{for $j=2$}.
    \end{cases}
  \end{align*}
  For convenience we also set $A = A_1 \cup A_2$.
\end{definition}

Then we can reformulate what it means for a map $f\colon X \to S^2$ to
be equivariant:
\begin{remark}
  \label{RemarkClassIIEquivariance}
  Let $F\colon X \to S^2$ be a $G$-map. Then its restriction $f =
  \restr{F}{R}$ to $R$ satisfies
  \begin{enumerate}[(i)]
  \item $f(C) \subset E$ and
  \item $f_{A_1} = T \circ f_{A_2}$.
  \end{enumerate}
\end{remark}

\begin{lemma}
  \label{Class2EquivariantExtension}
  Maps $R \to S^2$ resp. homotopies $I \times X \to S^2$ with the
  above two properties extend uniquely to equivariant maps $X \to S^2$
  resp. equivariant homotopies $I \times X \to S^2$.
\end{lemma}
\begin{proof}
  This is basically the same proof as for
  lemma~\ref{Class1EquivariantExtension}. We let $R'$ be the opposite
  pseudofundamental region, which has the intersection $A_1 \cup A_2
  \cup C$ with $R$. Then we define the extension of the map on $R'$ to
  be the $T$-conjugate of the map on $R$. The compatibility condition
  along $A_1 \cup A_2$ (from
  definition~\ref{Class2CompatibilityCondition}) guarantees that the
  resulting extension is globally well-defined on $X$.
\end{proof}

\paragraph{Homotopy Invariant}

In the type I classification we have identified degree \emph{triples}
as the desired (complete) homotopy invariant; in the type II case we
need something slightly different. In order to obtain well-defined
notion of fixed point degree we must choose an orientation on the
fixed point circle $C \subset X$ (compare with
p.~\pageref{CjS1Identification}). \label{CircleOrientationRemark}
Although we assume for this paragraph that an orientation has already
been chosen, we will wait until paragraph~\ref{RedcutionToClassI}
(p.~\pageref{RedcutionToClassI}) before we fix this orientation on
$C$.

\begin{definition}
  \label{DefinitionDegreePair} (Degree pair map)
  The \emph{degree pair map} is the map
  \begin{align*}
    \mathcal{P}\colon \mathcal{M}_G(X,S^2) &\to \Z^2\\
    f &\mapsto \Pair{d_C}{d},
  \end{align*}
  where $d$ is the total degree of the map $f$ and $d_C$ is the
  fixed point degree of $f$. For a map $f\colon X \to S^2$ we define its
  \emph{degree pair} (or simply \emph{pair}) to be
  $\mathcal{P}(f)$. We call a given pair \emph{realizable} if it is
  contained in the image $\Im(\mathcal{P})$.
\end{definition}
We need to be careful to distinguish the \emph{degree pair}
$\Pair{d_C}{d}$ of an equivariant map $X \to S^2$ (for the type II
involution) from the \emph{fixed point degrees} $\Bideg{d_0}{d_1}$ of an
equivariant map $X \to S^2$ (for the type I involution), but this
should always be clear from the context. After having defined the
degree pair, we immediately obtain:
\begin{remark}
  \label{DegreePairInvariant}
  The degree pair is a $G$-homotopy invariant. \qed
\end{remark}

\paragraph{Reduction to Type I}
\label{RedcutionToClassI}

Although, at first glance, the type II involution appears to be quite
different from the type I involution, it is nevertheless possible to
use the results from the type I classification
(section~\ref{SectionMapsToSphereClassI}).

\begin{definition}
  \label{DefinitionClass2Normalized} (Type II normalization)
  Let $f\colon X \to S^2$ be a $G$-map. We say $f$ is \emph{type II
    normalized} if
  \begin{enumerate}[(i)]
  \item $f_{A_1} = f_{A_2} = c_{p_0}$, where
    $c_{p_0}\colon X \to S^2$ denotes the constant map whose image is
    $\{p_0\} \subset E$ and
  \item $f$ is normalized on the circle $C$ in the sense that, using
    the identifications of $C$ and $E$ with $S^1$, the map
    $\restr{f}{C}$ corresponds to $z \mapsto z^k$ for some
    $k$.
  \end{enumerate}
\end{definition}

\begin{proposition}
  \label{Class2Normalization}
  Any $G$-map $f\colon X \to S^2$ is $G$-homotopic to a map $f'$ which
  is type II normalized.
\end{proposition}

\begin{proof}
  The statement is proved in several steps. Consider $X$ with the
  $G$-CW-decomposition as in figure~\ref{MSInvolutionFigureCW}
  consisting of the six cells $e^0$, $e^1_{1,2,3}$, and
  $e^2_{1,2}$. Restrict $f$ to the pseudo-fundamental region $R$. Then
  $A = A_1 \cup A_2$ is a subcomplex of $R$. We now successively apply
  two homotopies to $f$. Observe that the maps $f_{A_j}$ ($j=1,2$) are
  loops in $S^2$ at some point $p \in E$. Using the
  simply-connectedness of $S^2$ there exists a null-homotopy
  $\rho\colon I \times I \to S^2$ of $f_{A_1}$ to the constant curve
  $c_p$. For the first homotopy, we define $H$ on $\{0\} \times R$ to
  be the original map $f$ and on $I \times A_1$ resp. $I \times A_2$
  we set
  \begin{align*}
    (t,s) \mapsto \rho(t,s) \;\;\text{resp.}\; (t,1+is) \mapsto T \circ \rho(t,s).
  \end{align*}
  By the homotopy extension property (HEP, see proposition~\ref{HEP})
  this can be extended to a homotopy $H\colon I \times R \to S^2$ and
  the resulting map $H(1,\cdot)$, which we denote by
  $\smash{\tilde{f}}$, has the property that
  $\smash{\tilde{f}}(A_j) = \{p\}$ ($j=1,2$).

  Define the second homotopy as follows: On $\{0\} \times R$ we define
  $H$ to be the map $\smash{\tilde{f}}$. On $I \times C$ we make a
  homotopy $H_C\colon I \times C \to E$ to the normalized form $z
  \mapsto z^k$. In particular, $e_0$ will then be mapped to $p_0 \in
  E$. But when we change the map on $C$ we only need to adjust it
  accordingly on $A$, hence we define the homotopy on $I \times A_1$
  resp. $I \times A_2$ to be
  \begin{align*}
    (t,s) \mapsto H_C(t,0) \;\text{resp.}\; (t,1+is) \mapsto H_C(t,0).
  \end{align*}
  The map $H$ is then well-defined on $I \times e_0$ and extends by
  the HEP to a homotopy $I \times R \to S^2$. By
  lemma~\ref{Class2EquivariantExtension}, we can equivariantly extend
  this homotopy to a homotopy $I \times X \to S^2$ and the resulting
  map has the desired properties.
\end{proof}

For the reduction to type I it will be very useful to replace the
torus $X$ with the quotient $X/A$. This will be formalized next. With
the above remark we can assume without loss of generality that any map
in $\mathcal{M}_G(X,S^2)$ is type II normalized. Type II normalized
maps $f\colon X \to S^2$ push down to the quotient $X/A \to S^2$. Let
us denote the topological quotient $X \to X/A$ by $\pi_A$ and the
image of $A$ under $\pi_A$ with $A/A$.  In the next lemma we
study the geometry of $X/A$.
\begin{lemma}
  \label{ClassIIQuotientSphereIdentification}
  We state:
  \begin{enumerate}[(i)]
  \item The $T$-action on $X$ pushes down to the quotient.
  \item The projection map $\pi_A\colon X \to X/A$ is equivariant.
  \item The space $(X/A,T)$ can then be equivariantly identified with
    $(S^2, T)$, where $T$ acts on $S^2$ as the standard reflection
    along the equator.
  \end{enumerate}
\end{lemma}

\begin{proof}
  Regarding (i): First we prove that the $T$-action on $X$ pushes down
  to $X/A$: The $T$-action stabilizes $A \subset X$. Hence, $T$ is
  compatible with the relation on $X$, which identifies the points in
  $A$. By theorem~\ref{TopologyMapInducedOnQuotient} $T$ therefore
  induces a map $X/A \to X/A$, which we, by abuse of notation, again
  denote by $T$. Since $T\colon X \to X$ is an involutive
  homeomorphism on $X$, the same applies to $T\colon X/A \to X/A$ (see
  e.\,g. theorem~\ref{TopologyMapInducedOnQuotient}).

  Regarding (ii): This is true by definition of the $T$-action on the
  quotient: $T([x]) = [T(x)]$.

  Regarding (iii): The space $X/A$ consists of two $2$-cells, glued
  together along their boundary, which is stabilized by $T$. We can
  define an equivariant homeomorphism to $S^2$ by mapping one of the
  $2$-cells to e.\,g. the lower hemisphere of the sphere, mapping the
  boundary of the $2$-cell to the equator. The map on the second cell
  can then be defined by equivariant extension, thus mapping it to the
  upper hemisphere.
\end{proof}

Thus we can from now on identify $X/A$ with $S^2$ as topological
manifolds and by this identification equip $X/A$ with the smooth
structure from $S^2$.  We denote the point in $S^2$ corresponding to
$A/A \in X/A$ by $P_0$. The projection map $\pi_A$ induces an
equivariant map
\begin{align}
  \label{TypeIIDiscussionS2}
  \pi_{S^2}\colon X \to S^2,
\end{align}
whose restriction $X\setminus A \to S^2\setminus\{P_0\}$ is smooth.
The degree of $\pi_{S^2}$ is either $+1$ or $-1$ but we can assume that
it is $+1$ (if it were $-1$, we could compose with a reflection along
the equator in $S^2$). We must be careful to not confuse the 2-sphere
introduced in (\ref{TypeIIDiscussionS2}) with the 2-sphere which has
been identified as a deformation retract of $\mathcal{H}_{(1,1)}$. In
particular their equators must be considered as distinct objects. Thus
we denote the equator of the 2-sphere which appears as a quotient of
the torus $X$ by $E'$.

Now, let $(X',T)$ be a type I torus, that is, the torus $\C/\Lambda$
equipped with the type I involution $T$ and denote the north resp. the
south pole of $S^2$ by $O_\pm$. Observe that also the sets
$X'\setminus C_0$ and $S^2\setminus\{O_\pm\}$ are $G$-spaces. The
former is an open cylinder and the latter is a doubly punctured
2-sphere. We note:
\begin{remark}
  \label{Class1TorusWithSphereIdentification}
  There exists a smooth and orientation preserving $G$-diffeomorphism
  \begin{align*}
    \Psi\colon X'\setminus C_0 \xrightarrow{\;\sim\;} S^2\setminus \{O_\pm\}.
  \end{align*}
\end{remark}

\begin{proof}
  Regard $S^2$ as being embedded as the unit-sphere in $\R^3$. Then we
  define the map $\Psi$ as follows:
  \begin{align*}
    \Psi\colon X'\setminus C_0 &\xrightarrow{\;\sim\;} S^2\setminus \{O_\pm\}\\
    [x+iy] &\mapsto
    \begin{pmatrix}
      \sqrt{1-(2y-1)^2} \cos(2\pi x) \\
      \sqrt{1-(2y-1)^2} \sin(2\pi x) \\
      2y-1
    \end{pmatrix}
  \end{align*}
  This defines an equivariant and orientation preserving
  diffeomorphism.
\end{proof}

On p.~\pageref{CircleOrientationRemark} we mentioned that we still
need to decide for an orientation on the circle $C \subset X$, which
is equivalent to chosing an identification of $C$ with $S^1$ -- we
will do this now. Observe that the restriction of the above map $\Psi$
to $C_1 \subset X'$ identifies the circle $C_1$ with the equator $E'$
in $S^2$, which the map $\pi_{S^2}$ identifies with with $C \subset
X$. We define the identifications of $C \subset X$ and $E' \subset S^2$
with $S^1$ in terms of this identification with $C_1 \subset
X'$. Analogously to definition~\ref{DefinitionClass2Normalized} we
make the following
\begin{definition}
  We call a $G$-map $f\colon S^2 \to S^2$ \emph{type II
    normalized}\footnote{as a map defined on the 2-sphere, not on the
    type II torus $X$}, if the north pole $O_+$ and the south pole
  $O_-$ both are mapped to $p_0 \in E \subset S^2$ and the map is
  normalized on the equator $E'$ according to the above identification
  of $E'$ with $S^1$.
\end{definition}

\begin{remark}
  Any $G$-map $S^2 \to S^2$ is equivariantly homotopic to a type II
  normalized $G$-map.
\end{remark}

\begin{proof}
  This follows with the equivariant homotopy extension
  (corollary~\ref{G-HEP}): There exists a $G$-CW decomposition of
  $S^2$ such that the set $B = \{O_\pm\} \cup E'$ is a $G$-CW
  subcomplex of $S^2$.
\end{proof}

To summarize the above discussion: Above we have constructed a map
from the set of type II normalized $G$-maps $X \to S^2$ to the set of
type II normalized $G$-maps $S^2 \to S^2$. We denote this map by
$\psi$ and state:
\begin{proposition}
  \label{Class2CorrespondenceXS2}
  The map $\psi$ is a bijection.  Furthermore, the degree pairs of $f$
  and $\psi(f)$ agree.
\end{proposition}

\begin{proof}
  Let us begin by reviewing how the map $\psi$ is defined: If $f\colon
  X \to S^2$ is a type II normalized map, it is by definition constant
  along $A$. Hence it pushes down to a $G$-map $f\colon X/A \to
  S^2$. Now $X/A$ can be equivariantly identified with $S^2$, hence,
  by means of this identification, $f$ defines a $G$-map $S^2 \to
  S^2$. By construction these two maps are related by the following
  diagram:
  \[
  \begin{xy}
    \xymatrix{
      X \ar[d]_{\pi_{S^2}} \ar[dr]^f \\
      S^2 \ar[r]_{f'} & S^2
    }
  \end{xy}
  \]
  Injectivity of this map is clear: If $\psi(f) = \psi(g)$, then also
  \begin{align*}
    f = \psi(f) \circ \pi_{S^2} = \psi(g) \circ \pi_{S^2} = g.
  \end{align*}
  For the surjectivity, let $f'\colon S^2 \to S^2$ be a type II
  normalized $G$-map, then $f = f' \circ \pi_{S^2}\colon X \to S^2$ is a
  type II normalized $G$-map such that $\psi(f) = f'$. 

  Since $\pi_{S^2}$ has degree $+1$ and using the functorial property
  of homology we see that the total degrees of $f$ and $f'$ must
  agree. Regarding the fixed point degree: $\pi_{S^2}$ homeomorphically
  maps the circle $C \subset X$ to the equator $E \subset S^2$. Now
  the only possibility would be that the fixed point degrees differ by a
  sign. But, by definition, the identifications of $C \subset X$ and
  $E \subset S^2$ with $S^1$ differ only by $\pi_{S^2}$ and above the
  orientation of $C \subset X$ has been chosen such that
  $\restr{\pi_{S^2}}{C}$ preserves the orientation.
\end{proof}

This correspondence between type II normalized $G$-maps $X \to S^2$
and type II normalized $G$-maps $S^2 \to S^2$ also implies a statement
on the level of equivariant homotopy:
\begin{remark}
  \label{Class2HomotopiesLiftFromSphere}
  Given two type II normalized $G$-maps $f,g\colon X \to S^2$ and an
  equivariant homotopy $H\colon S^2 \to S^2$ between the induced maps
  $f',g'\colon S^2 \to S^2$, then $f$ and $g$ are also equivariantly
  homotopic.
\end{remark}

\begin{proof}
  Compose the homotopy with the map $\id \times \pi_{S^2}$.
\end{proof}

We now describe a basic correspondence between type II normalized
$G$-maps $S^2 \to S^2$ and type I normalized $G$-maps $X' \to
S^2$. Assume $f$ is a type II normalized $G$-map $S^2 \to S^2$. This
induces a a map
\begin{align*}
  \phi(f)\colon X' &\to S^2\\
  p &\mapsto
  \begin{cases}
    f \circ \Psi (p) & \;\text{if $p \not\in C_0$}\\
    p_0 & \;\text{if $p \in C_0$}
  \end{cases}
\end{align*}
Since $f$ is type II normalized, i.\,e. $f(O_\pm) = p_0$, the map
$\phi(f)$ is continuous everywhere and type I normalized as a map $X'
\to S^2$. By construction $f$ and $\phi(f)$ are related by
\begin{align}
  \label{PsiCorrespondence}
  \Restr{f}{S^2\setminus\{O_\pm\}} \circ \Psi = \Restr{\phi(f)}{X'\setminus
    C_0},
\end{align}
where $\Psi$ is the orientation preserving diffeomorphism from
remark~\ref{Class1TorusWithSphereIdentification}.  The reduction to
type I is now summarized in the following two statements:
\begin{proposition}
  \label{LemmaClass2ReductionToClass1}
  The above map $\psi$ defines a bijection between type II normalized
  $G$-maps $S^2 \to S^2$ and type I normalized $G$-maps $X' \to S^2$ whose
  fixed point degree along the circle $C_0$ is zero. For a type II
  normalized $G$-map $f$, the degree pair of $f$ is $\Pair{d_C}{d}$
  iff the degree triple of $\phi(f)$ is $\Triple{0}{d}{d_C}$.
\end{proposition}

\begin{proof}
  Regarding the injectivity: Let $f$ and $g$ be two type II normalized
  $G$-maps $S^2 \to S^2$ such that $\phi(f) = \phi(g)$.  By
  (\ref{PsiCorrespondence}) it follows that
  $\restr{f}{S^2\setminus\{O_\pm\}} =
  \restr{g}{S^2\setminus\{O_\pm\}}$. Since $f$ and $g$ are assumed to
  be type II normalized, hence in particular continuous, this implies
  that $f = g$. Regarding the surjectivity: Let $f'$ be a type I
  normalized $G$-map $X' \to S^2$ with fixed point degree
  $\Bideg{0}{d_1}$. Then (\ref{PsiCorrespondence}) defines a map
  $f\colon S^2\setminus\{O_\pm\} \to S^2$ which, by defining $f(O_\pm) =
  p_0$, uniquely extends to a type II normalized $G$-map $S^2 \to
  S^2$. By construction we obtain $\phi(f) = f'$.
  
  Regarding the correspondence between the degree invariants: Assume
  we are given a $G$-map $f\colon S^2 \to S^2$ with degree pair
  $\Pair{d_C}{d}$. By corollary~\ref{SmoothApproximation2}, $f$ can be
  equivariantly smoothed. Note that the map $\Psi$ in
  (\ref{PsiCorrespondence}) is an orientation preserving
  diffeomorphism. This implies that $\phi(f)$ is smooth, at least away
  from the boundary circle $C_0$. Using
  theorem~\ref{SmoothApproximation1} we can smooth $\phi(f)$ near the
  circle $C_0$ while keeping it unchanged away from $C_0$. Now, let
  $q$ be a regular value of $f$ such that its preimage points
  $p_1,\ldots,p_m$ are contained in $S^2\setminus\{O_\pm\}$. Then $q$
  is also a regular of $\phi(f)$ and its preimage points are
  $\Psi^{-1}(p_j)$, $j=1,\ldots,m$. Since $\Psi$ preserves the
  orientation, $\deg f = \deg \phi(f)$. By definition of the
  identification of $E \subset S^2$ with $S^1$ (through $\Psi$), the
  fixed point degree $d_C$ remains unchanged as well. Finally, by
  construction, the fixed point degree of $\phi(f)$ along $C_0$ is zero.
\end{proof}

Now we can show that equivariant homotopies for maps $X' \to S^2$
induce equivariant homotopies for maps $S^2 \to S^2$:
\begin{proposition}
  \label{LemmaClass2ReductionToClass1Homotopy}
  Given two type II normalized $G$-maps $f,g\colon S^2 \to S^2$ with
  the same degree pair, then they are $G$-homotopic rel $\{O_\pm\}$.
\end{proposition}

\begin{proof}
  Without loss of generality we can assume that $f$ and $g$ are both
  type II normalized $G$-maps $S^2 \to S^2$ with the same degree pair.
  Then, by proposition~\ref{LemmaClass2ReductionToClass1}, the induced
  maps $\phi(f),\phi(g)\colon X' \to S^2$ share the same degree
  triple. Hence there exists a $G$-homotopy $H'$ from $\phi(f)$ to
  $\phi(g)$ which is in particular relative with respect to the
  boundary circle $C_0$. In other words, throughout the homotopy, $H'$
  maps the circle
  $C_0$ to $p_0 \in E \subset S^2$. This homotopy then induces a an
  equivariant homotopy $H$ between $f$ and $g$:
  \begin{align*}
    H\colon I \times S^2 &\to S^2\\
    (t,p) &\to
    \begin{cases}
      H'(t,\Psi^{-1}(p)) & \;\text{if $p \not\in \{O_\pm\}$}\\
      p_0 & \;\text{if $p \in \{O_\pm\}$}
    \end{cases}
  \end{align*}
  Hence, $f$ and $g$ are equivariantly homotopic rel $\{O_\pm\}$.
\end{proof}

\paragraph{Main Result}

The main result in this section is the following
\begin{theorem}
  \label{Classification2}
  The $G$-homotopy class of a map $f \in \mathcal{M}_G(X,S^2)$ is
  uniquely determined by its degree pair $\mathcal{P}(f)$. The image
  $\Im(\mathcal{P})$ of the degree pair map
  $\mathcal{P}\colon \mathcal{M}_G(X,S^2) \to \Z^2$ consists of those
  pairs $\Pair{d_C}{d}$ satisfying
  \begin{align*}
    d \equiv d_C \mod 2.
  \end{align*}
\end{theorem}

\begin{proof}
  First we prove that the degree pairs $\Pair{d_C}{d}$ are a complete
  invariant of equivariant homotopy for maps $X \to S^2$. For this,
  let $f$ and $g$ be two $G$-maps $X \to S^2$ with the same degree
  pair $\Pair{d_C}{d}$. Without loss of generality we can assume that
  $f$ and $g$ are type II normalized
  (Proposition~\ref{Class2Normalization}). In this case, by
  proposition~\ref{Class2CorrespondenceXS2}, $f$ and $g$ push down to
  type II normalized maps $f',g'\colon S^2 \to S^2$ with the same
  degree pair. With
  proposition~\ref{LemmaClass2ReductionToClass1Homotopy} it now
  follows that $f'$ and $g'$ are equivariantly homotopic. Then
  remark~\ref{Class2HomotopiesLiftFromSphere} establishes the
  existence of an equivariant homotopy from $f$ to $g$.

  Regarding the image of $\Im(\mathcal{P})$: Let $d_C$ be an
  integer. We now have to compute all possible integers $d$ such that
  the degree pair $\Pair{d_C}{d}$ is contained in the
  $\Im(\mathcal{P})$. Let us first deduce a necessery condition for
  the given degree pair to be in the image $\Im(\mathcal{P})$. For
  this, let $f\colon X \to S^2$ be a $G$-map with the aforementioned
  degree pair. By proposition~\ref{Class2CorrespondenceXS2} we obtain an
  induced $G$-map $f'\colon S^2 \to S^2$ with the same degree pair. Now
  proposition~\ref{LemmaClass2ReductionToClass1} implies the existence
  of an induced map $\phi(f)\colon X' \to S^2$ with the degree triple
  $\Triple{0}{d}{d_C}$. By theorem~\ref{Classification1}, we must have
  $d \equiv d_C \mod 2$. This shows that the condition is necessary
  for a pair to be contained in $\Im(\mathcal{P})$; what remains to
  show is that this condition is also sufficient. Let $f''\colon X'
  \to S^2$ be a $G$-map with degree triple $\Triple{0}{d}{d_C}$. By
  proposition~\ref{LemmaClass2ReductionToClass1} there exists a
  $G$-map $f'\colon S^2 \to S^2$ with the degree pair $\Pair{d_C}{d}$
  and by proposition~\ref{Class2CorrespondenceXS2} this induces a
  $G$-map $f\colon X \to S^2$ with the same degree pair. This proves the
  statement.
\end{proof}

\subsection{Classification of Maps to $\mathcal{H}_2^*$}
\label{SectionHamiltonianClassificationRank2}

Note that so far the degree triple map (resp. the degree pair map) has
only been defined on $\mathcal{M}_G(X,S^2)$. But since $S^2$ has been
shown to be an equivariant strong deformation retract, we can also
define the degree triple map (resp. degree pair map) on
$\mathcal{M}_G(X,\mathcal{H}_{(1,1)})$. In particular this allows us
to speak of degree triples (resp. of degree pairs) of $G$-maps $X \to
\mathcal{H}_{(1,1)}$. Now we state and prove the main result of this
section:

\HamiltonianClassificationRankTwo
\begin{proof}
  \label{HamiltonianClassificationRankTwoProof}
  The fact that the cases $\sig = (2,0)$ and $\sig = (0,2)$ constitute
  $G$-homotopy classes on their own has already been shown at the
  beginning of this chapter
  (remark~\ref{Rank2DefiniteComponentsContractible}). Regarding the
  case $\sig = (1,1)$: By remark~\ref{Rank2MixedSignatureReduction},
  $\mathcal{H}_{(1,1)}$ has $S^2$ as equivariant, strong deformation
  retract. Therefore,
  \begin{align*}
    [X,\mathcal{H}_{(1,1)}]_G \cong [X,S^2]_G.
  \end{align*}
  Then theorem~\ref{Classification1} (for the type I involution) and
  theorem~\ref{Classification2} (for the type II involution) complete
  the proof.
\end{proof}

\subsection{An Example from Complex Analysis}

The Weierstrass $\wp$-function is a meromorphic and doubly-periodic
function on the complex plane associated to a lattice $\Lambda$. It
can be defined as follows:
\begin{align*}
  \wp(z) = \frac{1}{z^2} + \sum_{\lambda \in \Lambda\setminus\{0\}} \left(\frac{1}{(z-\lambda)^2} - \frac{1}{\lambda^2}\right).
\end{align*}
This function can be regarded as a holomorphic map to $\Proj_1$. Since
the lattice $\Lambda$ is stable under complex conjugation, we have the
identity
\begin{align*}
  \wp(\overline{z}) = \overline{\wp(z)}.
\end{align*}
By the orientation-preserving idenfication $\Proj_1 \cong S^2$ (see
p.~\ref{P1S2OrientationDiscussion})we can regard the map $\wp$ as an
equivariant map $X \to S^2$. Since its total degree as a map to
$\Proj_1$ is two, its degree as a map to $S^2$ is also
two. Analogously, the derivative $\wp'$ defines an equivariant map $X
\to \Proj_1 \cong S^2$ of degree three. Furthermore we state:
\begin{remark}
  \label{BidegreeOfWP}
  The Weierstrass $\wp$-function (resp. $\wp'$), regarded as a map
  \begin{align*}
    X \to \Proj_1 \cong S^2 \hookrightarrow \mathcal{H}_{(1,1)}
  \end{align*}
  has fixed point degrees $\Bideg{0}{0}$ (resp. $\Bideg{1}{0}$).
\end{remark}
\begin{proof}
  Regarding $\wp$: The restriction $\restr{\wp}{C_0}$ to be boundary
  circle $C_0$ is not surjective to the compactified real line
  $\smash{\widehat{\R}}$, since in the square lattice case, the only zero of
  $\wp(z)$ is at the point $z=\text{\textonehalf}(1+i)$ (see
  e.\,g. \cite{EichlerZagier}).  Similarly, the restriction of $\wp$
  to the boundary circle is not surjective to $\RProj_1$, because the
  only pole of $\wp$ is at $z = 0$ (of order 2). This proves that the
  fixed point degrees are both zero.

  Regarding $\wp'$: Since $\wp'(z)$ has its only pole at $z=0$, its
  restriction to the circle $C_1 = I + \text{\textonehalf}i$ cannot be
  surjective to $\smash{\widehat{\R}}$, therefore its fixed point
  degree $d_1$ must be zero. But its fixed point degree $d_0$ is
  $\pm 1$: It is known that the only pole of $\wp'(z_0)$ is at $z=0$
  and its only zero in $C_0$ is at $z=\text{\textonehalf}$. A
  computation shows that $\wp'$ is negative along
  $(0,\text{\textonehalf})$. In other words, along the curve segment
  $(0,\text{\textonehalf})$, the image under $\wp'$ moves from
  $\infty$ to the negative real numbers and finally to $0$. By the
  identity
  \begin{align*}
    \wp'(-z) = -\wp'(z),
  \end{align*}
  it follows that on the curve segment $(\text{\textonehalf})$ the
  image under $\wp'$ moves from $0$ to the positive real numbers until
  it finally reaches $\infty$. By definition of the orientation on
  $\smash{\widehat{\R}}$ (see the discussion on
  p.~\ref{P1S2OrientationDiscussion}), this is a loop of degree $+1$.
\end{proof}

For the type II involution we consider the scaled Weierstrass
functions
\begin{align*}
  &F = i\wp\\
  \text{and}\; &G = e^{i\frac{3\pi}{4}} \wp'.
\end{align*}
As above, they can be considered as equivariant maps $X \to S^2$.
\begin{remark}
  The fixed point degree of $F$ is $0$ and that of $G$ is $1$.
\end{remark}
\begin{proof}
  Regarding $F$: The map $F(z)$ has its only pole at $z=0$ and its
  only zero at $\text{\textonehalf}(1+i)$. Thus, when $z$ linearly
  moves from the origin to the point $\text{\textonehalf}(1+i)$, then
  its image under $F$ defines a curve from $\infty$ to $0$. Since $F$
  is an even function it follows that on the second segment of the
  diagonal circle, $F$ reverses the previous curve, going back from
  $0$ to $\infty$. Thus the fixed point degree for $F$ is $0$.

  Regarding $G$: Along the curve segment from $0$ to
  $\text{\textonehalf}(1+i)$ it defines a curve from $\infty$ to
  $0$. A computation shows that this curve is along the negative real
  numbers. Since $G$ is not an odd map, its restriction to the segment
  from $\text{\textonehalf}(1+i)$ to $1+i$ defines a curve from $0$ to
  $\infty$, but with the opposite sign, i.\,e. along the positive real
  numbers. This defines a degree $+1$ loop.
\end{proof}
To summarize the above:
\begin{remark}
  With respect to the identification $\Proj_1 \cong S^2$ described on
  p.~\ref{P1S2OrientationDiscussion}, we have:
  \begin{align*}
    &\mathcal{T}(\wp) = \Triple{0}{2}{0}\\
    &\mathcal{T}(\wp') = \Triple{1}{3}{0}\\
    &\mathcal{P}(i\wp) = \Pair{0}{2}\\
    &\mathcal{P}(e^{i\frac{3\pi}{4}} \wp') = \Pair{1}{3}.
  \end{align*}
\end{remark}

By choosing a different equivariant identification $\Proj_1 \cong
S^2$, we might introduce signs for the total degrees or for the fixed
point degrees or for both. For instance, we can always compose the
$G$-diffeomorphism $\Proj_1 \to S^2$ with reflections on $S^2$. Thus,
in some sense, the numbers in the above remark are only well-defined
up to sign.

\section{Maps to  $\mathcal{H}_n$}

In this chapter we generalize the results of the previous section to
arbitrary $n > 2$. For this we begin by noting that the unitary group
$U(n)$ acts on $\mathcal{H}_{(p,q)}$ by conjugation. Since every
matrix in $\mathcal{H}_{(p,q)}$ can be diagonalized by unitary
matrices, the set of $U(n)$-orbits in $\mathcal{H}_{(p,q)}$ is
parametrized by the set of (unordered) real eigenvalues $\lambda_1^+,
\ldots, \lambda_p^+, \lambda_1^-,\ldots, \lambda^-_q$. The involution
$T$ stabilizes the components $\mathcal{H}_{(p,q)}$, since the
matrices $H$ and $T(H)$ clearly have the same spectrum. The $T$-action
is also compatible with the $U(n)$-orbit structure of
$\mathcal{H}_{(p,q)}$, as shown in the next remark:
\begin{remark}
  \label{RemarkRealStructureRestrictsToEachOrbit}
  The $T$-action on $\mathcal{H}_n$ stabilizes each $U(n)$-orbit in
  each $\mathcal{H}_{(p,q)}$.
\end{remark}
\begin{proof}
  Let $D = \Diag(\lambda_1,\ldots,\lambda_n)$ be a diagonal matrix in
  $\mathcal{H}_{(p,q)}$. In particular all $\lambda_j$ are non-zero
  real numbers. Let $\mathcal{O}$ be the $U(n)$-orbit of $D$. We have
  to show that $\mathcal{O}$ is $T$-stable. For this, let $M$ be a
  matrix in $\mathcal{O}$. By definition it is of the form $M = UDU^*$
  for some $U \in U(n)$. Then we have
  \begin{align*}
    T(UDU^*) = \overline{UDU^*} = \overline{U}D\overline{U}^*.
  \end{align*}
  But since $U \in U(n)$ implies $\overline{U} \in U(n)$, we can
  conclude that $T(UDU^*)$ is again contained in the orbit
  $\mathcal{O}$.
\end{proof}

In each component $\mathcal{H}_{(p,q)}$ there is a $U(n)$-orbit which
is particularly convenient to work with:
\begin{remark}
  \label{MinimalOrbitInComponent}
  The $U(n)$-orbit of the block diagonal matrix
  \begin{align*}
    \I{p}{q} = \begin{pmatrix}
      \I{p} & 0 \\
      0 & -\I{q}
    \end{pmatrix}
  \end{align*}
  is diffeomorphic to the Grassmannian $\Gr_p(\mathbb{C}^n)$.
\end{remark}

\begin{proof}
  The $U(n)$-isotropy of $\I{p}{q}$ is $U(p) \times U(q)$. Hence, for
  the $U(n)$-orbit we have the following smooth identification:
  \begin{equation*}
    U(n).\I{p}{q} \cong \frac{U(n)}{U(p) \times U(q)} \cong \Gr_p(\mathbb{C}^n)\qedhere
  \end{equation*}
\end{proof}

In order to reduce the classification of maps to
$\mathcal{H}_{(p,q)}$ to the classification of maps to this
Grassmann manifold, it is important to understand the $T$-action on
the latter. For this we make the following remark:
\begin{remark}
  \label{InducedActionOnGrassmannian}
  The induced $T$-action on $\Gr_p(\mathbb{C}^n)$ is given by
  \begin{align*}
    V \mapsto \overline{V},
  \end{align*}
  for each $p$-dimensional subvector space $V \subset \mathbb{C}^n$.
\end{remark}

\begin{proof}
  As a first step we compute the induced $T$-action under the
  identification
  \begin{align*}
    \frac{U(n)}{U(p) \times U(q)} &\xrightarrow{\sim} U(\I{p}{q}) \subset \mathcal{H}_{(p,q)}\\
    [U] &\mapsto U\I{p}{q}U^*.
  \end{align*}
  Now, let $[U]$ be a point in the quotient. The above isomorphism
  sends $[U]$ to $U \I{p}{q} U^*$. The $T$-action on $\mathcal{H}_n$
  maps this to
  $\smash{\overline{U \I{p}{q} U^*} = \overline{U} \I{p}{q}
    \overline{U}^*}$,
  which under the above isomorphism, corresponds to the point
  $\smash{\left[\overline{U}\right]}$ in the quotient. Thus, the
  induced $T$-action on the quotient is via
  $\smash{T\left([U]\right) = \left[\overline{U}\right]}$. For the
  next step we need to use the isomorphism
  \begin{align*}
    \frac{U(n)}{U(p) \times U(q)} &\xrightarrow{\;\sim\;} \Gr_p(\mathbb{C}^n)\\
    [U] &\mapsto U(V_0),
  \end{align*}
  where $V_0$ denotes a fixed base point in the Grassmannian,
  e.\,g. the subvector space spanned by $e_1,\ldots,e_p$.  Now, let $V
  \in \Gr_p(\mathbb{C}^n)$ be a subvector space of $\mathbb{C}^n$
  generated by the vectors $v_1,\ldots,v_p$. This subspace uniquely
  corresponds to a coset $[U]$ such that $U(V_0) = V$. Recalling that
  $V_0$ is by definition generated by the standard basis
  $e_1,\ldots,e_p$, we see that the first $p$ coloumns
  $u_1,\ldots,u_p$ of $U$ constitute a unitary basis of the
  subvector space $V$. The $T$-action on the quotient maps $[U]$ to
  $\smash{\left[\overline{U}\right]}$, which then corresponds to the
  subvector space $\smash{\overline{U}}(V_0)$. This subvector space can be
  described in terms of the basis
  \begin{align*}
    \overline{U}(e_1),\ldots,\overline{U}(e_p),
  \end{align*}
  which is just $\overline{u_1},\ldots,\overline{u_p}$.  Therefore,
  the induced $T$-action on $\Gr_p(\mathbb{C}^n)$ is
  \begin{align*}
    V \mapsto \overline{V},
  \end{align*}
  which finishes the proof.
\end{proof}

As a special case we also note:
\begin{remark}
  \label{InducedActionOnProjectiveSpace}
  The induced action on $\Proj_n = \Gr_1(\mathbb{C}^{n+1})$ in terms
  of homogeneous coordinates is given by
  \begin{align*}
    [z_0:\ldots:z_n] \mapsto \left[\overline{z_0}:\ldots:\overline{z_n}\right].
  \end{align*}
\end{remark}

\begin{proof}
  Follows from remark~\ref{InducedActionOnGrassmannian} together with
  the fact that the homogeneous coordinates for a given line $L =
  \C(\ell_1,\ldots,\ell_{n+1})$ in $\C^{n+1}$ are
  $[\ell_1:\ldots:\ell_{n+1}]$.
\end{proof}

\begin{remark}
  Regard the Grassmannian $\Gr_k(\C^n)$ as being equipped with the
  standard real structure given by complex conjugation. If $T$ is a
  real structure on the $n$-dimensional vectorspace $W$, then there is
  a biholomorphic $G$-map
  \begin{align*}
    \Gr_k(W) \xrightarrow{\;\sim\;} \Gr_k(\C^n).
  \end{align*}
\end{remark}

\begin{proof}
  The real structure $T$ on $W$ induces a decomposition of $W$ as $W =
  W_\R \oplus iW_\R$, where $W_\R$ is an $n$-dimensional, real
  subvector space of $W$. The vectorspace $W$ can be identified with
  $\C^n$ by sending an orthonormal basis $w_1,\ldots,w_n$ of $W_\R$ to
  the standard basis $e_1,\ldots,e_n$. This is equivariant by
  construction and induces an equivariant biholomorphism $\Gr_k(W) \to
  \Gr_k(\C^n)$.
\end{proof}

Let us now look at the topology of the two connected components
$\mathcal{H}_{(n,0)}$ and $\mathcal{H}_{(0,n)}$ of $\mathcal{H}_n^*$:
\begin{remark}
  \label{RemarkDefiniteComponentsRetractable}
  The component $\mathcal{H}_{(n,0)}$ (resp. $\mathcal{H}_{(0,n)}$)
  has $\{\I{n}\}$ (resp. $\{-\I{n}\}$) as a strong
  equivariant deformation retract.
\end{remark}

\begin{proof}
  We only prove the statement for the component $\mathcal{H}_{(n,0)}$.
  For this we define:
  \begin{align*}
    \rho\colon I \times \mathcal{H}_{(n,0)} &\to \mathcal{H}_{(n,0)}\\
    (t,A) &\mapsto (1-t)A + t\I{n}
  \end{align*}
  Clearly, for $t=0$ this is the identity on $\mathcal{H}_{(n,0)}$ and
  for $t=1$ it is constant at the identity $\I{n}$. Hence, in
  order to show that this is a well-defined deformation retract, we
  have to prove that there exists no matrix $A \in
  \mathcal{H}_{(n,0)}$ such that $\rho_t(A)$ is singular for some $t
  \in (0,1)$.

  Let $A$ be a positive definite matrix in $\mathcal{H}_n$. By
  assumption we know that $\det A$, which coincides with the product
  of all eigenvalues of $A$, is positive. Let us assume
  that $\det \rho_t(A) = 0$. By definition this means
  \begin{align}
    \label{PositiveDefiniteRetraction}
    \det ((1-t)A + t\I{n}) = 0.
  \end{align}
  This means that $-t$ is an eigenvalue of $A' := (1-t)A$. But $A'$ is
  just $A$, scaled by a positive real number. Hence $A'$ is also
  positive definite and therefore cannot have a negative eigenvalue
  $-t$. Thus we obtain a contradiction, which proves that $\rho_t$
  really has its image contained in $\mathcal{H}_{(n,0)}$. The proof
  for the positive negative case works completely analogously.

  What is now left to show is equivariance of the above deformation
  retract. But this follows from the fact that the involution on
  $\mathcal{H}_n$ acts by conjugation and this is compatible with
  scaling by real numbers as done in
  (\ref{PositiveDefiniteRetraction}).
\end{proof}

This shows that the components $\mathcal{H}_{(n,0)}$ and
$\mathcal{H}_{(0,n)}$ are equivariantly contractible to a point, thus
the spaces $[X,\mathcal{H}_{(n,0)}]_G$ and $[X,\mathcal{H}_{(0,n)}]_G$
are trivial. On the other hand, as we will see in the following,
interesting topology will occur in the case of mixed signatures: For
$0 < p,q < n$ the components $\mathcal{H}_{(p,q)}$ have the
$U(n)$-orbit of $\I{p}{q}$ as equivariant strong deformation
retract\footnote{Here $\I{p}{q}$ denotes the block diagonal matrix
  with $\I{p}$ in the upper left corner and $-\I{q}$ in the lower
  right corner.}. Proving this requires some preperations; we begin by
introducing the following surjection:
\begin{align*}
  \pi\colon \mathcal{H}_{(p,q)} &\to \Gr_p(\mathbb{C}^n)\\
  H &\mapsto E^+(H),
\end{align*}
where $E^+(H)$ denotes the direct sum of the positive eigenspaces of
$H$. This is a $p$-dimensional subspace of $\C^n$. Fixing a basepoint
$E_0 = \left<e_1,\ldots,e_p\right>$, this defines a fiber bundle with
neutral fiber $F = \pi^{-1}(\{E_0\})$. The fiber consists of those
matrices in $\mathcal{H}_{(p,q)}$ which have $E_0$ as the direct sum
of their positive eigenspaces. Note that $K := U(n)$ acts transitively
on $\Gr_p(\mathbb{C}^n)$. Thus, denoting the stabilizer of $E_0$ in
$K$ with $L$, we can identify $\Gr_p(\mathbb{C}^n)$ with $K/L$. We
equip the product $K \times F$ with a $K$-action given by
multiplication on the first factor:
\begin{align*}
  k(k',H) \mapsto (kk', H).
\end{align*}
The $L$-action on $K \times F$ is given by
\begin{align*}
  \ell(k,H) = (k\ell,\ell^*H\ell).
\end{align*}
It induces the quotient $K \times F \to K \times_L F$. The $K$-action
on $K \times_L F$ is then given by
\begin{align*}
  k[(k',H)] = [k(k',H)] = [(kk',H)].
\end{align*}
The quotient $K \times_L F \to K/L$ is $K$-equivariant by definition
of the respective $K$-actions. We obtain the following diagram of
$K$-spaces:
\[
\xymatrix{
  K \times F \ar[r] \ar[d] & K \times_L F \ar[d]\\
  K \ar[r] & K/L
}
\]

After these remarks we state:
\begin{proposition}
  \label{GrassmannianStrongDeformationRetract}
  The component $\mathcal{H}_{(p,q)}$ has the orbit $U(n).\I{p}{q}$ as
  $T$-equivariant, strong deformation retract.
\end{proposition}
\begin{proof}
  The idea of this proof is to use the above diagram and define a
  strong deformation retract of the fiber $F$ which we then globalize
  to a strong deformation retraction on $K \times_L F$. Using the
  identifications $K \times_L F \cong \mathcal{H}_{(p,q)}$ and $K/L
  \cong \Gr_p(\C^n)$ this defines a $T$-equivariant strong deformation
  retract from $\mathcal{H}_{(p,q)}$ to $U(n).\I{p}{q} \subset
  \mathcal{H}_{(p,q)}$ (see figure~\ref{fig:UnOrbitsInHpq}).

  Note that we have the following isomorphism
  \begin{align}
    \label{EquivariantFiberBundleIsoPsi}
    \Psi\colon K \times_L F &\xrightarrow{\;\sim\;} \mathcal{H}_{(p,q)}\\
    \left[(k,H)\right] &\mapsto kHk^*.\nonumber
  \end{align}
  Set $\Sigma = K.\I{p}{q} \subset \mathcal{H}_{(p,q)}$. This defines
  (the image of) a global section of the bundle $K \times_L F \to
  K/L$. Observe that a matrix $H \in F$ is, by definition,
  positive-definite on $E_0 = \left<e_1,\ldots,e_p\right>$ and
  therefore negative-definite on $E_0^\perp =
  \smash{\left<e_{p+1},\ldots,e_{p+q}\right>}$\footnote{On $\C^n$ we use the
    standard unitary structure.}. This implies that $H$ is of the form
  \begin{align*}
    H =
    \begin{pmatrix}
      H_p & 0 \\
      0 & H_q
    \end{pmatrix},
  \end{align*}
  where $H_p$ and $H_q$ are both hermitian, $H_p$ is positive-definite
  on $E_0$ and $H_q$ is negative-definite on $E_0^\perp$. Now we can
  define a retraction $\rho\colon I \times F \to F$ of the neutral
  fiber:
  \begin{align}
    \label{EquationRetractRho}
    \rho_t(H) =
    \begin{pmatrix}
      (1-t)H_p + t\I{p} & 0 \\
      0 & (1-t)H_q - t\I{q}
    \end{pmatrix}.
  \end{align}
  This is a well-defined map with image in the fiber $F$: The proof of
  remark~\ref{RemarkDefiniteComponentsRetractable} shows that the
  intermediate matrices during a homotopy of a definite matrix to
  $+\I_p$ resp. $-\I_q$ stay definite. This generalizes to the
  situation at hand: Given a matrix $H$ of signature $(p,q)$, then
  $\rho_t(H)$ will be of the same signature for every $t$. Therefore,
  $\rho_t(F) \subset F$ for all $t$. Next, extend $\rho$ to a map
  \begin{align*}
    \hat{\rho}\colon I \times K \times F &\to K \times F\\
    (t,(k,H)) &\mapsto (k,\rho_t(H))
  \end{align*}
  Now we prove that $\smash{\hat{\rho}}$ is $L$-equivariant and
  therefore it pushes down to a map
  \begin{align*}
    I \times K \times_L F \to K \times_L F
  \end{align*}
  For this, let $\ell$ be in $L$. Then
  \begin{align*}
    \ell(\hat{\rho}(t,k,H)) = \ell(k,\rho_t(H)) = (k\ell,\ell^* \rho_t(H) \ell)
  \end{align*}
  On the other hand:
  \begin{align*}
    \hat{\rho}(\ell(t,k,H)) = \hat{\rho}(t,k\ell,\ell^* H \ell) = (k\ell, \rho_t(\ell^* H \ell))
  \end{align*}
  Now, $L$-equivariance follows from the fact that matrix conjugation
  commutes with addition and scalar multiplication of matrices and
  therefore
  \begin{align*}
    \ell^* \rho_t(H) \ell = \rho_t(\ell^* H \ell). 
  \end{align*}
  By the isomorphism $K \times_L F \cong \mathcal{H}_{(p,q)}$,
  $\smash{\hat{\rho}}$ induces a map
  $\smash{\tilde{\rho}}\colon I \times \mathcal{H}_{(p,q)} \to
  \mathcal{H}_{(p,q)}$.
  Note that $\rho_0$ and therefore also $\smash{\hat{\rho_0}}$ as well as
  the push-down to the quotient is the identity.  On the other hand,
  $\rho_1$ retracts the fiber $F$ to $\I{p}{q}$. Thus,
  $\smash{\hat{\rho_1}}$ retracts $K \times F$ to $K \times
  \{\I{p}{q}\}$.
  Using the isomorphism $K \times_L F \cong \mathcal{H}_{(p,q)}$, it
  follows that $\smash{\tilde{\rho}}$, retracts $\mathcal{H}_{(p,q)}$ to
  $\Sigma = K.\I{p}{q}$.

  It remains to prove that $\smash{\tilde{\rho}}_t$ is $T$-equivariant for
  each $t$:
  \begin{align}
    \label{EquationRetractionEquivariance}
    \tilde{\rho}_t(\overline{H}) = \overline{\tilde{\rho}_t(H)}.
  \end{align}
  A short computation shows that the isomorphism $\Psi$
  (\ref{EquivariantFiberBundleIsoPsi}) is $T$-equivariant with respect
  to the $T$-action on $K \times_L F$ given by
  \begin{align*}
    T\left(\left[(k,H)\right]\right) = \left[\left(\overline{k},\overline{H}\right)\right].
  \end{align*}
  Therefore, in order to prove (\ref{EquationRetractionEquivariance})
  we need to show that the induced map $I \times K \times_L F \to K
  \times_L F$ is $T$-equivariant. By definition this boils down to
  showing that
  \begin{align*}
    \left[\left(\overline{k},\rho_t\left(\overline{H}\right)\right)\right] = \left[\left(\overline{k},\overline{\rho_t(H)}\right)\right],
  \end{align*}
  which is a direct consequence of the definition of $\rho_t$ as in
  (\ref{EquationRetractRho}). Summarizing the above we have seen that
  there exists a strong deformation retract $\smash{\tilde{\rho}}$ from
  $\mathcal{H}_{(p,q)}$ to the orbit $K.\I{p}{q}$ which is
  $T$-equviariant.
\end{proof}
\begin{figure}[h]
  \centering
  \begin{tikzpicture}[scale=0.9]
    \foreach \i in {5,6,7,8,10,11,12,13,14}
      \draw[color=gray!60] (0,0) ellipse (40mm-\i*2mm and 30mm-\i*2mm);
    \draw[color=red,thick] (0,0) ellipse (40mm-9*2mm and 30mm-9*2mm);
    \draw[color=blue,thick] (2mm,0) -- (32mm,0mm);
    \draw[font=\footnotesize] (-40mm,8mm) node[left] {$U(n).\I{p}{q}$};
    \draw[color=black,->] (-40mm,8mm) -- (-18mm,8mm);
    \draw[font=\footnotesize] (-40mm,0mm) node[left] {$F$};
    \draw[color=black,->] (-40mm,0mm) -- (0mm,0mm);
    \draw[thick,decorate,decoration={brace,mirror,amplitude=10pt}] [yshift=-0mm](36mm,-20mm) -- (36mm,20mm) node[yshift=-20mm,xshift=5mm,right] {$U(n)$-orbits};
    \draw[color=blue,->] (2mm,0mm) -- (16mm,0mm);
    \draw[color=blue,->] (32mm,0mm) -- (26mm,0mm);
    \end{tikzpicture}
  \caption{The $U(n)$-orbits in $\mathcal{H}_{(p,q)}$.}
  \label{fig:UnOrbitsInHpq}
\end{figure}
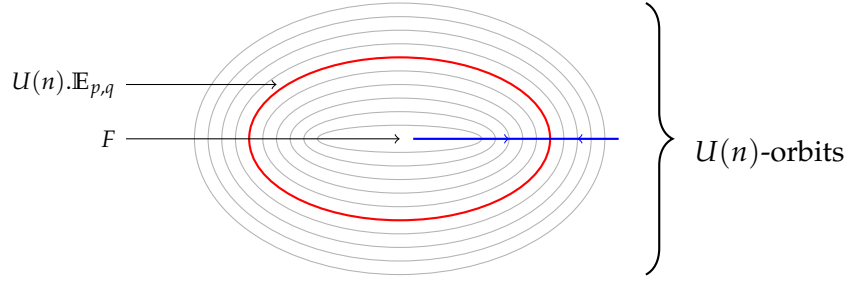

The $U(n)$-orbit of $\I{p}{q}$ can be equivariantly and
diffeomorphically identified with the complex Grassmann manifold
$\Gr_p(\C^n)$.
This reduces the problem of describing the sets
$\smash{[X,\mathcal{H}_{(p,q)}]_G}$ to the problem of describing the sets
$\smash{[X,\Gr_p(\C^n)]_G}$. As a first step towards the study of equivariant
maps to $\Gr_p(\C^n)$ we describe $(\Gr_p(\C^n))_{\mathbb{R}}$, the
space of real points with respect to the real structure $T$ on
$\Gr_p(\C^n)$:
\begin{remark}
  The space $(\Gr_p(\C^n))_{\mathbb{R}}$ of real points is
  diffeomorphic to the real Grassmannian $\Gr_p(\mathbb{R}^n)$.
\end{remark}

\begin{proof}
  Given any $T$-stable subspace $V$ in $\C^n$, then $T$ defines a
  decomposition $V = V_\R \oplus i V_\R$ where $V_\R = \Fix T$ is a
  subspace of $\R^n$. Note that the group $O(n,\R)$ acts transitively
  on $\Gr_p(\R^n)$. Therefore, for any two $T$-stable $p$-dimensional
  subvector spaces $V, V' \subset \C^n$, there exists an orthogonal
  transformation $g$ mapping the $V_\R$ to $V'_\R$. Such a
  transformation extends uniquely to a unitary transformation of
  $\C^n$ which sends $V$ to $V'$. Now, a computation shows that the
  stabilizer of the base point
  $V_0 = \smash{\left<e_1,\ldots,e_p\right>}$ in $O(n,R)$ consists of
  the matrices of the form
  \begin{align*}
    g =
    \begin{pmatrix}
      g_1 & 0 \\
      0 & g_2
    \end{pmatrix},
  \end{align*}
  where $g_1 \in O(p,\R)$ and $g_2 \in O(q,\R)$. It follows that
  \begin{equation*}
    (\Gr_p(\C^n))_\R \cong \frac{O(n,\R)}{O(p,\R) \times O(q,\R)} \cong \Gr_p(\R^n).\qedhere
  \end{equation*}
\end{proof}

We now proceed with the classification of equivariant maps to the
space of non-singular $n \times n$ hermitian matrices:
\begin{align*}
  H\colon X \to \mathcal{H}^*_n.
\end{align*}
As shown earlier, such a map has a well-defined signature $(p,q)$
(with $p+q = n$) specifying that $H(x)$ has exactly $p$ positive
eigenvalues and $q$ negative eigenvalues for every $x \in X$. The
classification for the image space $\mathcal{H}_n$ will be achieved in
the following way:
\begin{enumerate}
\item First we handle the classification of maps $X \to \mathcal{H}_3$
  seperately. In this situation the relevant image space can be
  reduced to the Grassmannian $\Gr_1(\C^3) \cong \mathbb{P}_2$. The
  classification with respect to this image space works by making a
  reduction to the case of maps to $\Proj_1$.
\item Then, for a general Grassmannian $\Gr_p(\mathbb{C}^n)$ we apply
  a reduction method iteratively until we have reduced the situation
  to a copy of $\mathbb{P}_1$ embedded in $\Gr_p(\C^n)$ as a Schubert
  variety. The seperate discussion of $\Proj_2$ will act as a guiding
  example and provide useful statements for the final step of the
  iterative reduction.
\end{enumerate}

For the following we need a notion of \emph{total degree}, as the
standard notion of degree (i.\,e. the Brouwer degree) is not
applicable here, since $X$ and a general Grassmann manifold
$\Gr_p(\C^n)$ have different dimensions. But because of the convenient
geometrical properties of Grassmannians (see
proposition~\ref{TopologyGrassmannians}) we in fact do not have to
change much: The second homology group is still infinite cyclic, which
allows us to generalize the notion of total degree. For this, let
$f\colon X \to \Gr_p(\C^n)$ be a $G$-map. The homology group
$H_2(X,\mathbb{Z})$ is generated by the fundamental class $[X]$. The
homology group $H_2(\Gr_p(\C^n),\Z)$ is generated by $[\mathcal{S}]$,
where $\mathcal{S}$ denotes the Schubert variety
\begin{align}
  \label{DefOfSchubertVarietyC}
  \mathcal{S} = \left\{E \in \Gr_p(\C^n)\colon \C^{p-1} \subset E \subset \C^{p+1}\right\}.
\end{align}
Here, $\C^\ell$ is regarded as being embedded in $\C^n$, for $\ell <
n$, via
\begin{align*}
 (z_1,\ldots,z_\ell) \mapsto (z_1,\ldots,z_\ell,\underbrace{0,\ldots,0}_{n-\ell}).
\end{align*}
In the terminology of e.\,g. \cite{GriffithsHarris},
$\mathcal{S}$ is the Schubert variety\label{SchubertVarietyDiscussion}
\begin{align*}
  \mathcal{S} = \left\{E \in \Gr_p(\C^k)\colon \dim(E \cap V_{k-p+i-a_i}) \geq i \;\text{for all $i=1,\ldots,p$}\right\}
\end{align*}
defined by the sequence
\begin{align*}
  a_i =
  \begin{cases}
    n - p & \;\text{for $1 \leq i < p-1$} \\
    n - p - 1 & \;\text{for $p - 1 \leq i \leq p$}.
  \end{cases}
\end{align*}
After having fixed the standard flag $V_1 \subset V_2 \subset \ldots
\subset V_n$, where $V_j$ is the vector space generated by the
standard basis vectors $e_1,\ldots,e_j$, $\mathcal{S}$ is the unique
one-dimensional Schubert variety in $\Gr_p(\C^n)$. It is biholomorphic
to $\Proj_1$ and comes equippd with its canonical orientation as a
complex manifold. As it has been discussed on
p.~\pageref{P1S2OrientationDiscussion}, this space can be
equivariantly identified with the $2$-sphere. Now we can make the
following
\begin{definition}
  \label{TotalDegreeGr}
  Let $f\colon X \to \Gr_p(\C^n)$ be a $G$-map. Then its induced map
  on the second homology is of the form
  \begin{align*}
    f_*\colon H_2(X,\mathbb{Z}) &\to H_2(\Gr_p(\C^n),\mathbb{Z})\\
    [X] &\mapsto d[\mathcal{S}].
  \end{align*}
  We define the \emph{total degree} of $f$ to be this integer $d$ and
  write $\deg f$ for the total degree of $f$.
\end{definition}
Clearly, the total degree defined this way is a homotopy invariant of
maps $X \to \Gr_p(\C^n)$. We also need a slight generalization of
definition~\ref{DefinitionFixpointDegree} to account for the
fundamental group of $(\Gr_p(\C^n))_{\R}$ not being infinite cyclic
unless $p = 1$ and $n=2$:
\begin{definition}
  \label{DefinitionFixpointSignature} (Fixed point signature) Let
  $f\colon M \to Y$ be a $G$-map between $G$-manifolds. Assume that
  the fixed point set $Y^G \subset Y$ is path-connected, $\pi_1(Y^G) =
  C_2$ and that the fixed point set $M^G$ is the disjoint union
  \begin{align*}
    M^G = \bigcup_{j=0,\ldots,k} K_j
  \end{align*}
  of circles $K_j$. Then we define the \emph{fixed point signature} of
  $f$ to be the $k$-tuple $(m_0,\ldots,m_k)$ where
  \begin{align*}
    m_j =
    \begin{cases}
      0 & \;\text{if the loop $\Restr{f}{K_j}\colon K_j \to Y^G$ is contractible in $Y^G$}\\
      1 & \;\text{else}.
    \end{cases}
  \end{align*}
\end{definition}

In the following let $n > 2$, i.\,e. we exclude the case
$\Gr_p(\C^n) = \Proj_1$, which has already been dealt with in
section~\ref{SectionN=2}. For the type I involution on $X$ we make the
following definition:
\begin{definition} (Degree triple map)
  For $n > 2$, we define the \emph{degree triple map} to be the map
  \begin{align*}
    \mathcal{T}\colon \mathcal{M}_G(X,\Gr_p(\C^n)) &\to \{0,1\} \times \Z \times \{0,1\}\\
    f &\mapsto \Triple{m_0}{d}{m_1},
  \end{align*}
  where $\Bideg{m_0}{m_1}$ is the fixed point signature of $f$ with
  respect to the circles $C_0, C_1$ and $d$ is the total degree of
  $f$. We call a given triple \emph{realizable} if it is contained in
  the image $\Im(\mathcal{T})$.  For a map
  \begin{align*}
   f\colon (Z,C) \to \left(\Gr_p(\C^n),\Gr_p(\C^n)_\R\right)
  \end{align*}
  we define its \emph{degree triple} to be the degree triple of the
  equivariant extension $\smash{\hat{f}}\colon X \to \Gr_p(\C^n)$.
\end{definition}
  
For the type II involution on $X$ we define:
\begin{definition} (Degree pair map)
  For $n > 2$, we define the \emph{degree pair map} to be the map
  \begin{align*}
    \mathcal{P}\colon \mathcal{M}_G(X,\Gr_p(\C^n)) &\to \Z \times \{0,1\}\\
    f &\mapsto \Pair{m}{d},
  \end{align*}
  where $(m)$ is the fixed point signature of $f$ and $d$ is the total
  degree of $f$. We call a given pair \emph{realizable} if it is
  contained in the image $\Im(\mathcal{P})$.  For a map
  \begin{align*}
    f\colon (Z,C) \to \left(\Gr_p(\C^n),\Gr_p(\C^n)_\R\right)
  \end{align*}
  we define its \emph{degree pair} to be the degree pair of the
  equivariant extension $\smash{\hat{f}}\colon X \to \Gr_p(\C^n)$.
\end{definition}

The degree triple resp. the degree pair is clearly an invariant of
equivariant homotopy.

\subsection{Maps to $\mathcal{H}_3$}

In this section we study the equivariant homotopy of maps $X \to
\mathcal{H}_3^*$. The connected components of $\mathcal{H}_3^*$ are
parameterized by the eigenvalue signatures $(3,0)$, $(2,1)$, $(1,2)$
and $(0,3)$. The group actions are exactly the same as in the first
chapter: we can regard $X$ as being defined by the standard square
torus, equipped with either the type I or the type II real structure
and the involution on $\mathcal{H}_3$ is given by conjugation. It
turns out that the classification of maps in the case where the image
space is $\mathcal{H}_3^*$ is not fundamentally different from the
classification where it is $\mathcal{H}_2^*$. As before we start with
a reduction procedure.

Let $H\colon X \to \mathcal{H}_3^*$ be a map with the signature
$(p,q)$. Hence, $H$ can be regarded as a map $X \to
\mathcal{H}_{(p,q)}$. In the concrete case of maps to $\mathcal{H}_3$
with mixed signature we can now conclude with
proposition~\ref{GrassmannianStrongDeformationRetract} that
$\mathcal{H}_{(2,1)}$ (resp. $\mathcal{H}_{(1,2)}$) can be
equivariantly retracted to the orbit $U(3).\I{2}{1}$
(resp. $U(3).\I{1}{2}$). These orbits are equivariantly and
diffeomorphically identifiable with $\Gr_2(\mathbb{C}^3)$
(resp. $\Gr_1(\mathbb{C}^3)$). Since these two Grassmann manifolds are
by remark~\ref{GrassmannianIdentificationEquivariant} equivalent as
$G$-spaces, we can assume without loss of generality that we are
dealing with $G$-maps of the form
\begin{align*}
  X \to \Gr_1(\mathbb{C}^3) \cong \Proj_2,
\end{align*}
where the involution $T$ on $\Proj_2$ is given by standard complex
conjugation:
\begin{align*}
  T\colon [z_0:z_1:z_2] \mapsto \left[\overline{z_0}:\overline{z_1}:\overline{z_2}\right].
\end{align*}

\subsubsection{Maps to $\Proj_2$}
\label{SectionClassificationP2}

For the following, let $p_0$ be the base point $[0:0:1] \in
\mathbb{P}_2$ and regard $\Proj_1$ as being embedded in $\Proj_2$ as
\begin{align*}
  [z_0:z_1] \mapsto [z_0:z_1:0].
\end{align*}

\begin{lemma}
  \label{RetractFromP2toP1}
  There exists an equivariant strong deformation retract $\rho$ from
  $\mathbb{P}_2\setminus\{p_0\}$ to $\mathbb{P}_1 \subset
  \mathbb{P}_2$.
\end{lemma}

\begin{proof}
  Define $\rho$ as follows:
  \begin{align*}
    \rho\colon I \times \mathbb{P}_2\setminus\{p_0\} &\to \mathbb{P}_2\setminus\{p_0\}\\
    \left(t,[z_0:z_1:z_2]\right) &\mapsto [z_0:z_1:(1-t) z_2].
  \end{align*}
  For $t=0$ this is the identity of $\mathbb{P}_2\setminus\{p_0\}$,
  for $t=1$, this is the map
  \begin{align*}
    [z_0:z_1:z_2] \mapsto [z_0:z_1:0].
  \end{align*}
  This is well-defined, since $\rho$ is only defined on the complement
  of the point $p_0$.  This construction is equivariant, because the
  multiplication with real numbers commutes with complex conjugation.
\end{proof}

\begin{figure}[h]
  \centering
  \begin{tikzpicture}[scale=0.7]
    \draw (0,{sqrt(3)*2}) coordinate (p0) node[color=red,above] {$p_0$};
    \draw [color=blue] 
    (-2,0) -- node[below] {$\Proj_1$} (2,0);
    \draw (-2,0) -- (0,{sqrt(3)*2})
    (0,{sqrt(3)*2}) -- (2,0);
    \draw
    (1.5,0) -- (p0)
    (1,0) -- (p0)
    (0.5,0) -- (p0)
    (0,0) -- (p0)
    (-0.5,0) -- (p0)
    (-1.0,0) -- (p0)
    (-1.5,0) -- (p0);
    \draw[thick,decorate,decoration={brace,amplitude=10pt}] [xshift=-2mm](-2,-0.1) -- node[yshift=0mm,xshift=-4mm,left] {$\Proj_2$} (-2,{sqrt(3)*2+0.1});
    \draw[red,fill=red] (p0) circle (2pt);

    \draw (5,{sqrt(3)*2}) coordinate (p1) ;
    \draw [color=blue] 
    (3,0) -- node[below] {$\Proj_1$} (7,0);
    \draw (3,0) -- (5,{sqrt(3)*2})
    (5,{sqrt(3)*2}) -- (7,0);
    \draw
    (6.5,0) -- (p1)
    (6,0) -- (p1)
    (5.5,0) -- (p1)
    (5,0) -- (p1)
    (4.5,0) -- (p1)
    (4.0,0) -- (p1)
    (3.5,0) -- (p1);
    \fill [white] (3,{sqrt(3)*2-1.5}) rectangle (7,{sqrt(3)*2+0.1});

    \draw[thick,->] (2.2,1) -- (2.8,1);
    \draw[thick,->] (7.2,1) -- (7.8,1);
    \draw [color=blue] 
    (8,0) -- node[below] {$\Proj_1$} (12,0);
  \end{tikzpicture}
  \caption{$\Proj_2\smallsetminus\{p_0\}$ retracts to $\Proj_1 \subset \Proj_2$.}
  \label{P2VectorBundleOverP1}
\end{figure}

The strong deformation retract $\rho$ (see
figure~\ref{P2VectorBundleOverP1} for a depiction) constructed in the
previous lemma also defines a strong deformation retract of the fixed
point set $(\Proj_2)_\R$ to the fixed point set $(\Proj_1)_\R$. For
understanding what happens to fixed point signatures of maps
$X \to \Proj_2$ during such a deformation retract we make the
following remark:
\begin{remark}
  \label{NonTrivialLoopInRPn}
  Assume $n > 1$. Let $\gamma\colon I \to \RProj_1$ be the loop which
  wraps around $\RProj_1$ once. Let $\iota_n$ be the following embedding:
  \begin{align*}
    \iota_n\colon \RProj_1 &\to \RProj_n\\
    [x_0:x_1] &\mapsto [x_0:x_1:0:\ldots:0].
  \end{align*}
  Then $\iota_n \circ \gamma$ defines the (up to homotopy) unique
  non-trivial loop in $\RProj_n$.
\end{remark}

\begin{proof}
  To simplify the notation set $\iota = \iota_n$.
  We have the following commutative diagram
  \[
  \begin{xy}
    \xymatrix{
      S^1 \ar@{^{(}->}[r]^{\hat{\iota}} \ar[d] & S^n \ar[d]\\
      \mathbb{RP}_1 \ar@{^{(}->}[r]_\iota & \mathbb{RP}_n
    }
  \end{xy}
  \]
  Here, $\smash{\hat{\iota}}$ can be regarded as the restriction of
  the embedding $\mathbb{R}^2 \hookrightarrow \mathbb{R}^{n+1}$ to
  $S^1 \subset \mathbb{R}^2$. Denote the loop which wraps around
  $\mathbb{RP}_1$ once by $\gamma\colon I \to \mathbb{RP}_1$. It can
  be lifted to a path $\smash{\hat{\gamma}}\colon I \to S^1$, which
  makes a one-half loop in $S^1$. Then, by commutativity of the
  diagram, it follows that
  $\smash{\widehat{\iota \circ \gamma}} := \smash{\hat{\iota}} \circ
  \smash{\hat{\gamma}}$
  is a lift of $\iota \circ \gamma$, the latter being a representant
  of the homotopy class $\iota_*[\gamma] \in \pi_1(\RProj_n)$. It is
  well known (see e.\,g. \cite[p.~74]{Hatcher}) that $\pi_1(\RProj_n)$
  (for $n > 1$) is generated by the projection to $\RProj_n$ of a
  curve in the universal covering space $S^n$ which connects two
  antipodal points.  Thus it is enough to show that
  $\widehat{\iota \circ \gamma}$ is not a closed loop in $S^n$. But
  this is clear, since $\smash{\hat{\gamma}}$ is not a closed loop and
  $S^1 \hookrightarrow S^n$ is just an embedding.
\end{proof}

As a consequence of the previous remark we can formulate:
\begin{remark}
  \label{RPnFundamentalGroupIsomorphism}
  Assume $n > 1$ and let $\iota$ be the embedding
  \begin{align*}
    \iota\colon \RProj_n &\to \RProj_{n+1}\\
    [x_0:\ldots:x_n] &\mapsto [x_0:\ldots:x_n:0].
  \end{align*}
  Then the induced map $\iota_*\colon \pi_1(\RProj_n) \to
  \pi_1(\RProj_{n+1})$ is an isomorphism.
\end{remark}

\begin{proof}
  Since both fundamental groups are cyclic of order 2 it suffices to
  show that $\iota_*$ maps the non-trivial homotopy class $[\gamma]
  \in \pi_1(\RProj_n)$ to the non-trivial homotopy class in
  $\pi_1(\RProj_{n+1})$. Let $\alpha$ be the non-trivial loop in
  $\RProj_n$ which is defined in terms of the embedding of $\RProj_1
  \hookrightarrow \RProj_n$ (see
  remark~\ref{NonTrivialLoopInRPn}). Then we have $[\gamma] =
  [\alpha]$. The homotopy class $\iota_*([\gamma])$ can thus be
  regarded as the homotopy class of the loop $\iota \circ \alpha$. But
  this is just the non-trivial loop constructed in the previous remark
  for $\RProj_{n+1}$. Therefore, $\iota \circ \alpha$ is non-trivial
  and the statement is proven.
\end{proof}

Note that there is a canonical embedding of $S^1 \cong \mathbb{RP}_1
\hookrightarrow \mathbb{RP}_2$ as the line at infinity:
\begin{align*}
  [x_1:x_2] \mapsto [x_1:x_2:0]
\end{align*}
Fix a point $p_0$ in $\mathbb{RP}_1$. We now define two loops: Let
$\gamma_0$ be the loop which is constant at $p_0$ and let $\gamma_1$
be a fixed loop which wraps around $\mathbb{RP}_1 \subset
\mathbb{RP}_2$ once.  By remark~\ref{NonTrivialLoopInRPn} this defines
the non-trivial element in $\pi_1(\mathbb{RP}_2)$.

\begin{definition}
  \label{DefinitionFixpointNormalizationP2}
  A $G$-map $f\colon X \to \Proj_2$ is \emph{fixed point normalized} if
  $\restr{f}{C_j}$ is either $\gamma_0$ or $\gamma_1$ for all fixed point
  circles $C_j$ in $X$.
\end{definition}

\begin{remark}
  \label{FixpointNormalizationP2}
  Let $X$ be equipped with the type I or the type II involution. Then
  every $G$-map $f\colon X \to \Proj_2$ can be fixed point normalized.
\end{remark}
\begin{proof}
  Using the equivariant version of the homotopy extension property of
  the pair $(X,C)$ (see e.\,g. corollary~\ref{G-HEP}) we can prescribe
  a homotopy on each of the the fixed point circles of the torus which
  deforms the corresponding loop to either $\gamma_0$ or $\gamma_1$,
  depending on the fixed point signature of the map.
\end{proof}

Thus, after this normalization we can assume that maps with the same
fixed point signature agree along the boundary circles.

\begin{lemma}
  \label{ReductionP2toP1}
  Let $f\colon X \to \mathbb{P}_2$ be a fixed point normalized $G$-map
  for the type I or the type II involution.  Then:
  \begin{enumerate}[(i)]
  \item The map $f$ is equivariantly homotopic to a map $f'$ which has
    its image contained in $\mathbb{P}_1 \subset \mathbb{P}_2$.
  \item The degree of $f$ (definition~\ref{TotalDegreeGr}) agrees with
    $\deg f'$, where $f'$ is regarded as a map $X \to \mathbb{P}_1$.
  \end{enumerate}
\end{lemma}

\begin{proof}
  Regarding (i): Let $\smash{\widehat{C}}$ be the union of the
  (finitely many) fixed point circles.  Because of dimension reasons
  there exists a point $p \in \mathbb{RP}_2$ which is not contained in
  the image of $\smash{\restr{f}{\widehat{C}}}$. Note that $\SOR{3}$ acts
  transitively on $\mathbb{RP}_2$. Thus, let $g_1 \in \SOR{3}$ be the
  transformation in $\SOR{3}$ which sends $p$ to $p_0 :=
  [0:0:1]$.
  This transformation $g_1$ can also be regarded as a transformation
  of $\mathbb{P}_2$, since $\SOR{3} \subset \SU(3)$. Also, since
  $\SOR{3}$ is path-connected we can find a path $g_t$ from the
  identity in $\SOR{3}$ to $g_1$. The composition $g_t \circ f$ is
  then a homotopy of $f$ to a map $\smash{\tilde{f}}$ such that
  $p_0 \not\in \Im \smash{\tilde{f}}$.
  This homotopy is also equivariant. Thus, without loss of generality
  we can assume that, at least after a $G$-homotopy, a given $G$-map
  $f\colon X \to \mathbb{P}_2$ does not have $p_0$ in its image. Let
  us assume that the given map $f$ has this property. Then we can
  compose $f$ with the equivariant strong deformation retract from
  remark~\ref{RetractFromP2toP1} and therefore obtain an equivariant
  homotopy from $f$ to a map $f'$ such that $\Im f' \subset
  \mathbb{P}_1 \subset \mathbb{P}_2$.

  Regarding (ii): The degree $\deg f$ is the number $d$ such that
  \begin{align*}
    f_*\colon H_2(X,\mathbb{Z}) \to H_2(\mathbb{P}_2,\mathbb{Z})
  \end{align*}
  is the map
  \begin{align*}
    [X] \mapsto d[\Proj_1].
  \end{align*}
  while the degree $\deg f'$ is the number $d'$ such that
  \begin{align*}
    f'_*\colon H_2(X,\mathbb{Z}) \to H_2(\mathbb{P}_1,\mathbb{Z})
  \end{align*}
  is the map
  \begin{align*}
    [X] \mapsto d'[\Proj_1].
  \end{align*}
  The inclusion $\iota\colon \mathbb{P}_1 \hookrightarrow
  \mathbb{P}_2$ induces an isomorphism in the second homology
  group. Thus we can also regard $f_*$ as taking on values in
  $H_2(\mathbb{P}_2,\mathbb{Z})$. But since $f$ and $\iota \circ f'$
  are homotopic, their induced maps in homology must agree. This means
  $d = d'$.
\end{proof}

\paragraph{Type I}

In this paragraph we are assuming the involution on $X$ to be of type
I.

\begin{remark}
  \label{ReductionP2toP1Class1}
  Let $f\colon X \to \mathbb{P}_2$ be a fixed point normalized
  $G$-map. Then the map $f$ has fixed point signature $\Pair{m_0}{m_1}$
  iff $f'$ (from lemma~\ref{ReductionP2toP1}) has fixed point degree
  $\Bideg{m_0}{m_1}$.
\end{remark}

\begin{proof}
  Let us first assume that $\restr{f}{C_j}\colon C_j \to
  \mathbb{RP}_2$ is trivial, i.\,e. $m_j = 0$. Since $f$ is fixed point
  normalized, we know that $\restr{f}{C_j}$ is the constant
  map. Therefore, after making the deformation retract to $\RProj_1
  \subset \mathbb{P}_1$ it is still constant, hence has degree
  zero. On the other hand, if a map $S^1 \to \RProj_1 \subset
  \mathbb{P}_1$ has degree zero, then it is already null-homotopic as
  a map to $\RProj_1$. Therefore it will also be null-homotopic as a
  map to $\RProj_2 \supset \RProj_1$.

  Let us now assume that $\restr{f}{C_j}\colon C_j \to \mathbb{RP}_2$
  is non-trivial, i.\,e. $m_j = 1$. In this case $\restr{f}{C_j}$ is
  the map which wraps around $\mathbb{RP}_1$ once. In particular, this
  loop will be fixed when composing with the strong deformation
  retract to $\mathbb{P}_1$. Therefore, after the retraction, the
  resulting map will still have degree one along the boundary circle
  $C_j$. The reverse direction follows from
  remark~\ref{NonTrivialLoopInRPn}, where we discussed the loop which
  wraps around $\RProj_1$ once; using the canonical embedding
  $\mathbb{RP}_1 \hookrightarrow \mathbb{RP}_2$, this loops defines a
  non-trivial element in $\pi_1(\mathbb{RP}_2)$.
\end{proof}

\begin{theorem}
  \label{TheoremClassificationMapsToP2Class1}
  The $G$-homotopy class of a map $f \in \mathcal{M}_G(X,\Proj_2)$ is
  uniquely determined by its degree triple $\mathcal{T}(f)$. The image
  $\Im(\mathcal{T})$ of the degree triple map consists of those
  triples $\Triple{m_0}{d}{m_1}$ (with $m_0, m_1 \in \{0,1\}$)
  satisfying
  \begin{align*}
    d \equiv m_0 + m_1 \mod 2.
  \end{align*}
\end{theorem}

\begin{proof}
  First we show that two maps with the same degree triple are
  equivariantly homotopic. Thus, let $f$ and $g$ be two $G$-maps $X
  \to \mathbb{P}_2$, both having the triple $\Triple{m_0}{d}{m_1}$.
  Using lemma~\ref{ReductionP2toP1} (i) we can assume that $f$ and $g$
  have their images contained in $\mathbb{P}_1 \subset
  \mathbb{P}_2$. Applying lemma~\ref{ReductionP2toP1} (ii) and
  lemma~\ref{ReductionP2toP1Class1} it then follows that the degree
  triple of $f$ and $g$, regarded as a map $X \to \mathbb{P}_1$ is
  again $\Triple{m_0}{d}{m_1}$. Hence, by
  theorem~\ref{Classification1}, these maps are equivariantly
  homotopic.

  Regarding the image $\Im(\mathcal{T})$: Let us first assume that
  $\Triple{m_0}{d}{m_1}$ is in the image. Then, by definition, there
  exists a map $f\colon X \to \mathbb{P}_2$ with this
  triple. Lemma~\ref{ReductionP2toP1} then implies the existence of a
  map $f'\colon X \to \mathbb{P}_1$ with the same triple. Hence, by
  theorem~\ref{Classification1}, $d \equiv m_0 + m_1 \mod 2$. On the
  other hand, given a triple $\Triple{m_0}{d}{m_1}$ with $m_0,m_1 \in
  \{0,1\}$ and $d \equiv m_0 + m_1 \mod 2$, there exists a map
  $f\colon X \to S^2 \cong \mathbb{P}_1$ with this triple. Composing
  with the embedding $\mathbb{P}_1 \hookrightarrow \mathbb{P}_2$ we
  have produced a map $X \to \mathbb{P}_2$ which, again by
  lemma~\ref{ReductionP2toP1}, has the same triple
  $\Triple{m_0}{d}{m_1}$.
\end{proof}

\paragraph{Type II}

In this paragraph we are assuming the involution on $X$ to be of type
II.

\begin{remark}
  \label{ReductionP2toP1Class2}
  Let $f\colon X \to \mathbb{P}_2$ be a fixed point normalized
  $G$-map. Then the map $f$ has the fixed point signature $m$ iff $f'$
  (from lemma~\ref{ReductionP2toP1}) has the fixed point degree $m$.
\end{remark}

\begin{proof}
  The proof is exactly the same as in
  lemma~\ref{ReductionP2toP1Class1}. One only has consider the single
  circle $C$ in $X$ instead of two circles $C_0$ and $C_1$.
\end{proof}

\begin{theorem}
  \label{TheoremClassificationMapsToP2Class2}
  The $G$-homotopy class of a map $f \in \mathcal{M}_G(X,\Proj_2)$ is
  uniquely determined by its degree pair $\mathcal{P}(f)$. The image
  $\Im(\mathcal{P})$ of the degree pair map consists of those pairs
  $\Pair{m}{d}$ (with $m \in \{0,1\}$) satisfying
  \begin{align*}
    d \equiv m \mod 2.
  \end{align*}
\end{theorem}

\begin{proof}
  The proof works completely analogously to the proof of
  theorem~\ref{TheoremClassificationMapsToP2Class1}. For the first
  statement it suffices to show that two maps with the same degree
  pair are equivariantly homotopic. Thus, let $f$ and $g$ be two
  $G$-maps $X \to \mathbb{P}_2$, both having the pair
  $\Pair{m}{d}$. Using lemma~\ref{ReductionP2toP1} (i) we can assume
  that $f$ and $g$ have their images contained in $\mathbb{P}_1 \subset
  \mathbb{P}_2$. Applying lemma~\ref{ReductionP2toP1} (ii) and
  lemma~\ref{ReductionP2toP1Class2} it then follows that the degree
  pair of $f$ and $g$, regarded as a map $X \to \mathbb{P}_1$ is again
  $\Pair{m}{d}$. Hence, by theorem~\ref{Classification2}, these maps
  are equivariantly homotopic, which proves the statement.

  Regarding the image $\Im(\mathcal{P})$: Let us assume that
  $\Pair{m}{d}$ is in the image. By definition there exists a map
  $f\colon X \to \mathbb{P}_2$ with this
  pair. Lemma~\ref{ReductionP2toP1} then implies the existence of a
  map $f'\colon X \to \mathbb{P}_1$ with the same pair. Hence, by
  theorem~\ref{Classification2}, $d \equiv m \mod 2$ is a necessary
  condition for a pair to be contained in $\Im(\mathcal{P})$. On the
  other hand, given a pair $\Pair{m}{d}$ with $m \in \{0,1\}$ and $d
  \equiv m \mod 2$, there exists a map $f\colon X \to S^2 \cong
  \mathbb{P}_1$ with this pair. Composing with the embedding
  $\mathbb{P}_1 \hookrightarrow \mathbb{P}_2$ we have produced a map
  $X \to \mathbb{P}_2$ which, again by lemma~\ref{ReductionP2toP1},
  has the same pair $\Pair{m}{d}$. Therefore, $\Pair{m}{d}$ is
  contained in the image $\Im(\mathcal{P})$.
\end{proof}

\subsubsection{Classification of Maps to $\mathcal{H}_3^*$}

So far the degree triple map (resp. the degree pair map) has only been
defined for $G$-maps $X \to \Proj_2$. But after fixing an
identification of the orbit $U(3).\I{p}{q} \subset
\mathcal{H}_{(p,q)}$, which has been shown to be an equivariant strong
deformation retract, with $\Proj_2$, the degree triple map (resp. the
degree pair map) is also defined on all mapping spaces
$\mathcal{M}_G(X,\mathcal{H}_{(p,q)})$ with $p+q=3$. This remark then
allows us to state and prove the main result of this section:

\begin{theorem}
  \label{HamiltonianClassificationRank3}
  Let $X$ be a torus equipped with either the type I or the type II
  involution. Then:
  \begin{enumerate}[(i)]
  \item The sets $[X,\mathcal{H}_{(3,0)}]_G$ and
    $[X,\mathcal{H}_{(0,3)}]_G$ are trivial.
  \item Two $G$-maps $X \to \mathcal{H}_{(p,q)}$ (with $0<p,q<3$) are
    $G$-homotopic iff their degree triples (type I) resp. their degree
    pairs (type II) agree.
  \item The realizable degree triples $\Triple{m_0}{d}{m_1}$ (type I)
    resp. degree pairs $\Pair{m}{d}$ (type II) are exactly those
    which satisfy
    \begin{align*}
      d \equiv m_0 + m_1 \mod 2 \;\text{resp.}\;d \equiv m \mod 2.
    \end{align*}
  \end{enumerate}
\end{theorem}

\begin{proof}
  The $\sig = (3,0)$ resp. the $\sig = (0,3)$ case is handled by
  remark~\ref{RemarkDefiniteComponentsRetractable}, which proves that
  maps $X \to \mathcal{H}_{(2,0)}$ (resp. $X \to \mathcal{H}_{(0,2)}$)
  are equivariantly retractable to the map which is constant the
  identity (resp. minus identity). In the case of the signature being
  $(2,1)$ (resp. $(1,2)$) the image space has the projective space
  $\Proj_2$ as an equivariant strong deformation retract. In this case
  theorem~\ref{TheoremClassificationMapsToP2Class1} (for type I) and
  theorem~\ref{TheoremClassificationMapsToP2Class2} (for type II)
  complete the proof.
\end{proof}

\begin{remark}
  \label{HamiltonianMapRealizationRank3}
  Clearly, the case $\sig = (3,0)$ (resp. $\sig = (0,3)$) is trivially
  realized by the map, which is constant the identity (resp. minus
  identity). Representants for the homotopy classes for $\sig = (2,1)$
  and $\sig = (1,2)$ can be constructed as follows: Let
  $f\colon X \to S^2$ be a $G$-map. Using the orientation-preserving
  diffeomorphism $S^2 \xrightarrow{\;\sim\;} \Proj_1$, we can regard
  $f$ as a map to $\Proj_1$. Composing this map with the embedding
  $\mathbb{P}_1 \hookrightarrow \mathbb{P}_2$ yields a $G$-map
  $\smash{\tilde{f}}\colon X \to \Proj_2$. Finally, we can compose
  $\smash{\tilde{f}}$ with one of the two embeddings
  $\iota_{(2,1)},\iota_{(1,2)}\colon \mathbb{P}_2 \hookrightarrow
  \mathcal{H}_3^*$,
  which embed $\mathbb{P}_2$ in the respective component of
  $\mathcal{H}_3^*$ associated to the signature $(2,1)$ or $(1,2)$.
\end{remark}

\subsection{Iterative Retraction Method}

In this section we describe a reduction procedure which can be used
iteratively until we arrive in a known situation, i.\,e.
$\Gr_1(\C^2) \cong \Proj_1$. We call this procedure \emph{iterative
  retraction} of Grassmann manifolds. Recall
section~\ref{SectionClassificationP2}, in particular
lemma~\ref{ReductionP2toP1}: There we start with a $G$-map
$f\colon X \to \Proj_2$ and remove a certain subset $S$ (in this case,
$S$ is a point) from $\Proj_2$ which -- at least after a $G$-homotopy
-- is not contained in the image $\Im(f)$. The resulting space
$\Proj_2\setminus S$ has the submanifold $\Proj_1 \subset \Proj_2$ as
equivariant, strong deformation retract. In this section we show that
this method can be generalized to general Grassmann manifolds
$\Gr_p(\C^n)$. Fundamental for this reduction is the fact that the
Schubert variety $\mathcal{S} \subset \Gr_p(\C^n)$ is equivariantly
identifiable with $\Proj_1 \cong S^2$ (see the definition of
discussion of $\mathcal{S}$ on p.~\pageref{DefOfSchubertVarietyC}).

The strategy in this chapter is to identify a higher-dimensional
analog of the subset $S$ in $\Gr_p(\C^\ell)$ such that
$\Gr_p(\C^\ell)\setminus S$ has $\Gr_p(\C^{\ell-1}) \subset
\Gr_p(\C^n)$ as equivariant strong deformation retract. This reduction
step is the basic building block for the reducing the classification
of maps to $\Gr_p(\C^n)$ to the classification of maps to the Schubert
variety $\mathcal{S}$. It works as follows:

Starting with a $G$-map $f\colon X \to \Gr_p(\C^n)$, we identify a
subset $S$ with the aforementioned property and such that (after a
$G$-homotopy) the image $\Im(f)$ does not intersect $S$. Then, using
the equivariant deformation retract to $\Gr_p(\C^{n-1})$, we can
regard $f$ as having its image contained in $\Gr_p(\C^{n-1})$. This
step can be repeated until it can be assumed that the map $f$ has its
image contained in $\Gr_p(\C^{p+1})$. The latter can be equivariantly
identified with $\Gr_1(\C^{p+1})$. It follows that the reduction
procedure just outlined can be repeated again until it can be assumed
that the map $f$ as its image contained in $\Gr_1(\C^2)$. This
subspace $\Gr_1(\C^2)$ identified using the procedure just described
is exactly the Schubert variety $\mathcal{S}$.
We begin with the following remark in this direction:
\begin{remark}
  \label{GrassmannianIdentificationEquivariant}
  Using the standard unitary structure on $\C^n$, the canonical map
  \begin{align*}
    \Psi\colon \Gr_k(\mathbb{C}^n) \to \Gr_{n-k}(\mathbb{C}^n),
  \end{align*}
  where both Grassmannians are equipped with the real structure $T$
  sending a space $V$ to $\overline{V}$, is equivariant with respect
  to $T$.
\end{remark}

\begin{proof}
  It must be shown that, for $V \in \Gr_k(\mathbb{C}^n)$, $T(V)^\perp =
  T(V^\perp)$. This is equivalent to showing that $T(V)$ and
  $T(V^\perp)$ are orthogonal with respect to the standard hermitian
  form on $\mathbb{C}^n$. Thus, let $T(v)$ be in $T(V)$ (with $v \in
  V$) and $T(w)$ be in $T(V^\perp)$ (with $w \in V^\perp$). We have to
  show that $\left<T(v),T(w)\right> = 0$. Equivalently we can show
  that $\smash{\overline{\left<T(v),T(w)\right>}} = 0$. A short computation
  yields:
  \begin{align*}
    \overline{\left<T(v),T(w)\right>} = \overline{T(v)^*T(w)} = \overline{\overline{v}^*\overline{w}} = v^*w = 0.
  \end{align*}
  This proves that $T(V)$ is orthogonal to $T(V^\perp)$ and therefore
  $T(V)^\perp = T(V^\perp)$.
\end{proof}

For the technical arguments in this section we require that the smooth
manifolds $X$ and $\Gr_p(\mathbb{C}^n)$ are both smoothly embedded in
some euclidean space (using e.\,g. the Whitney embedding theorem).
By $\dim_H(\cdot)$ we denote the Hausdorff dimension of a
topological space. See section~\ref{AppendixHausdorffDimension} for
some introductory statements about Hausdorff dimensions. The set $S$,
being the higher dimensional analog of a point in the $\Proj_2$
situation, will be denoted by $\mathcal{L}_L$, where $L$ is a
line\footnote{By \emph{line} we mean a $1$-dimensional complex
  subspace.} in $\C^n$, and is defined as follows:
\begin{align*}
  \mathcal{L}_L = \left\{E \in \Gr_p(\mathbb{C}^n)\colon L \subset E\}\right.
\end{align*}

\begin{remark}
  \label{LLManifold}
  The set $\mathcal{L}_L$ is a submanifold of $\Gr_p(\C^n)$ of
  complex dimension $(p-1)(n-p)$.
\end{remark}
\begin{proof}
  Consider the map
  \begin{align*}
    f\colon \Gr_{p-1}(L^\perp) &\to \Gr_p(\C^n)\\
    E &\mapsto \left<L,E\right>,
  \end{align*}
  where $\left<L,E\right>$ denotes the $p$-plane that is spanned by
  the line $L$ together with the $p-1$-dimensional plane $E \subset
  L^\perp$. Clearly, $\Im f = \mathcal{L}_L$. One can write down $f$
  in local coordinates for the Grassmannians, which shows that $f$ is
  a holomorphic immersion. Furthermore, it is injective. Together with
  the compactness of $\Gr_p(L^\perp)$ this shows that $f$ is an
  holomorphic embedding (see
  e.\,g. \cite[p.~214]{FritzscheGrauert}). In other words,
  $\mathcal{L}_L$ is biholomorphically equivalent to
  $\Gr_{p-1}(\C^{n-1})$, a complex manifold of dimension $(p-1)(n-p)$.
\end{proof}

Let $\C^n = L \oplus W$ be a decomposition of $\C^n$ as direct sum of
a $T$-stable line $L$ and a 1-codi\-men\-sio\-nal, $T$-stable
subvector space $W$. In this setup, we can regard the Grassmann
manifold $\Gr_p(W)$ as a submanifold of $\Gr_p(\C^n)$. The usefulness
of the submanifold $\mathcal{L}_L$ is illustrated by the following
lemma:
\begin{lemma}
  \label{GrassmannianDeformationRetract}
  The space $\Gr_p(\C^n)\smallsetminus \mathcal{L}_L$ has $\Gr_p(W)$ as
  equivariant strong deformation retract.
\end{lemma}
\begin{proof}
  To ease the notation set $\Omega = \Gr_p(\C^n)\setminus
  \mathcal{L}_L$. The idea of this proof is that
  \begin{align*}
    \pi\colon \Omega \to \Gr_p(W)
  \end{align*}
  can be regarded as a complex rank-$k$ vector bundle, where $\pi$ is
  defined in terms of the projection $\C^n \to W$. Here we use the
  fact that a $p$-dimensional plane $E$ in $\C^n$ which does not
  contain the line $L$ projects down to a $p$-dimensional plane in
  $W$.
  The local trivializations of this bundle are defined using local
  trivializations of the tautological bundle $\mathcal{T} \to
  \Gr_p(W)$: Let $U_\alpha$ be a trivializing neighborhood in $\Gr_p(W)$ of
  the tautological bundle and
  \begin{align*}
    \psi_\alpha\colon \Restr{\mathcal{T}}{U_\alpha} \to U_\alpha \times \C^k  
  \end{align*}
  a local trivialization. In other words, the $k$ maps
  \begin{align*}
    f_{\alpha,j}\colon U_\alpha &\to \Restr{\mathcal{T}}{U_\alpha}\\
    E &\mapsto \psi^{-1}_\alpha(E,e_j)\;\text{ for $j=1,\ldots,k$}
  \end{align*}
  define a local frame over $U_\alpha$.  Let $\ell$ be non-zero vector
  in the line $L$. Notice that the fiber of $\Omega \to \Gr_p(W)$ over
  a plane $E \in U_\alpha$ consists of those planes $\smash{\widehat{E}}$
  which are spanned by bases of the form
  \begin{align*}
    \lambda_1 \ell + f_{\alpha,1}(E), \ldots, \lambda_k \ell + f_{\alpha,k}(E),
  \end{align*}
  for a unique vector $\lambda = (\lambda_1,\ldots,\lambda_k)$.  The
  uniqueness follows from the fact that every vector in $\C = L \oplus
  W$ uniquely decomposes into $v_L + v_W$. Now we can define a local
  trivialization for $\Omega \to \Gr_p(W)$ over $U_\alpha$:
  \begin{align*}
    \varphi_\alpha\colon U_\alpha &\to \C^k\\
    \widehat{E} &\mapsto \left(\pi\left(\widehat{E}\right), \lambda\right).
  \end{align*}
  This also defines the vectorspace structure on the fibers: For a given $k$-vector $\lambda$, define
  \begin{align*}
    \widehat{E}_\lambda = \left<\lambda_1 \ell + f_{\alpha,1}(E),\ldots,\lambda_k \ell + f_{\alpha,k}(E)\right>.
  \end{align*}
  Using this notation we obtain
  \begin{align*}
    \widehat{E}_\lambda + \widehat{E}_\mu = \widehat{E}_{\lambda + \mu} \;\text{for all $\lambda,\mu \in \C^k$} \;\text{and}\; a \widehat{E}_{\lambda} = \widehat{E}_{\alpha\lambda} \;\text{for all $a \in \C$}.
  \end{align*}
  Thus, $\pi\colon \Omega \to \Gr_p(W)$ defines a rank-$k$ vector
  bundle. Note that the zero section of this bundle can be naturally
  identified with $\Gr_p(W)$. Thus it remains to show that there is a
  $G$-equivariant retraction of $\Omega$ to its zero section. We
  define the retraction $\varphi_t\colon \Omega \to \Omega$ in local
  trivializations for a covering $\{U_\alpha\}$ via
  \begin{align*}
    \varphi_{\alpha}\colon I \times \C^k &\to \C^k\\
    \lambda &\mapsto t\lambda.
  \end{align*}
  We now show that this gives a globally well-defined map
  $\varphi\colon I \times \Omega \to \Omega$: Let $E$ be a plane in
  $U_\alpha \cap U_\beta$. Over $U_\alpha$, the map $\varphi_\alpha$
  is defined using the frame $f_{\alpha,1},\ldots,f_{\alpha,k}$, while
  $\varphi_\beta$ is defined using the frame
  $f_{\beta,1},\ldots,f_{\beta,k}$. Let $\smash{\widehat{E}}$ be a
  plane in $\restr{\Omega}{U_\alpha \cap U_\beta}$ such that
  \begin{align*}
    \widehat{E} = \left<\ell^\alpha_1 + w^\alpha_1,\ldots,\ell_k^\alpha + w_k^\alpha\right> = \left<\ell^\beta_1 + w^\beta_1,\ldots,\ell^\beta_k + w^\beta_k\right>,
  \end{align*}
  where
  \begin{align*}
    w^\alpha_j = f_{\alpha,j}\left(\pi\left(\widehat{E}\right)\right) \;\text{resp.}\; w^\beta_j = f_{\beta,j}\left(\pi\left(\widehat{E}\right)\right).
  \end{align*}
  Since these two bases generate the same plane, there exists a $k
  \times k$ matrix $A = (a_{ij})$ such that
  \begin{align*}
    \ell^\alpha_j + w^\alpha_j = \sum a_{ij}\left(\ell^\beta_i + w^\beta_i\right).
  \end{align*}
  Using the fact that the intersection $L \cap W$ is trivial we obtain
  \begin{align*}
    \ell^\alpha_j = \sum a_{ij} \ell^\beta_i \;\text{ and }\; w^\alpha_j = \sum a_{ij} w^\beta_i.
  \end{align*}
  But then the bases
  \begin{align*}
    t\ell^\alpha_1 + w^\alpha_1,\ldots,t\ell_k^\alpha + w_k^\alpha \;\text{ and }\; t\ell^\beta_1 + w^\beta_1,\ldots,t\ell^\beta_k + w^\beta_k
  \end{align*}
  are related by the same matrix $A$, hence they define the same plane
  $\smash{\widehat{E}}_t$. In other words, the deformation retract $\varphi_t$
  is globally well-defined on $\Omega$. In particular the above shows
  that for any plane $E = \left<\ell_1 + w_1, \ldots, \ell_k +
    w_k\right>$, where the $W$-vectors are not necessarily the ones
  defined by one of the trivializing frames of the tautological bundle
  we have
  \begin{align*}
    \varphi_t(E) = \left<(1-t)\ell_1 + w_1, \ldots, (1-t)\ell_k + w_k\right>.
  \end{align*}
  Now, using the fact that $L$ and $W$ are both $T$-stable,
  equivariance follows by a computation:
  \begin{align*}
    \varphi_t \circ T(E) &= \varphi_t\left(T\left(\left<\ell_1 + w_1, \ldots, \ell_k + w_k\right>\right)\right) \\
    &= \varphi_t\left(\left<\overline{\ell_1 + w_1}, \ldots, \overline{\ell_k + w_k}\right>\right) \\
    &= \left<(1-t)\overline{\ell_1} + \overline{w_1}, \ldots, (1-t)\overline{\ell_k} + \overline{w_k}\right> \\
    &= \left<\overline{(1-t)\ell_1 + w_1}, \ldots, \overline{(1-t)\ell_k + w_k}\right> \\
    &= T\left(\left<(1-t)\ell_1 + w_1, \ldots, (1-t)\ell_k + w_k\right>\right) \\
    &= T\left(\varphi_t\left(\left<\ell_1 + w_1, \ldots, \ell_k + w_k\right>\right)\right)\\
    &= T \circ \varphi_t(E).
  \end{align*}
  This proves the statement.
\end{proof}

The next intermediate result will be the statement that, given a map
$H$ from $X$ to $\Gr_p(\C^n)$ where $n - 1 > p$, there always exists a
line $L$ in $\C^n$ such that -- at least after a $G$-homotopy -- the
image of $H$ does not intersect $\mathcal{L}_L$. Then,
lemma~\ref{GrassmannianDeformationRetract} is applicable and we obtain
an equivariant homotopy from $H$ to a map $H'$ having its image
contained in the lower-dimensional Grassmannian embedded in
$\Gr_p(\C^n)$. We begin by outlining some technical preparations and
introducing a convenient notation. For a plane $p$-plane $E$ in $\C^n$
we set $E_T = E \cap T(E)$.  Note that $0 \leq \dim E_T \leq p$. For
$k=0,\ldots,p$ define
$\mathcal{E}_k = \{E \in \Gr_p(\mathbb{C}^n)\colon \dim E_T = k\}$.
This induces a stratification of $\Gr_p(\C^n)$:
\begin{align*}
  \Gr_p(\C^n) = \bigcup_{k=0,\ldots,p} \mathcal{E}_k.
\end{align*}
Note that in particular we have
\begin{align*}
 \mathcal{E}_p = \left\{E \in \Gr_p(\C^n)\colon T(E) = E\right\} = (\Gr_p(\C^n))_{\R}
\end{align*}
Define the following incidence set:
\begin{align*}
  \mathcal{I} = \left\{(L,E) \in \Proj(\C^n) \times \Gr_p(\C^n)\colon L \subset E\right\}.
\end{align*}
The set $\mathcal{I}$ is a manifold, which can be seen by considering
the diagonal action of $U(n)$ on the product manifold
$\Proj(\C^n) \times \Gr_p(\C^n)$: $U(n)$ acts transitively on
$\mathcal{I}$. Furthermore we introduce
\begin{align*}
  \mathcal{I}_{\R} = \left\{(L,E) \in (\Proj(\C^n))_{\R} \times \Gr_p(\C^n)\colon L \subset E\right\} = \left\{(L,E) \in \mathcal{I}\colon T(L) = L\right\}.
\end{align*}
We obtain two projections, namely
\begin{align*}
  & \pi_1\colon \mathcal{I} \to \Proj(\C^n)\\
  \;\text{and}\; & \pi_2\colon \mathcal{I} \to \Gr_p(\C^n),
\end{align*}
together with their ``real'' counterparts
\begin{align*}
  & \pi_{1,\R}\colon \mathcal{I}_{\R} \to (\Proj(\C^n))_\R\\
  \;\text{and}\; & \pi_{2,\R}\colon \mathcal{I}_{\R} \to \Gr_p(\C^n).
\end{align*}
We regard the incidence manifold $\mathcal{I}$ as being smoothly
embedded in some $\R^N$ (e.\,g. with the Whitney embedding theorem).

Defining $M_k := \smash{\pi_{2,\R}^{-1}}(\mathcal{E}_k)$ we obtain -- for each
$k$ -- a fiber bundle
\begin{align*}
  M_k \xrightarrow{\;\pi_{2,\R}\;} \mathcal{E}_k
\end{align*}
with the fiber, over a plane $E \in \mathcal{E}_k$, being
\begin{align*}
  F_k &= \left\{(L,E)\colon \text{$L \in \Proj(\C^n)$, $T(L) = L$ and $L \subset E$}\right\}\\
  &= \left\{(L,E)\colon \text{$L \in \Proj(\C^n)$ and $L \subset E_T$}\right\}\\
  &\cong \Proj(E_T)\\
  &\cong \Proj_{k-1}
\end{align*}
for all $k=1,\ldots,p$. In particular we obtain the estimate
\begin{align}
  \label{IterativeRetractionDiscussionFirstEstimate}
  \dim_{\R} F_k = \dim_{\C}(E_T) - 1 = k - 1 \leq p - 1.
\end{align}

Let us now focus on the image $\Im H$ of a map $H\colon X \to
\Gr_p(\C^n)$. Without loss of generality we can assume that $H$ is
equivariantly smoothed (theorem~\ref{SmoothApproximation1}). For each
$k=0,\ldots,p$ we set $H(X)_k = H(X) \cap \mathcal{E}_k$ and $M =
\smash{\pi_{2,\R}^{-1}}(H(X)) \subset \mathcal{I}_\R$. Note that
\begin{align*}
  M = \left\{(L,E) \in (\Proj(\C^n))_\R \times \Gr_p(\C^n)\colon \text{$E \in \Im H$ and $L \subset E$}\right\}.
\end{align*}
Having the above formalism in place, we can now focus on the
fundamental problem: Understanding the image of $M$ under the
projection $\smash{\pi_{1,\R}}$. The following remark highlights what
kind of dimension estimate we need in order to conclude the existence
of a line $L$ such that the image $\Im H$ does not intersect
$\mathcal{L}_L$.
\begin{remark}
  If the inequality
  \begin{align*}
    \hdim(\pi_{1,\R}(M)) < \hdim(\Proj(\C^n)_\R) = n-1
  \end{align*}
  is satisfied, then there exists a line $L$ in $\C^n$ such that
  $\mathcal{L}_L \cap H(X) = \emptyset$.
\end{remark}

\begin{proof}
  The inequality implies that
  $\pi_{1,\R}(M) \subsetneq \Proj(\C^n)_\R$, since otherwise both sets
  would have the same Hausdorff dimension. Recall that
  $M = \smash{\pi_{2,\R}^{-1}}(H(X))$. Thus, the non-surjectivity of
  $\smash{\restr{\pi_{1,\R}}{M}}$ means that there exists a $T$-fixed
  line $L$ in $\C^n$ which has no $\smash{\pi_{1,\R}}$-preimage in the
  set
  \begin{align*}
    \pi_{2,\R}^{-1}(H(X)) = \left\{(L,E) \in \Proj(\C^n)_\R \times \Gr_p(\C^n)\colon \text{$E \in H(X)$ and $L \subset E$}\right\}.
  \end{align*}
  If there was a plane $E \in H(X)$ with $L \subset E$, then $(L,E)$
  would be in ${\pi_{2,\R}}^{-1}(H(X))$, which is a contradiction. Hence,
  there cannot exist such a plane $E$ in the image of $H$.
\end{proof}

By the above remark together with theorem~\ref{HDimProperties} (5) it
suffices to prove that $\hdim(M)$ is smaller than $n - 1$ in order to
conclude the existence of a $T$-fixed line $L$ such that $\Im H$ has
empty intersection with $\mathcal{L}_L$. This is the next goal. The
preimage $M$ can be written as a finite union
\begin{equation*}
  M = \bigcup_{k=0,\ldots,p} {\pi_{2,\R}}^{-1}(H(X)_k).
\end{equation*}
With corollary~\ref{HDimViaSetDecomposition} it follows that
\begin{align}
  \label{HDimOfMAsMax}
  \hdim M = \max_k \left\{\hdim(\pi_{2,\R}^{-1}(H(X)_k))\right\}.
\end{align}

But each set $\smash{\pi_{2,\R}^{-1}}(H(X)_k)$ is contained in the
total space $M_k$ of the fiber bundle $M_k \to
\mathcal{E}_k$. Applying corollary~\ref{HDimOfTotalSpace} yields
\begin{align*}
  \hdim\left(\pi_{2,\R}^{-1}(H(X)_k)\right) = \hdim(H(X)_k) + \dim(F_k).
\end{align*}
Substituting the dimension of the fiber computed in
(\ref{IterativeRetractionDiscussionFirstEstimate}) yields
\begin{align*}
  \hdim\left(\pi_{2,\R}^{-1}(H(X)_k)\right) = \hdim(H(X)_k) + k - 1.
\end{align*}
Therefore, (\ref{HDimOfMAsMax}) implies:
\begin{align}
  \label{HDimOfMAsMaxConcrete}
  \hdim(M) = \max_k \left\{\hdim(H(X)_k) + k - 1\right\}.
\end{align}
Thus, in order to control $\hdim(M)$ it suffices to control each
$\hdim(H(X)_k)$. But since $H(X)$ cannot have Hausdorff dimension
bigger than two ($H$ is assumed to be smooth) and under the assumption
that $p \leq n - 2$, we directly obtain the estimate
\begin{align*}
  \hdim(M) &\leq \max\{p, \hdim(H(X)_p) + p - 1\}\\
  &\leq \max\{n-2,\hdim(H(X)_p) + p - 1\}.
\end{align*}
Thus we only have to control the dimension of $H(X)_p$. In order to
obtain the desired estimate $\hdim(M) < n-1$ it suffices to have
$\hdim(H(X)_p) \leq 1$. In the following we show that -- up to
$G$-homotopy -- this can be assumed. Although, for proving the main
result of this section, it suffices to reduce a general Grassmannian
$\Gr_p(\C^n)$ until we arrive at $\Gr_1(\C^3) \cong \Proj_2$, we state
the following in its general form, i.\,e. including the statement for
the reduction from $\Proj_2$ to $\Proj_1$:
\begin{proposition}
  \label{PropIterativeRetractionDecomposition}
  Given an equviariant map $f\colon X \to \Gr_p(\mathbb{C}^n)$ with
  $n \geq 3$ and $1 \leq p \leq n-2$, there exists a decomposition of
  $\mathbb{C}^n$ into the direct sum of a line
  $L \subset \mathbb{C}^n$ and a 1-codimensional subvector space
  $W \subset \mathbb{C}^n$, both $T$-stable, such that
  $\mathcal{L}_L \cap \Im f = \emptyset$.
\end{proposition}

\begin{proof}
  We prove the statement in two parts. First we prove that the
  statement is true for $0 < p \leq n - 3$ and then we seperately deal
  with the case $p = n - 2$. Both proofs work by analyzing the image
  $\Im f$ and then using dimension estimates to deduce the existence
  of a line $L \subset \mathbb{C}^n$ with $\mathcal{L}_L \cap \Im f$
  being the empty set. Without loss of generality we can assume that
  $f$ is a smooth $G$-map (see
  theorem~\ref{SmoothApproximation1}). Recall that smooth functions
  are, in particular, Lipschitz continuous.

  First assume $p \leq n - 3$. Since $\hdim(X) = 2$, it follows that
  $\hdim(H(X)) \leq 2$ by theorem~\ref{HDimProperties} (5). In
  particular, $\hdim(H(X)_k) \leq 2$ for all $k=0,\ldots,p$. Thus,
  using (\ref{HDimOfMAsMaxConcrete}) we obtain
  \begin{align*}
    \hdim(M) &= \max_{k=0,\ldots,p} \left\{\hdim(H(X)_k) + k - 1\right\} \\
    &\leq p + 1 \\
    &\leq n - 2 \\
    &< n - 1.
  \end{align*}
  Hence, the non-surjectivity of $\smash{\restr{\pi_{1,\R}}{M}}$ is
  established in the case $p \leq n - 3$. For the remaining case $p =
  n - 2$ we can apply lemma~\ref{LemmaSmallIntersectionWithStrata}
  (see below) which guarantees that, at least after a $G$-homotopy:
  \begin{align*}
    \hdim(H(X)_p) \leq 1.
  \end{align*}
  In this case we can write
  \begin{align*}
    \hdim(M) = \max_{k=0,\ldots,p} \{\hdim(H(X)_k) + k - 1\} \leq p = n - 2 < n - 1.
  \end{align*}

  Thus, in both cases we have constructed a homotopy from the map $H$
  to a map $H'$ such that $\mathcal{L}_L \cap \Im(H') = \emptyset$ for
  some $T$-fixed line $L \subset \C^n$. Take $W$ to be the orthogonal
  complement $L^\perp \subset \C^n$. In particular $W$ is also
  $T$-stable.
  This finishes the proof.
\end{proof}

The following lemma completes the proof of the previous
proposition~\ref{PropIterativeRetractionDecomposition}:
\begin{lemma}
  \label{LemmaSmallIntersectionWithStrata}
  Every $G$-map $H\colon X \to \Gr_p(\mathbb{C}^n)$ is
  equivariantly homotopic to a $G$-map $H'$ such that $\hdim(H'(X)_p)
  \leq 1$.
\end{lemma}

\begin{proof}
  Recall that $H(X)_p = H(X) \cap \mathcal{E}_p$ and $\mathcal{E}_p =
  (\Gr_p(\C^n))_{\R}$. Thus in order to minimize the dimension of
  $H(X)_p$ we need to modify $H$ by moving its image away from the
  real points in $\Gr_p(\C^n)$. To ease the notation we set $Y =
  \Gr_p(\C^n)$. Now we let
  \begin{align*}
    N \xrightarrow{\;\pi_{Y,\R}\;} Y_\R
  \end{align*}
  be the normal vector bundle of $Y_{\R}$ in $Y$.
  The bundle $N$ comes equipped with a bundle norm $\|\cdot\|$.
  Now we can use the standard method
  of diffeomorphically identifying an open tubular neighborhood $U$ of
  the 0-section of $N$ (which can be identified with $Y_\R$ itself)
  with an open neighborhood $V$ of $Y_\R$ in $Y$. We denote this
  diffeomorphism $U \to V$ by $\Psi$.
  In the following we construct an equivariant homotopy $H_t$ such
  that $H_0 = H$ and $H_1$ has the desired property of $\hdim(H_1(X)
  \cap Y_\R)$ being at most one.

  For this, let $s$ be a generic section in $\Gamma(Y_\R,N)$. In
  particular this means that there are only finitely many points over
  which $s$ vanishes. Furthermore, let $\chi$ be a smooth cut-off
  function which is constantly one in a small neighborhood $U'
  \subset\subset U$ of the zero section in $N$ and which vanishes
  outside of $U$. Now define
  \begin{align*}
    g^s_t(v) = v - t\chi(v)s(\pi_{Y,\R}(v)).
  \end{align*}
  After scaling $s$ by a constant such that $\|s(\cdot)\|$ is
  sufficiently small, $g^s_t$ defines a diffeomorphism $U \to U$.
  Then, via the diffeomorphism $\Psi$, we obtain an induced
  diffeomorphism $\smash{\tilde{g}^s_t}\colon V \to V$. Since the
  cut-off function $\chi$ vanishes near the boundary of $U$,
  $\smash{\tilde{g}^s_t}$ extends to a diffeomorphism $Y \to Y$ which
  is the identity outside of $V$.  Define
  \begin{align*}
    H^s_t := \tilde{g}^s_t \circ H\colon X \to Y.
  \end{align*}

  Note that $H^s_t$ is not necessarily equivariant anymore for
  positive $t$.  This will be corrected later in the proof by
  restricting the maps $H^s_t$ to a (pseudo-)fundamental region $R$ of
  $X$ and then equivariantly extending to all of $X$. As a first step
  towards equivariance of $H^s_t$, we need to make sure that $H^s_t$
  maps $\Fix T \subset X$ into $Y_\R$ for all $t$. To guarantee this,
  let $f$ be a $C^\infty$ function on $Y_\R$ such that
  \begin{align}
    \label{IterativeRetractionProofHelperFunction}
    \{f = 0 \} = \Psi^{-1}(H(\Fix T)) \subset U.
  \end{align}
  By multiplying the section $s$ with this function $f$ (and, by abuse
  of notation, denoting the resulting section again $s$) it is
  guaranteed that $\smash{\tilde{g}^s_t}$ is the identity along
  $\Fix T \subset X$, hence $H^s_t$ still maps the fixed point set
  $X^T$ into $Y^T = Y_\R$ for all $t$. In the next step we need to
  estimate of the ``critical set''
  \begin{align*}
    C_s = \{x \in X\colon H^s_1(x) \in Y_\R\}.
  \end{align*}
  By construction, a matrix $H^s_1(x)$ is contained in the real part
  $Y_\R$ iff $\Psi^{-1} \circ H^s_1(x)$ is contained in the
  zero-section in the normal bundle $N$.  By definition this means
  \begin{align*}
    g^s_1 \circ \Psi^{-1} \circ H(x) = 0,
  \end{align*}
  which, by definition of $g^s_t$, is equivalent to saying that
  \begin{align*}
    \Psi^{-1}(H(x)) = s\left(\pi_{Y,\R}\left(\Psi^{-1}(H(x))\right)\right).
  \end{align*}
  Using the fact that $C_s$ must be contained in $H^{-1}(V)$ we can
  conclude that
  \begin{align*}
    C_s = \left\{x \in H^{-1}(V)\colon \Psi^{-1} \circ H(x) = s \circ \pi_{Y,\R} \circ \Psi^{-1} \circ H(x)\right\}.
  \end{align*}
  We have to prove that, for some choice of $s$, $\hdim(C_s) \leq 1$.
  If $H(x) \in Y_\R$ for some $x$, then $H^s_1(x) \in Y_\R$ is almost
  never satisfied ($s$ is almost never zero). Hence, it suffices to
  estimate the dimension of the set
  \begin{align*}
    C_s \setminus H^{-1}(Y_\R) = \{x \in \Omega\colon \Psi^{-1} \circ H(x) = s \circ \pi_{Y,\R} \circ \Psi^{-1} \circ H(x)\},
  \end{align*}
  where
  \begin{align*}
    \Omega = H^{-1}(V)\setminus H^{-1}(Y_\R).
  \end{align*}
  We show that by scaling the section $s$ appropriately, we obtain the
  desired estimate $\hdim(C_s) \leq 1$.  Observe that $\Omega$ is an
  open set in $X$, hence in particular a real manifold of dimension
  two, not neccesarily connected. But it has at most countably many
  connected components which we denote by $\{\Omega_j\}_{j \in J}$.
  Define the following function
  \begin{align*}
    h_s\colon \Omega &\to \R\\
    x &\mapsto \frac{\|s(\pi_{Y,\R}(\Psi^{-1}(H(x))))\|}{\|\Psi^{-1}(H(x))\|}.
  \end{align*}
  This quotient is well-defined on $\Omega$, since $\Omega$ does not
  include the $H$-preimage of $Y_\R$, which corresponds to the
  zero-section in $U$. Therefore $h_s$ defines a smooth function on
  $\Omega$. It now follows that
  \begin{align*}
    C_s \setminus H^{-1}(Y_\R) = h_s^{-1}(\{1\}).
  \end{align*}
  Now we introduce the scaling of the section $s$ such that
  $h_s^{-1}(\{1\})$ is at most one-dimensional. Since the number of
  connected components $\Omega_j$ ist at most countable, there exists
  a number $\varepsilon$ arbitrary close to 1 such that there exists
  no component $\Omega_j$ on which $h_s \equiv \varepsilon$ and
  furthermore such that $\varepsilon$ is a regular value for $h_s$ on
  all the components $\Omega_j$ on which $h_s$ is not constant. It
  then follows that $h_s^{-1}(\{\varepsilon\})$ is one-dimensional and
  \begin{align*}
    h_s^{-1}(\{\varepsilon\}) = h_{\varepsilon s}^{-1}(\{1\}) = C_{\varepsilon s}\smallsetminus H^{-1}(Y_\R).
  \end{align*}
  Thus, for the section $\varepsilon s$ we have the desired dimension
  estimate of the critical set.

  Finally we correct the missing equivariance of $H_1$ as follows: In
  the type I case we let $Z$ be the cylinder fundamental region and
  restrict the homotopy $H_t$ just constructed to $Z$. By
  remark~\ref{Class1EquivariantExtension}, the homotopy $\restr{H}{Z}$
  extends uniquely to an equivariant homotopy
  $\smash{\widetilde{H}}\colon I \times X \to Y$. Define
  $H' = \widetilde{H}_1$.
  It remains to check that the above estimate of $h_s^{-1}(\{1\})$
  remains valid.  But this is follows with
  corollary~\ref{HDimViaSetDecomposition}, which proves the statement
  for the type I involution.

  For the type II case we let $R$ be the pseudofundamental region
  introduced in the beginning of section
  (p.~\pageref{ParagraphGeometryOfClass2}). In this case we need to
  make sure that we do not destroy the equivariance propery on the set
  $A = A_1 \cup A_2$ (see p.~\pageref{ParagraphGeometryOfClass2}).  In
  order to be able to use the method of restriction (to $R$) followed
  by equivariant extension, we need to make sure that the homotopy
  $\restr{H}{I \times R}$ behaves well on the boundary of $R$ (see
  remark~\ref{RemarkClassIIEquivariance}). For this we make two small
  adjustments to the above construction. First, using the homotopy
  extension property together with the simply-connectedness of
  $Y$
  we make a homotopy to the original map $H$ such that it is constant
  along $A_1, A_2 \subset \partial R$ (compare with the type II
  normalization, in particular proposition~\ref{Class2Normalization},
  p.~\ref{Class2Normalization}).  It follows that the images $H(A_1)$
  and $H(A_2)$ are one-point sets contained in $Y_\R$. Second, we
  modify the function $f$ introduced in
  (\ref{IterativeRetractionProofHelperFunction}): Instead of letting
  this helper function $f$ vanish exactly over $\Psi^{-1}(H(\Fix T))$,
  we let it vanish over the bigger set $\Psi^{-1}(H(\partial R))$
  (notice that $\Fix T \subset \partial R$). It then follows that the
  homotopy $H_t$ does not change $H$ along the image $H(\partial R)
  \subset V$. In particular, it preserves the compatibility condition
  on $A \subset \partial R$ which is required for the
  equivariance. Now we can construct the equivariant extension to all
  of $X$ as in the type I case above. The desired dimension estimate
  remains satisfied because of
  corollary~\ref{HDimViaSetDecomposition}.
\end{proof}

Having proposition~\ref{PropIterativeRetractionDecomposition} in
place, we formulate the main result of this section:
\begin{restatable*}{proposition}{PropositionReductionOfGrToCurve}
  \label{LemmaReductionOfGrToP1}
  Assume $n \geq 3$ and $1 \leq p \leq n-1$. Let $f\colon X \to
  \Gr_p(\C^n)$ be a $G$-map. Then $f$ is equivariantly homotopic to
  the map $\iota \circ f'$ where $\Im(f')$ is contained in the
  Schubert variety $\mathcal{S}$ and $\iota$ is the above embedding of
  $\mathcal{S}$ into $\Gr_p(\C^n)$. By identifying $\mathcal{S} \cong
  \Proj_1$, the degree triples (resp. degree pairs) of $f\colon X \to
  \Gr_p(\C^n)$ and $f'\colon X \to \mathcal{S} \cong \Proj_1$ agree.
\end{restatable*}
Its proof will be given on p.~\pageref{LemmaReductionOfGrToP1}. By
proposition~\ref{PropIterativeRetractionDecomposition} we know there
exists a line $L$ such that $\mathcal{L}_L$ is not contained in the
image of a given map $H$. For using this statement as the building
block for the iterative retraction procedure it is convenient to be
able to normalize this line $L$. This is made precise in the following
remark:
\begin{remark}
  \label{SOnRNormalizingInGrassmannian}
  Let $L$ be a $T$-stable line in $\C^n$. Then there exists a curve
  $g(t)$ in $\SOR{n}$ such that $g(0) = \Id$ and $g(1)$ maps
  $\mathcal{L}_L$ to $\mathcal{L}_{L_0}$ where $L_0$ is the $n$-th
  standard line $L_0 = \C.e_n$.
\end{remark}

\begin{proof}
  By assumption the line $L$ is $T$-stable. This implies that $L$ is
  generated by a vector $v_n$ of unit length such that $T(v_n) =
  v_n$. In other words, $v_n$ is in $(\C^n)_\R = \R^n$. Furthermore,
  the orthogonal complement $L^\perp$ of $L$ is also
  $T$-invariant. Let $(v_1,\ldots,v_{n-1})$ be an orthonormal basis of
  $(L^\perp)_\R$. Then, $(v_1,\ldots,v_n)$ is an orthonormal basis of
  $\C^n$ consisting solely of real vectors. Define $g$ as the element
  of $\SOR{n}$ which maps $v_j$ to $e_j$ for all
  $j=1,\ldots,n$. Using the path-connectedness of $\SOR{n}$ we can
  find a path $g(t)$ such that $g(0) = \Id$ and $g(1) = g$. Hence, by
  construction, $g(1)$ maps the line $L$ to $L_0$. It remains to show
  that
  \begin{align*}
    g(\mathcal{L}_L) = \mathcal{L}_{L_0}. 
  \end{align*}
  For this, let $E$ be a $p$-plane in $\mathcal{L}_L$. Then, by
  definition, $L \subset E$. But then also $g(L) = L_0 \subset
  g(E)$. On the other hand, given a plane $E$ in with $L_0 \subset E$,
  then define $E' = g^{-1}(E)$. It follows that $L \subset E'$, on
  other words $E' \in \mathcal{L}_L$. Now we see that $E = g(E') \in
  \mathcal{L}_L$.
\end{proof}

Combining the above we can make a reduction from $\Gr_p(\C^n)$ to
the smaller manifold $\Gr_p(\C^{n-1})$:
\begin{proposition}
  \label{PropositionIterativeRetraction}
  Given an equivariant map $f\colon X \to \Gr_p(\mathbb{C}^n)$ where
  $3 \leq n$ and $1 \leq p \leq n - 2$, then there exists a
  $G$-homotopy from $f$ to a map $f'$ whose image is contained in
  $\Gr_p(\mathbb{C}^{n-1})$, where $\C^{n-1}$ is regarded as being
  embedded in $\C^n$ as $(z_1,\ldots,z_n) \mapsto (z_1,\ldots,z_n,0)$.
\end{proposition}

\begin{proof}
  Due to the conditions on $p$ and $n$,
  proposition~\ref{PropIterativeRetractionDecomposition} is
  applicable. Hence there exists a $T$-stable decomposition $\C^n = L
  \oplus W$ such that $\mathcal{L}_L \cap \Im f = \emptyset$. Using
  remark~\ref{SOnRNormalizingInGrassmannian} it follows that there
  exists a curve $g(t)$ of $\SOR{n}$ transformations such that
  $g(1)$ maps $\mathcal{L}_L$ to $\mathcal{L}_{L_0}$ where $L_0$ is
  the standard line $\C.e_n$. Now define a homotopy $F_t =
  g(t)f$. Then, by construction, $F_0 = f$ and $\Im F_1 \cap
  \mathcal{L}_{L_0} = \emptyset$. In this situation
  lemma~\ref{GrassmannianDeformationRetract} is applicable and we
  obtain an equivariant homotopy from $f$ to a map $f'$ such that
  $\Im(f')$ is contained in the lower dimensional Grassmannian
  $\Gr_p(\C^{n-1})$.
\end{proof}

Taking the degree invariants (triples and pairs) into account we state
the following addition to the previous proposition:
\begin{lemma}
  \label{DegreeInvariantsConstantDuringRetraction}
  As before, assume $n \geq 4$ and $1 \leq p \leq n - 2$. Let $f$ be a
  $G$-map $X \to \Gr_p(C^n)$ and denote the canonical embedding
  $\Gr_p(\C^{n-1}) \hookrightarrow \Gr_p(\C^n)$ by $\iota$. Assume
  there exists a $G$-map $f'\colon X \to \Gr_p(\C^{n-1})$ such that
  $f$ and $\iota \circ f'$ are equivariantly homotopic. Then,
  depending on the involution type, the degree triples (type I)
  resp. the degree pairs (type II) of $f$ and $f'$ agree for any two
  such maps.
\end{lemma}

\begin{proof}
  The existence of the map $f'$ is the statement of
  proposition~\ref{PropositionIterativeRetraction}: we obtain a map $f'$,
  equivariantly homotopic to $f$, whose image is contained in
  $\Gr_p(\C^{n-1})$. Thus, we can regard $f'$ as a map to
  $\Gr_p(\C^{n-1})$ and clearly we then obtain $f = \iota \circ f'$.
  Regarding the degree invariants: By
  remark~\ref{GrassmannianEmbeddingIsoInHomology}, the embedding
  $\iota$ induces an isomorphism
  \begin{align*}
    \iota_*\colon H_2(\Gr_p(\C^{n-1}),\Z) \xrightarrow{\;\sim\;} H_2(\Gr_p(\C^n),\Z).
  \end{align*}
  Note that $\iota_*$ maps a generating cycle $[\mathcal{C}]$ of
  $H_2(\Gr_p(\C^{k-1}),\Z)$ to the generating cycle $\iota_*([C]) =
  [\iota(\mathcal{C})]$ of $H_2(\Gr_p(\C^k),\Z)$. We regard
  $\mathcal{C}$ as being canonically oriented as a complex manifold.
  Since $f_* = \iota_* \circ f'_*$, it follows that $f_*$ and $f'_*$
  are defined in terms of the same multiplication factor; in other
  words: The total degree of the map does not change when we regard it
  as a map to the lower dimensional Grassmann manifold.

  It remains to check that the fixed point signatures do not change
  during the retraction. For this let $C \subset X$ be one of the
  fixed point circles (any of the two circles in type I or the unique
  circle in type II). The restrictions
  \begin{align*}
    \Restr{f}{C}\colon C \to \Gr_p(\R^n) \;\text{ and }\; \Restr{\iota \circ f'}{C}\colon C \to \Gr_p(\R^n)
  \end{align*}
  are homotopic, thus they define the same class $[\gamma]$ in the
  fundamental group of $\Gr_p(\R^n)$. By construction $f'$ has its
  image contained in $\Gr_p(\R^{n-1}) \subset \Gr_p(\R^n)$, thus its
  restriction to the circle $C$ can be regarded as a map $C \to
  \Gr_p(\R^{n-1})$, defining a homotopy class $[\gamma']$ in
  $\pi_1(\Gr_p(\R^{n-1}))$. By
  remark~\ref{RealGrassmannianEmbeddingIsoInHomotopy}, the inclusion
  $\Gr_p(\R^{n-1}) \hookrightarrow \Gr_p(\R^n)$ induces an isomorphism
  of their fundamental groups if $n \geq 4$. This implies that $[\gamma]$
  and $[\gamma']$ are either both trivial or both non-trivial, which
  proves the statement.
\end{proof}
To complete the previous lemma we need the following two remarks:
\begin{remark}
  \label{GrassmannianEmbeddingIsoInHomology}
  Assume $0 < p < k - 1$. The embedding
  \begin{align*}
    \iota\colon \Gr_p(\C^{k-1}) \hookrightarrow \Gr_p(\C^k) 
  \end{align*}
  of Grassmann manifold induces an isomorphism on their second homology
  groups. More precisely, let $\mathcal{S}$ be the Schubert variety
  generating $H_2(\Gr_p(\C^{k-1}),\Z)$ (see
  p.~\pageref{DefOfSchubertVarietyC}), then
  \begin{align*}
    \iota_*([\mathcal{S}]) = [\iota(\mathcal{S})].
  \end{align*}
\end{remark}

\begin{proof}
  Note that the second homology groups of complex Grassmann manifolds
  are infinite cyclic. Fix the standard flag in $\C^k$ and let
  $\mathcal{C}$ be the Schubert variety with respect to this flag
  which generates $H_2(\Gr_p(\C^{k-1}),\Z)$. By means of the embedding
  $\iota$ it can also be regarded as being contained in the bigger
  Grassmannian $\Gr_p(\C^k)$, where it also generates
  $H_2(\Gr_p(\C^k),\Z)$. In other words: the embedding $\iota$ maps
  the generator of $H_2(\Gr_p(\C^{k-1}),\Z)$ to the generator of
  $H_2(\Gr_p(\C^k),\Z)$. It follows that the induced map
  \begin{align*}
    \iota_*\colon H_2\left(\Gr_p\left(\C^{k-1}\right),\Z\right) \to H_2\left(\Gr_p\left(\C^k\right),\Z\right)
  \end{align*}
  is an isomorphism.
\end{proof}

\begin{remark}
  \label{RealGrassmannianEmbeddingIsoInHomotopy}
  Let $k \geq 4$. The embedding
  \begin{align*}
    \iota\colon \Gr_p(\R^{k-1}) \hookrightarrow \Gr_p(\R^k)
  \end{align*}
  of real Grassmannians induces an isomorphism of their fundamental
  groups.
\end{remark}

\begin{proof}
  We begin the proof with a general remark: For every Grassmannian
  $\Gr_p(\R^m)$ we have the following double cover
  \begin{align}
    \label{GrassmannianCovering}
    \widetilde{\Gr}_p(\R^m) \to \Gr_p(\R^m),
  \end{align}
  where $\smash{\widetilde{\Gr}_p}(\R^k)$ denotes the \emph{oriented}
  Grassmannian. The covering map is given by forgetting the
  orientation of each subspace. It is known that for $m > 2$ the
  oriented Grassmannian $\Gr_p(\R^m)$ is simply connected. Thus, by
  assumption about $k$, the oriented Grassmannians
  $\smash{\widetilde{\Gr}_p}(\R^{k-1})$ and
  $\smash{\widetilde{\Gr}_p}(\R^k)$ are
  simply-connected.
  Thus, (\ref{GrassmannianCovering}) defines the universal cover of
  $\Gr_p(\R^m)$, and this implies that $\pi_1(\Gr_p(\R^m))$ is
  isomorphic to the Deck transformation group
  $\Deck(\smash{\widetilde{\Gr}_p}(\R^m)/\Gr_p(\R^m))$ (see
  \cite[p.~71]{Hatcher}). The Deck transformation group in this case
  consists of the single homeomorphism $\sigma$, which flips the
  orientation on each subspace, thus it is $C_2$ and we obtain
  $\pi_1(\Gr_p(\R^k)) \cong C_2$.

  To show that the induced map
  \begin{align*}
    \iota_*\colon \pi_1(\Gr_p(\R^{k-1})) \to \pi_1(\Gr_p(\R^k))
  \end{align*}
  is an isomorphism it suffices to show that a non-trivial loop
  $\gamma$ in $\Gr_p(\R^{k-1})$ will still be non-trivial when it is,
  using the embedding $\iota$, regarded as a loop in $\Gr_p(\R^k)$. We
  have the following diagram:
  \[
  \xymatrix{
    \widetilde{\Gr}_p(\R^{k-1}) \ar[d] \ar@{^{(}->}[r]^{\hat{\iota}} & \widetilde{\Gr}_p(\R^k) \ar[d]\\
    \Gr_p(\R^{k-1}) \ar@{^{(}->}[r]_{\iota} & \Gr_p(\R^k)
  }
  \]
  Let $\gamma$ be a non-trivial loop in $\Gr_p(\R^{k-1})$, say
  $\gamma(0) = \gamma(1) = E$. Its lift to the universal cover is a
  non-closed curve $\smash{\hat{\gamma}}$ with
  $\smash{\hat{\gamma}}(0) = E^+$ and $\smash{\hat{\gamma}}(1) = E^-$,
  where $E^+$ and $E^-$ denote the same plane $E$ but equipped with
  different orientations, i.\,e. $\sigma(E^+) = \sigma(E^-)$. It
  follows that $\smash{\hat{\iota}} \circ \smash{\hat{\gamma}}$ is a lift of
  $\iota \circ \gamma$. The curve
  $\smash{\hat{\iota}} \circ \smash{\hat{\gamma}}$ is not closed, as it is
  still a curve whose endpoints are related by the
  orientation-flipping map $\sigma$. Under the isomorphism from the
  Deck transformation group to the fundamental group of the base (see
  e.\,g. on p.~34 the proof of theorem~5.6 in
  \cite{ForsterRiemannSurfaces}), $\sigma$ corresponds to the curve
  $\iota \circ \gamma$. Since $\sigma$ is non-trivial, so is
  $\iota \circ \gamma$. This proves that the embedding $\iota$ induces
  an isomorphism on the fundamental groups.
\end{proof}

Now we can finally prove our main reduction statement:
\PropositionReductionOfGrToCurve
\begin{proof}
  Note that in the case $n=3$, this is just the statement of
  lemma~\ref{ReductionP2toP1} together with
  remark~\ref{ReductionP2toP1Class1}
  resp. remark~\ref{ReductionP2toP1Class2}.

  If $n \geq 4$, apply proposition~\ref{PropositionIterativeRetraction}
  iteratively until we arrive at $n = p + 1$, producing a $G$-homotopy
  from $f$ to a map $f'$ whose image is contained in
  $\Gr_p(\C^{p+1})$. Note that by assumption $p + 1 > 3$. This space
  can be equivariantly identified with $\Gr_1(\C^{p+1})$ by
  remark~\ref{GrassmannianIdentificationEquivariant}. Now
  proposition~\ref{PropositionIterativeRetraction} can be applied again
  iteratively to the map
  \begin{align*}
    \tilde{f}\colon X \to \Gr_p(\C^{p+1}) \cong \Gr_1(\C^{p+1})
  \end{align*}
  until we arrive at $p = 2$, yielding a map
  \begin{align*}
    f'\colon X \to \Gr_1(\C^3) \cong \Proj_2.
  \end{align*}
  By lemma~\ref{DegreeInvariantsConstantDuringRetraction}, up to this
  point, the degree triple (type I) resp. the degree pair (type II) of
  the map is unchanged. For the last reduction step to $\Gr_1(\C^2)$
  we first apply a $G$-homotopy in order to make sure that both maps
  are fixed point normalized (see
  definition~\ref{DefinitionFixpointNormalizationP2}). Then
  lemma~\ref{ReductionP2toP1} and remark~\ref{ReductionP2toP1Class1}
  (type I) resp. remark~\ref{ReductionP2toP1Class2} (type II) together
  with remark~\ref{GrassmannianEmbeddingIsoInHomology} imply that the
  final reduction step to $\mathcal{S} \cong \Proj_1$ also keeps the
  degree triple (resp. the degree pair) unchanged. Thus, in the end we
  have an equivariant homotopy from $f$ to a map whose image is
  contained in $\mathcal{S}$ and whose degree triples (resp. pairs) as
  a map $X \to \mathcal{S} \cong \Proj_1$ are those of $f$.
\end{proof}

We can now prove the main result for the equivariant homotopy
classification of maps $X \to \Gr_p(\C^n)$:
\begin{theorem}
  \label{ClassificationMapsToGrassmannians}
  Let the torus $X$ be equipped with the type I involution (resp. the
  type II involution).  Assume $n > 3$ and $1 < p < n$. Then the
  homotopy class of a map $f$ in $\mathcal{M}_G(X,\Gr_p(\C^n))$ is
  completely determined by its degree triple (type I) resp. its degree
  pair (type II). Furthermore, the image $\Im(\mathcal{T})$
  (resp. $\Im(\mathcal{P})$) consists of those degree triples
  $\Triple{m_0}{d}{m_1}$ (resp. degree pairs $\Pair{m}{d}$) satisfying
  \begin{align*}
    d \equiv m_0 + m_1 \mod 2 \;\text{(resp. $d \equiv m \mod 2$)}.
  \end{align*}
\end{theorem}

\begin{proof}
  We only prove the statement for the type I involution case, as the
  other case works analogously. Let $f$ and $g$ be two $G$-maps with
  the same degree triple. By proposition~\ref{LemmaReductionOfGrToP1}
  these maps are equivariantly homotopic to maps $\iota \circ f'$ and
  $\iota \circ g'$, where $f'$ and $g'$ are $G$-maps $X \to
  \mathcal{S}$ and $\iota$ is the embedding of the Schubert variety
  $\mathcal{S}$ into $\Gr_p(\C^n)$. Furthermore, by the same
  statement, the degree invariants remain unchanged.
  By transitivity, this proves the first statement.

  Regarding the image $\Im(\mathcal{T})$ (resp. $\Im(\mathcal{P})$):
  As before, we can use the embedding
  \begin{align*}
  \mathcal{S} \hookrightarrow \Gr_p(\C^n)  
  \end{align*}
  together with the iterative retraction method to show that the
  conditions
  \begin{align*}
    d \equiv m_0 + m_1 \mod 2\;\; \text{(type I)}\;\;\text{resp.}\;\; d \equiv m \mod 2\;\; \text{(type II)}
  \end{align*}
  are both sufficient and necessary for the degree triples
  (resp. degree pairs) to be in the image of $\mathcal{T}$ (type I)
  resp. $\mathcal{P}$ (type II).
\end{proof}

\subsection{Classification of Maps to $\mathcal{H}_n^*$}

As in the previous cases we note that the degree triple map (resp. the
degree pair) is so far only defined on the mapping spaces
$\mathcal{M}_G(X,\Gr_p(\C^n))$. But after fixing an identification of
each orbit $U(n).\I{p}{q} \subset \mathcal{H}_{(p,q)}$ ($0 < p,q < n$)
with $\Gr_p(\C^n)$, the degree triple map (resp. the degree pair map)
is also defined on the mapping spaces
$\mathcal{M}_G(X,\mathcal{H}_{(p,q)})$ for each non-definite signature
$(p,q)$. This allows us to state and prove the main result of this
section:

\begin{theorem}
  \label{HamiltonianClassificationRankN}
  Let $X$ be a torus equipped with either the type I or the type II
  involution. Assume $n \geq 3$ and $0 < p,q < 1$. Then:
  \begin{enumerate}[(i)]
  \item The sets $[X,\mathcal{H}_{(n,0)}]_G$ and
    $[X,\mathcal{H}_{(0,n)}]_G$ are trivial.
  \item Two $G$-maps $X \to \mathcal{H}_{(p,q)}$ are $G$-homotopic iff
    their degree triples (type I) resp. their degree pairs (type II)
    agree.
  \item The realizable degree triples $\Triple{m_0}{d}{m_1}$ (type I)
    resp. degree pairs $\Pair{m}{d}$ (type II) are exactly those
    which satisfy
    \begin{align*}
      d \equiv m_0 + m_1 \mod 2 \;\text{ resp. }\;d \equiv m \mod 2.
    \end{align*}
  \end{enumerate}
\end{theorem}

\begin{proof}
  The case $n=3$ has already been dealt with in
  theorem~\ref{HamiltonianClassificationRank3}. Therefore it suffices
  to consider the case $n>3$. The topological triviality of the
  definite signature cases is handled in
  remark~\ref{RemarkDefiniteComponentsRetractable}. Let $f$ and $g$ be
  two equivariant maps $X \to \smash{\mathcal{H}_{(p,q)}^*}$. By
  proposition~\ref{GrassmannianStrongDeformationRetract} the image
  space has $\Gr_p(\C^n)$ as equivariant deformation retract. Thus,
  $f$ and $g$ can be regarded as $G$-maps $X \to \Gr_p(\C^n)$ with the
  same degree triple resp. the same degree pair.  Now
  theorem~\ref{ClassificationMapsToGrassmannians} can be applied. Part
  (i) implies that $f$ and $g$ are $G$-homotopic while part (ii)
  contains the statement about realizable degree triples resp. degree
  invariants.
\end{proof}

Remark~\ref{HamiltonianMapRealizationRank3} explains how to concretely
realize maps to $\mathcal{H}_3$ of a given degree invariant. This
works equally well in the general situation. Let $(p,q)$ be a fixed
signature ($p+q=n$) such that $0 < p,q < n$. If $n$ is smaller then
$4$ we end up in one projective situations already handled
(i.\,e. maps to $\Proj_1$ or maps to $\Proj_2$). Thus, assume $n \geq
4$. In section~\ref{SectionN=2} we have seen how to construct maps of
a (realizable) degree triple resp. degree pair to $S^2 \cong
\Proj_1$. We then identify $\Proj_1$ with the Schubert variety
$\mathcal{S}$ and compose this map with the embedding $\mathcal{C}
\hookrightarrow \Gr_p(\C^n)$ to obtain a map into the
$\Gr_p(\C^n)$. The latter Grassmannian needs to be embedded as the
$U(n)$-orbit of $\I{p}{q}$ into $\mathcal{H}_{(p,q)}$ (see
proposition~\ref{GrassmannianStrongDeformationRetract}).


\chapter{Topological Jumps}
\label{ChapterJumps}

In this chapter we construct curves
\begin{align*}
  H\colon [-1,1] \times X \to \mathcal{H}_n
\end{align*}
of equivariant maps $X \to \mathcal{H}_n$ such that the maps $H_{-1}$
and $H_{+1}$, whose images are assumed to be contained in
$\mathcal{H}_n^*$, represent distinct $G$-homotopy classes. In order
to make this precise, we make the following definitions:
\begin{definition}
  Let $H\colon X \to \mathcal{H}_n$ be a $G$-map. Then we define its 
  \emph{singular set} to be the set
  \begin{align*}
    S(H) = \{x \in X\colon \det H(x) = 0\} \subset X.
  \end{align*}
  A $G$-map $X \to \mathcal{H}_n$ is called \emph{singular}
  (resp. non-singular) if its singular set is non-empty (resp. empty).
\end{definition}
\begin{definition}
  \label{DefinitionJumpCurve}
  A \emph{jump curve} from $H_-$ to $H_+$ (both in
  $\mathcal{M}_G(X,\mathcal{H}_n)$) is a $G$-map\footnote{$G$ is
    assumed to act trivially on the interval.}
  $H\colon [-1,1] \times X \to \mathcal{H}_n$ such that
  \begin{enumerate}[(i)]
  \item $H_{\pm 1} = H_\pm$ and
  \item $H_t = H(t,\cdot)$ is non-singular for $t \neq 0$.
  \end{enumerate}
\end{definition}
Given two $G$-maps $H_\pm\colon X \to \mathcal{H}_n$ belonging to
distinct $G$-homotopy classes and a jump curve $H_t$ from $H_-$ to
$H_+$, then $H_0$ must be singular. Otherwise $H_t$ would induce a
$G$-homotopy $I \times X \to \mathcal{H}_n^*$ from $H_-$ to $H_+$,
which would imply that $H_-$ and $H_+$ are equivariantly homotopic.

Of course, given a two $G$-maps $H_\pm\colon X \to \mathcal{H}_n$,
we can always consider the affine curve
\begin{align*}
  (1-t)H_- + tH_+
\end{align*}
of $G$-maps connecting $H_-$ and $H_+$ in the vectorspace
$\mathcal{H}_n$. But in this case we have no control over the singular
set; neither is it guaranteed that the degeneration only occurs at
$t=0$, nor that the singularity set $S(H_0)$ is in some sense
``small''. In this chapter we construct jump curves obeying the
restriction that the singular set $S(H_0)$ is \emph{discrete}. The
main result of this chapter is the description of a procedure for
constructing jump curves for $G$-maps
$X \to \mathcal{H}_{(p,q)} \subset \mathcal{H}_n^*$ ($n \geq 2$) from
any $G$-homotopy class to any other $G$-homotopy class with a finite
singular set; the only requirement is that the the signature $(p,q)$
remains unchanged. Note that jumps from one signature $(p,q)$ to a
different signature $(p',q')$ are not possible with a finite singular
set as shown in the following remark:
\begin{remark}
  If a curve $H_t\colon X \to \mathcal{H}^*_n$ of $G$-maps whose only
  degeneration occurs at $t=0$ jumps from one signature $(p,q)$ to a
  different signature $(p',q')$, then the singular set
  $S(H_0)$ is the whole space $X$.
\end{remark}
\begin{proof}
  Under the assumption that $(p,q) \not= (p',q')$, let $x$ be an
  arbitray point in $X$. We have to show that $H_0(x)$ is
  singular. Assume the opposite, i.\,e. that $H_0(x)$ is
  non-singular. Then $c(t) = H_t(x)$ is a continuous curve in
  $\mathcal{H}_n^*$ with $c(-1) \in \mathcal{H}_{(p,q)}$ and $c(1) \in
  \mathcal{H}_{(p',q')}$. But for $(p,q) \neq (p',q')$,
  $\mathcal{H}_{(p,q)}$ and $\mathcal{H}_{(p',q')}$ denote two
  distinct connected components of $\mathcal{H}^*_n$, this yields a
  contradiction. Therefore, $S(H_0) = X$.
\end{proof}
Note that in order to be able to construct jump curves, it does not
suffice to consider maps of the type $X \to Y$ where $Y$ is a
Grassmannians $\Gr_p(\C^n)$, as these spaces are the deformation
retracts of the components $\mathcal{H}_{(p,q)}$, which consist
entirely of non-singular matrices. Instead we need to let the image
space $Y$ be a subspace of the closure of $\mathcal{H}_{(p,q)}$ having
non-empty intersection with its boundary:
\begin{align*}
  & Y \subset \closure(\mathcal{H}_{(p,q)})\\
  & Y \cap \partial\mathcal{H}_{(p,q)} \not= \emptyset.
\end{align*}
For the curve $H_t$ to be a jump curve it must satisfy
\begin{align*}
  \Im(H_0) \cap \partial\mathcal{H}_{(p,q)} \not= \emptyset,
\end{align*}
as otherwise it would not be singular at $t=0$ and the possibility of
jumps would be excluded. See figure~\ref{fig:jumps} for a depiction of
a jump curve.

As in the previous chapter, we handle the case $n=2$ first and then
use the methods for $n=2$ for proving a general statement. For the
$n=2$ case we will use as image space a certain subspace of
$\cl(\mathcal{H}_{(1,1)})$, namely the vector space $i\mathfrak{su}_2
\subset \cl(\mathcal{H}_{(p,q)})$ consisting of the hermitian
operators of trace zero in $\mathcal{H}_2$. The space
$i\mathfrak{su}_2\setminus\{0\}$ contains the $U(2)$-orbit of
$\I{1}{1}$, which we have already identified as an equivariant strong
deformation retract of $\mathcal{H}_{(1,1)}$ (see
proposition~\ref{Rank2MixedSignatureReduction}), and the origin in
$i\mathfrak{su}_2$, which is the unique singular matrix among the
hermitian matrices of trace zero, is contained in the boundary
$\partial\mathcal{H}_{(1,1)}$. For the general case we embed
$i\mathfrak{su}_2$ into $\mathcal{H}_{(p,q)}$ such that it contains
the -- with respect to the standard flag for $\C^n$ -- unique
one-dimensional Schubert variety $\mathcal{S} \subset \Gr_p(\C^n)$
(see the discussion on p.~\pageref{SchubertVarietyDiscussion}). This
setup allows us to use the same methods as for the $n=2$ case.

\begin{figure}[h]
  \centering
  \subfloat[Map $H_t$ for $t < 0$.]{%
    \centering%
    \scriptsize%
    \def\svgwidth{49.4mm}%
    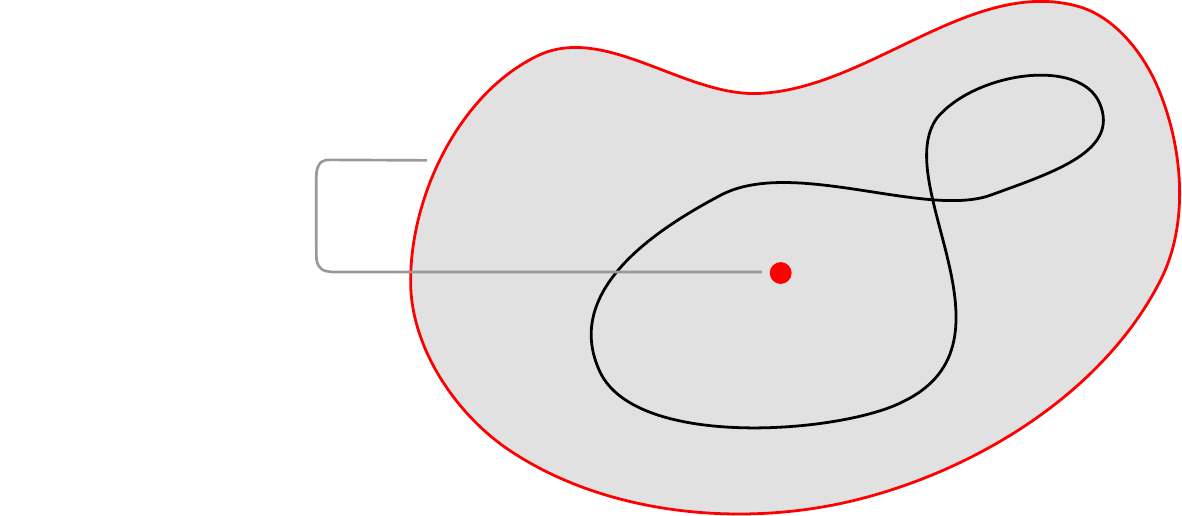}
  \subfloat[Singularity at $t = 0$.]{%
    \centering%
    \def\svgwidth{41mm}%
    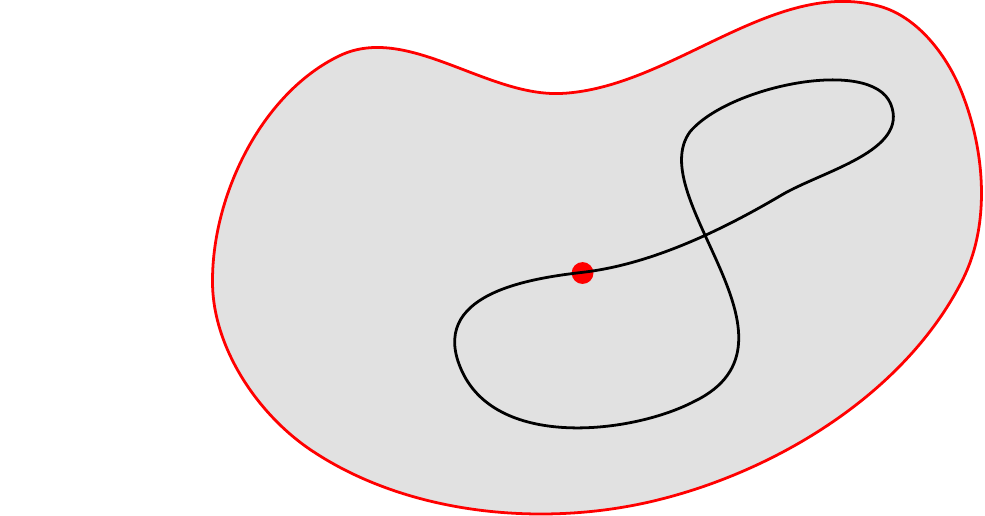}
  \subfloat[Map $H_t$ for $t > 0$.]{%
    \centering%
    \def\svgwidth{41mm}%
    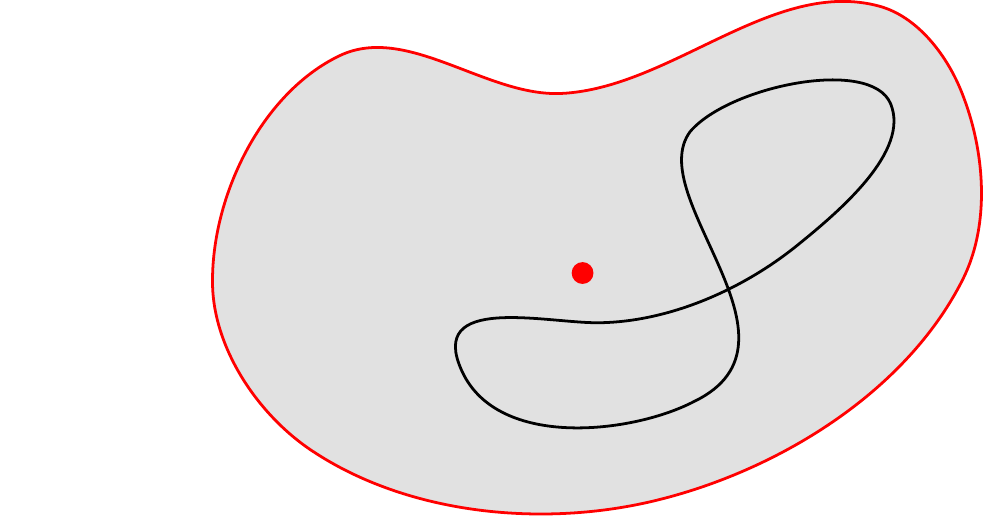}
  \caption{A jump curve.}
  \label{fig:jumps}
\end{figure}

\section{Maps into $\cl(\mathcal{H}_2)$}

Let us first understand how jumps can occur for maps into
$\mathcal{H}_2$. Since the components $\mathcal{H}_{(2,0)}$ and
$\mathcal{H}_{(0,2)}$ are topologically trivial, the only relevant
component to consider is $\mathcal{H}_{(1,1)}$. Recall the definition
of the vector space $i\mathfrak{su}_2$:
\begin{align*}
  i\su_2 = \left\{
  \begin{pmatrix}
    a & \phantom{-}b \\
    \overline{b} & -a
  \end{pmatrix}\colon a \in \R, b \in \C\right\}.
\end{align*}
The space $i\mathfrak{su}_2\setminus\{0\}$ is contained in
$\mathcal{H}_{(1,1)}$. It is important that this space is not
\emph{entirely} contained in $\mathcal{H}_{(1,1)}$, as we have
\begin{align*}
  i\su_2 \cap \partial\mathcal{H}_{(1,1)} = \left\{
    \begin{pmatrix}
      0 & 0 \\
      0 & 0
    \end{pmatrix}\right\}
\end{align*}
As discussed earlier, $i\mathfrak{su}_2$ can be linearly identified
with $\R^3$ via
\begin{align}
  \label{JumpDiscussionIsomToR3}
  \begin{pmatrix}
    a & \phantom{-}b \\
    \overline{b} & -a
  \end{pmatrix} \mapsto
  \begin{pmatrix}
    a \\
    \Re(b) \\
    \Im(b)
  \end{pmatrix}.
\end{align}
Recall that the $U(2)$-orbit of the diagonal matrix $\I{1}{1}$ is
contained in $i\su_2\smallsetminus\{0\}$ and is identified by the
isomorphism (\ref{JumpDiscussionIsomToR3}) with the unit sphere in
$\R^3$.

In the following we construct jump curves as follows: Given two
$G$-maps $H_\pm\colon X \to S^2$, we regard them as maps into
$S^2 \subset \R^3$ and then construct a map of the form
$H\colon [-1,1] \times X \to \R^3$ such that $H_{-1} = H_-$ and
$H_{+1} = H_+$. First we introduce some new definitions, which are
slightly more suitable for this concrete approach than
definition~\ref{DefinitionJumpCurve}:
\begin{definition}
  Let $H\colon X \to \R^3$ be a $G$-map. We define the \emph{singular
    set} $S(H)$ of $H$ to be the fiber $H^{-1}(\{0\})$.  The singular
  set of a map $(Z,C) \to (\R^3,\{z = 0\})$ is the singular set of its
  equivariant extension.
\end{definition}

\begin{definition}
  \label{DefinitionJumpCurveR3}
  A \emph{jump curve} from $H_-$ to $H_+$ (both in
  $\mathcal{M}_G(X,\R^3\smallsetminus\{0\})$) is a $G$-map
  $H\colon [-1,1] \times X \to \R^3$ such that
  \begin{enumerate}[(i)]
  \item $H_{\pm 1} = H_\pm$ and
  \item $0 \not\in \Im(H_t)$ for $t \neq 0$.
  \end{enumerate}
  Futhermore, a \emph{jump curve} from
  \begin{align*}
   H_- \in \mathcal{M}\left((Z,C),\left(\R^3\smallsetminus\{0\},\left(\R^3\smallsetminus\{0\}\right)\cap\{z=0\}\right)\right)
  \end{align*}
  to
  \begin{align*}
   H_+ \in \mathcal{M}\left((Z,C),\left(\R^3\smallsetminus\{0\},\left(\R^3\smallsetminus\{0\}\right)\cap\{z=0\}\right)\right)
  \end{align*}
  is a $G$-map $H\colon [-1,1] \times (Z,C) \to (\R^3,\{z = 0\})$
  such that the equivariant extension of $H$ to $[-1,1] \times X$ is a
  jump curve from the equivariant extension of $H_-$ to the
  equivariant extension of $H_+$.
\end{definition}
By means of the isomorphism (\ref{JumpDiscussionIsomToR3}), a jump
curve between maps $X \to \R^3\smallsetminus\{0\}$ induces a jump
curve between maps $X \to \mathcal{H}_n^*$.  For the concrete
construction of the curve $H_t$ we often require that the maps $H_\pm$
are in \emph{normal form} (see p.~\pageref{RemarkAboutNormalForms}) or
in \emph{modified normal form} (see
p.~\pageref{DefModifiedNormalForm},
definition~\ref{DefModifiedNormalForm}). We use the terms
\emph{curves} and \emph{homotopies} (of $G$-maps) interchangeably.

\subsection{Type I}

In this section, $X$ always denotes the torus equipped with the type I
involution. For illustrative purposes, we begin with a very naive
jumping method, which does not satisfy our requirement of the singular
set being small:
\begin{lemma}
  Let $X$ be the type I torus and let $d_0^\pm, d_1^\pm, d_0^\pm,
  d_1^\pm$ be integers. Denote by $H_\pm\colon X \to
  \mathcal{H}_{(1,1)}$ the $G$-maps in normal form for the triples
  \begin{align*}
    \Triple{d_0^\pm}{d_0^\pm-d_1^\pm}{d_1^\pm}.
  \end{align*}
  Then there exists a jump curve
  \begin{align*}
    H\colon [-1,1] \times X \to \closure(\mathcal{H}_{(1,1)})
  \end{align*}
  from $H_-$ to $H_+$ such that the singular set $S(H_0)$ consists of
  the two circles $C_0 \cup C_1$.
\end{lemma}

\begin{proof}
  Let $H_\pm\colon X \to S^2 \subset \R^3$ be the maps in normal form
  for the triples
  $\smash{\Triple{d_0^\pm}{d_0^\pm - d_1^\pm}{d_1^\pm}}$ as it has
  been constructed as part of the proof of
  proposition~\ref{TripleBuildingBlocks} (ii).  In the following we
  construct a jump curve
  \begin{align*}
    H\colon [-1,1] \times X \to \R^3
  \end{align*}
  from $H_-$ to $H_+$. The maps $H_\pm$ are of the form
  \begin{align*}
    H_{\pm} =
    \begin{pmatrix}
      x_\pm \\
      y_\pm \\
      z_\pm
    \end{pmatrix}.
  \end{align*}
  Since both maps are in normal form (see proof of
  proposition~\ref{TripleBuildingBlocks}), they have the convenient
  property that $z_- = z_+$. The functions $x_\pm$ and $y_\pm$ define
  the rotation defined by the respective degree triples.  Note that a
  matrix $H_\pm(p)$ is singular iff $x_\pm(p) = y_\pm(p) = z(p) = 0$.
  The functions $x_-$, $y_-$ and $z$ (resp. $x_+$, $y_+$ and $z$) do
  not simultaneously vanish, because the maps $H_\pm$ are non-singular
  by assumption. Now we define the jump curve as:
  \begin{align*}
    H_t = \begin{pmatrix}
      x_t \\
      y_t \\
      z
    \end{pmatrix}
  \end{align*}
  where
  \begin{align*}
    x_t =
    \begin{cases}
      |t| x_-  & \;\text{for $t < 0$}\\
      0  & \;\text{for $t = 0$}\\
      |t| x_+  & \;\text{for $t > 0$}\\
    \end{cases}
    \;\;\text{and}\;\;
    y_t =
    \begin{cases}
      |t| y_-  & \;\text{for $t < 0$}\\
      0  & \;\text{for $t = 0$}\\
      |t| y_+  & \;\text{for $t > 0$}\\
    \end{cases}
  \end{align*}
  Then $H_t$ defines a non-singular map for $t \neq 0$. The only
  points of degeneracy which were introduced by the $t$-scaling are
  those points in $X$ where together with $z$ also $t$ vanishes. Since
  the maps $H_\pm$ are in normal form, this means
  \begin{equation*}
    S(H_0) = \{p \in X\colon z(p) = 0 \} = C_0 \cup C_1,
  \end{equation*}
  which finishes the proof.
\end{proof}

For the construction of more interesting topological jumps we employ a
new normal form for maps for triples $\Triple{d_0}{d}{d_1}$. This is
described in the following definition
\begin{definition}
  \label{DefModifiedNormalForm}
  For triples $\Triple{d_0}{\pm(d_0 - d_1)}{d_0}$ we define $G$-maps
  $F\colon X \to S^2 \subset \R^3$ in \emph{modified normal form} as
  follows: Let $D$ be the closed two-disk, $D = \{z \in
  \mathbb{C}\colon |z| \leq 1\}$ and let $\iota_D\colon (D,\partial D)
  \hookrightarrow (S^2,E)$ be the embedding
  (\ref{IotaHemisphereEmbedding} from
  p.~\pageref{IotaHemisphereEmbedding} of $D$ onto one of the
  hemispheres of $S^2$. By $Z$ we denote the fundamental region
  cylinder of the torus $X$. Now define the map
  \begin{align*}
    f\colon (Z,C) &\to (D,\partial D)\\
    (t, \varphi) &\mapsto \begin{cases}
      (1-2t) e^{i d_1 \varphi} + 2t & \text{ for $0 \leq t \leq \frac{1}{2}$}\\
      (2t-1) e^{i d_2 \varphi} + 2(1-t) & \text{ for $\frac{1}{2} \leq t \leq 1$}
    \end{cases}
  \end{align*}
  and let $F\colon X \to S^2$ be the equivariant extension of the
  composition map $\iota_D \circ f$ (see
  lemma~\ref{Class1EquivariantExtension}). The resulting map has the
  triple $\Triple{d_0}{d_0-d_1}{d_1}$.
\end{definition}
The fact that the map $F$ constructed above has the triple
$\Triple{d_0}{d_0-d_1}{d_0}$ can be shown as in the proof for
proposition~\ref{TripleBuildingBlocks}~(ii).  The crucial point of
this modified normal form is that the fibers over the north-
resp. south pole of the sphere contains only a finite number of
points, precisely $|d_0| + |d_1|$ for each of the poles. In the original
normal form the fibers over the poles were copies of $S^1$ in $X$ (see
figure~\ref{FigureModifiedNormalForm}).

\begin{figure}[h]
  \centering
  \subfloat[The homotopy $(1-2t)e^{id_1} + 2t$.]{%
    \centering%
    \begin{tikzpicture}[scale=0.84]
      \draw (0,0) circle[radius=70pt];
      \draw[dashed] (35pt,0) circle[radius=35pt];
      \draw[dashed] (52.5pt,0) circle[radius=17.5pt];
      \draw[dashed] (17.5pt,0) circle[radius=52.5pt];
      \draw[font=\footnotesize] (70pt,0) coordinate (O) node[right] {};
      \draw[font=\footnotesize] (-80pt,45pt) node[left] {$t=\frac{1}{8}$};
      \draw[font=\footnotesize] (-80pt,30pt) node[left] {$t=\frac{1}{4}$};
      \draw[font=\footnotesize] (-80pt,15pt) node[left] {$t=\frac{3}{8}$};
      \draw[font=\footnotesize] (-80pt,0pt) node[left] {$t=\frac{1}{2}$};
      \draw[color=gray!60,dashed] (-75pt,45pt) -- (-20pt,45pt);
      \draw[color=gray!60,dashed] (-75pt,30pt) -- (10pt,30pt);
      \draw[color=gray!60,dashed] (-75pt,15pt) -- (34pt,15pt);
      \draw[color=gray!60,dashed] (-75pt,0pt) -- (65pt,0pt);
      \draw[font=\footnotesize] (50pt,60pt) node[above right] {};
      \fill [radius=2pt] (O) circle;
    \end{tikzpicture}}\qquad
  \subfloat[The homotopy $(1-2t)e^{id_2} + 2t$.]{%
    \centering%
    \begin{tikzpicture}[scale=0.84]
      \draw (0,0) circle[radius=70pt];
      \draw[dashed] (35pt,0) circle[radius=35pt];
      \draw[dashed] (52.5pt,0) circle[radius=17.5pt];
      \draw[dashed] (17.5pt,0) circle[radius=52.5pt];
      \draw[font=\footnotesize] (70pt,0) coordinate (O) node[right] {};
      \draw[font=\footnotesize] (-80pt,45pt) node[left] {$t=\frac{7}{8}$};
      \draw[font=\footnotesize] (-80pt,30pt) node[left] {$t=\frac{3}{4}$};
      \draw[font=\footnotesize] (-80pt,15pt) node[left] {$t=\frac{5}{8}$};
      \draw[font=\footnotesize] (-80pt,0pt) node[left] {$t=\frac{1}{2}$};
      \draw[color=gray!60,dashed] (-75pt,45pt) -- (-20pt,45pt);
      \draw[color=gray!60,dashed] (-75pt,30pt) -- (10pt,30pt);
      \draw[color=gray!60,dashed] (-75pt,15pt) -- (34pt,15pt);
      \draw[color=gray!60,dashed] (-75pt,0pt) -- (65pt,0pt);
      \draw[font=\footnotesize] (50pt,60pt) node[above right] {};
      \fill [radius=2pt] (O) circle;
    \end{tikzpicture}}
  \caption{Homotopy Visualization on the closed $2$-Disk.}
  \label{FigureModifiedNormalForm}
\end{figure}
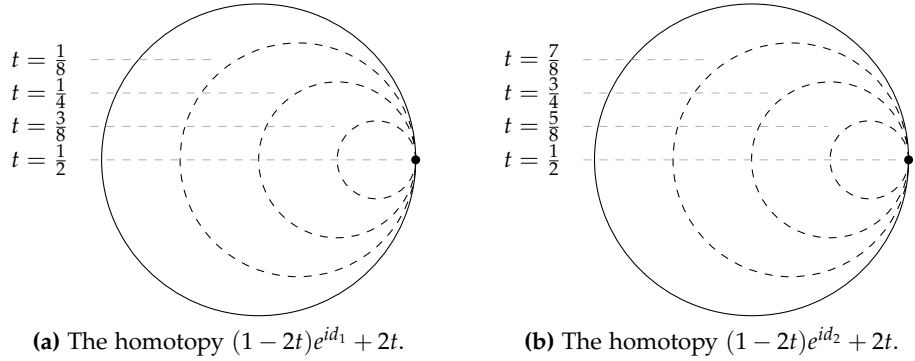

Now we construct a jump curve with only a discrete set of singular
points. This jump curve ``flips'' the total degree of maps in modified
normal form for basic triples.
\begin{lemma}
  \label{JumpsRank2Class1TotalDegree}
  Let $d_0$ and $d_1$ be two integers and assume that the maps
  \begin{align*}
    H_\pm\colon (Z,C) \to \left(S^2,E\right) \subset \left(\R^3,\{z=0\}\right)
  \end{align*}
  have the triples $\Triple{d_0}{\pm(d_1-d_0)}{d_1}$. Then there
  exists a jump curve
  \begin{align*}
    H\colon [-1,1] \times (Z,C) \to \left(\R^3,\{z=0\}\right)
  \end{align*}
  from $H_-$ to $H_+$ such that the singular set of the equivariant
  extension of $H_0$ consists of $2(|d_0| + |d_1|)$ isolated points in
  $X\setminus (C_0 \cup C_1)$. If the maps are already normalized
  along the boundary circles $C_0$ and $C_1$, then this homotopy $H_t$
  can be chosen to be relative to $C_0 \cup C_1$.
\end{lemma}

\begin{proof}
  After a $G$-homotopy we can assume that the maps $H_\pm$ are in
  modified normal form for the triples
  $\Triple{d_0}{\pm(d_1-d_0)}{d_1}$. This homotopy can be chosen to be
  relative to $C_0 \cup C_1$ if the original maps $H_\pm$ are already
  type I normalized. It now suffices to prove the statement under the
  assumption that $H_\pm$ are in modified normal form. We can regard
  $H_\pm$ as maps
  \begin{align*}
    H_\pm\colon (Z,C) &\to (\R^3,\{z=0\})\\
    p &\mapsto
    \begin{pmatrix}
      x_\pm(p) \\
      y_\pm(p) \\
      z_\pm(p)
    \end{pmatrix}
  \end{align*}
  where the functions $x_\pm$, $y_\pm$ and $z_\pm$ satisfy:
  \begin{align*}
    & x_\pm \circ T = x_\pm\\
    & y_\pm \circ T = y_\pm\\
    & z_\pm \circ T = -z_\pm.
  \end{align*}
  Note that the functions $x_-$, $y_-$ and $z_-$ (resp. $x_+$, $y_+$
  and $z_+$) do not simultaneously vanish. Also, since the maps in
  modified normal form for triples of the form $\Triple{d_0}{\pm(d_1 -
    d_0)}{d_1}$ only differ by a reflection along the $x,y$-plane, we
  can conclude that $x_- = x_+$ , $y_- = y_+$ and $z_- = -z_+$. We set
  $x = x_\pm$ and $y = y_\pm$ and define the jump curve $H_t$ as
  follows:
  \begin{align*}
    H\colon [-1,1] \times (Z,C) &\to (\mathbb{R}^3,\{z=0\})\\
    (t,p) &\mapsto
    \begin{pmatrix}
      x(p)\\
      y(p)\\
      z_t(p)
    \end{pmatrix}
  \end{align*}
  where
  \begin{align*}
    z_t(p) =
    \begin{cases}
      |t| z_-(p) & \;\text{for $t < 0$} \\
      0 & \;\text{for $t = 0$} \\
      |t| z_+(p) & \;\text{for $t > 0$}
    \end{cases}.
  \end{align*}
  As long as $t \neq 0$, the map $H_t$ is non-singular. For $t=0$, the
  singular set $S(H_0)$ consists of those points $p \in X$ which are
  mapped to the south resp. the north pole (this is the set $\{x = y =
  0\}$). In the cylinder $Z \subset X$ we obtain $|d_0|$ singular
  points during the shrinking process of the degree $d_0$ loop and
  $|d_1|$ points during the enlarging process of the degree $d_1$
  loop. By equivariant extension
  (lemma~\ref{Class1EquivariantExtension}), the same applies to the
  complementary fundamental cylinder $Z'$, yielding a total number of
  $2(|d_0|+|d_1|)$ singular points in $X$.
  Equivariance of the curve is guaranteed because the equivariance
  conditions on the functions $x_\pm$, $y_\pm$ and $z$ are compatible
  with scaling by real numbers.
\end{proof}

Given integers $d_0$ and $d_1$, let 
\begin{align*}
  H^{\Triple{d_0}{\pm(d_1-d_0)}{d_1}}\colon [-1,1] \times (Z,C) \to (\R^3,\{z=0\}).
\end{align*}
denote the jump curve from $H_-$ to $H_+$
(with $\mathcal{T}(H_\pm) = \Triple{d_0}{\pm(d_1-d_0)}{d_1}$) whose
existence is guaranteed by lemma~\ref{JumpsRank2Class1TotalDegree}. In
the following we consider jump curves, which modify the fixed point
degrees. For the proof of the next lemma it is useful to introduce a
new notation for maps $(Z,C) \to (S^2,E)$ having triples of the form
$\Triple{d_0}{d_0-d_1}{d_1}$. Let us recall how the normal form maps
$(Z,C) \to (S^2,E)$ for such triples were constructed in the proof of
proposition~\ref{TripleBuildingBlocks} (ii). The cylinder is to be
regarded as $I \times S^1$ where we use $s$ as the interval
coordinate. Denote the closed unit disk in $\C$ by $D$. Let
$\gamma_d\colon S^1 \to S^1$ be the normalized loop of degree $d$.
Let $\iota_D$ be the embedding (\ref{IotaHemisphereEmbedding}) of the
disk $D$ as one of the hemispheres in $S^2$ and let $f$ be the
following map:
\begin{align*}
  f\colon [0,1] \times S^1 &\to D\\
  (s,z) &\mapsto
  \begin{cases}
    (1-2s)\gamma_{d_0}(z) & \;\text{for $0 \leq t \leq \onehalf$}\\
    (2s-1)\gamma_{d_1}(z) & \;\text{for $\onehalf \leq t \leq 1$}\\
  \end{cases}.
\end{align*}
In particular, $f_0$ and $f_1$ both map $S^1$ to the boundary of
$D$. Therefore, $\iota_D \circ f$ can be regarded as a map
\begin{align*}
  (Z,C) \to \left(S^2,E\right) \subset \left(\R^3,\{z=0\}\right).
\end{align*}
We can then define a map for the above triple as the equivariant
extension of the composition $\iota_D \circ f$ (see
lemma~\ref{Class1EquivariantExtension}):
\begin{align*}
  F = \widehat{\iota_D \circ f}. 
\end{align*}
In other words, $F$ is the equivariant extension of the map
\begin{align}
  \label{TopJumpsDefBigF}
  I \times S^1 &\to S^2\\
  (s,z) &\mapsto
  \begin{cases}
    \begin{pmatrix}
      (1-2s)\Re(\gamma_{d_0}(z))\\
      (1-2s)\Im(\gamma_{d_0}(z)\\
      -2\sqrt{s-s^2}
    \end{pmatrix} & \;\text{for $s \in \left[0,\onehalf\right]$}\\
    \begin{pmatrix}
      (1-2s)\Re(\gamma_{d_1}(z))\\
      (1-2s)\Im(\gamma_{d_1}(z))\\
      -2\sqrt{s-s^2}
    \end{pmatrix} & \;\text{for $s \in \left[\onehalf,1\right]$}
  \end{cases}.\nonumber
\end{align}
What this shows is that normal form maps for a triple
$\Triple{d_0}{\pm(d_0-d_1)}{d_1}$ can be defined in terms of two loops
$S^1 \to S^1$ which are of degree $d_0$ resp. of degree $d_1$. This
motivates the following
\begin{definition}
  By $F_{\alpha,\beta}$ we denote the map $(Z,C) \to
  \left(S^2,E\right)$ which is defined in terms of the loops $\alpha$
  and $\beta$ as described above. In particular, $F_{\alpha,\beta}$
  has the triple $\Triple{\deg \alpha}{\deg \alpha - \deg \beta}{\deg
    \beta}$.
\end{definition}

Recall that for loops $\gamma_{d_0}, \ldots,\gamma_{d_n}$ we can form
the concatenation loop
\begin{align*}
  \gamma_{d_n} \ast \ldots \ast \gamma_{d_0}\colon S^1 &\to S^1\\
  z &\mapsto \gamma_{d_j}(z_{n+1}) \text{, where } \arg z \in\left[\frac{2\pi j}{n+1},\frac{2\pi(j+1)}{n+1}\right]
\end{align*}

\begin{lemma}
  \label{JumpsRank2Class1FixpointDegree}
  Let $d^\pm$ be two integers and $H_\pm\colon (Z,C) \to
  \left(S^2,E\right) \subset (\R^3,\{z=0\})$ be maps with the degree
  triples
  \begin{align*}
    \Triple{d^\pm}{d^\pm}{0}\;\text{(resp. $\Triple{0}{d^\pm}{d^\pm}$)}.
  \end{align*}
  Then there exists a jump curve
  $H\colon [-1,1] \times (Z,C) \to \left(\R^3,\{z=0\}\right)$ from
  $H_-$ to $H_+$ such that the singular set $S(H_0)$ consists of
  $|d^+ - d^-|$ isolated points in $C_0$ (resp. in $C_1$). If the
  original maps $H_\pm$ are already normalized along the boundary
  circle $C_1$ (resp. $C_0$), then $H_t$ can be chosen relative to
  $C_1$ (resp. $C_0$).
\end{lemma}

\begin{proof}
  We prove the statement only for the triples
  $\Triple{d^\pm}{d^\pm}{0}$; the other case works completely
  analogously. After a $G$-homotopy we can assume that $H_-$ is
  defined in terms of the loops $\gamma_k^{-1} \ast \gamma_k \ast
  \gamma_{d^-}$ and $\gamma_0$:
  \begin{align}
    H_- = F_{\gamma_k^{-1} \ast \gamma_k \ast \gamma_{d^-},\gamma_0}.
  \end{align}
  After defining $k = d^+ - d^-$ and after another $G$-homotopy we can
  assume that $H_+$ is defined in terms of the loops $\gamma_0 \ast
  \gamma_k \ast \gamma_{d^-}$ and $\gamma_0$:
  \begin{align}
    H_+ = F_{\gamma_0 \ast \gamma_k \ast \gamma_{d^-},\gamma_0}.
  \end{align}
  In order to prove the statement it suffices to construct a jump
  curve $H_t$ from $H_-$ to $H_+$. We define the curve $H_t$ as:
  \begin{align*}
    H_t = F_{\iota_t \ast \gamma_k \ast \gamma_{d^-},\gamma_0},
  \end{align*}
  where $\iota_t$ (``jump'') is the following null-homotopy of the
  loop $\gamma_k^{-1}$:
  \begin{align*}
    \iota_t\colon [-1,1] \times S^1 &\to \C\\
    (t,z) &\mapsto \frac{(1-t)z^{-k} + t + 1}{2}.
  \end{align*}
  By construction we have $H_{\pm 1} = H_\pm$. Let us compute the
  singular set of $H_t$. From the definition of the map
  $F_{\alpha,\beta}$ in (\ref{TopJumpsDefBigF}) we can conclude that
  the singular set consists of the common zeros of
  \begin{align*}
    & (1-2s)\iota_t \ast \gamma_k \ast \gamma_{d^-}(z)\\
    \text{and}\; & -2\sqrt{s-s^2}.
  \end{align*}
  Since the square root only vanishes for $s \in \{0,1\}$ and $(1-2s)$
  does not vanish for these values of $s$, the only possibility for
  both functions to vanish is
  \begin{align*}
    s = 0 \;\text{and}\;\iota_t \ast \gamma_k \ast \gamma_{d^-}(z) = 0.
  \end{align*}
  By construction, $\iota_t$ crosses the origin once and that happens
  at $t=0$. Therefore, $H_t$ is only singular for $t = 0$ and the
  singular set consists of $|k| = |d^+ - d^-|$ points in $C_0 \subset
  Z$.
\end{proof}

Note that this method also changes the total degree of the map. That
is to be expected, since we have all the freedom possible for changing
the fixed point degrees, but as we have seen earlier, not all
combinations of fixed point degrees and total degree are allowed. The
previous lemma also covers the special case $d^- = d^+$, in which
there are no singular points introduced. For two integers $d^\pm$ we
denote the jump curve from $H_-$ to $H_+$ (with
$\mathcal{T}(H_\pm) = \Triple{d^\pm}{d^\pm}{0}$
resp. $\mathcal{T}(H_\pm) = \Triple{0}{d^\pm}{d^\pm}$) whose existence
is guaranteed by lemma~\ref{JumpsRank2Class1FixpointDegree} by
\begin{align*}
  & H^{\Triple{d^\pm}{d^\pm}{0}}\colon [-1,1] \times (Z,C) \to \left(\R^3,\{z=0\}\right)\\
  \text{resp. }\; & H^{\Triple{0}{d^\pm}{d^\pm}}\colon [-1,1] \times (Z,C) \to \left(\R^3,\{z=0\}\right).
\end{align*}

Now we combine the results of lemma~\ref{JumpsRank2Class1TotalDegree}
and~\ref{JumpsRank2Class1FixpointDegree} into one theorem:
\begin{theorem}
  \label{JumpsRank2Class1}
  Let the torus $X$ be eqipped with the type I involution and let
  $H_\pm\colon X \to S^2 \subset \R^3$ be two $G$-maps with triples
  $\Triple{d_0^\pm}{d^\pm}{d_1^\pm}$. Then there exists a jump curve
  $H\colon [-1,1] \times X \to \R^3$ of $G$-maps from $H_-$ to $H_+$
  such that the singular set $S(H_0) \subset X$ consists of
  \begin{align*}
    |d_0^+ - d_0^-| + |d_1^+ - d_1^-| + d^+ - (d_0^+ + d_1^+) - \left[d^- - (d_0^- + d_1^-)\right]
  \end{align*}
  isolated points.
  
  If the original maps $H_\pm$ are type I normalized along the
  boundary circles $C_0$ and $C_1$ and $d_0^- = d_0^+$ or $d_1^- =
  d_1^+$, then the curve $H_t$ can be chosen such that it does not
  modify the maps on the respective boundary circles where they agree.
\end{theorem}

\begin{proof}
  Assume we are given two maps $H_\pm\colon X \to S^2 \subset \R^3$
  having the degree triples $\Triple{d_0^\pm}{d^\pm}{d_1^\pm}$. Define
  \begin{align*}
    k = \frac{d^+ - (d_0^+ + d_1^+) - \left[d^- - (d_0^- + d_1^-)\right]}{2}.
  \end{align*}
  Since $d^\pm - (d_0^\pm + d_1^\pm) \equiv 0 \mod 2$, $k$ is an integer.
  Note that we then have the identity
  \begin{align}
    \label{JumpsRank2Class1ProofDiscussionD}
    d^+ = d_0^+ + 2k + d^- - d_0^- - d_1^- + d_1^+.
  \end{align}
  The maps $H_\pm$ are equivariantly homotopic (as non-singular maps
  $X \to S^2 \subset \R^3$) to maps $H'_\pm$ which are assumed to be
  the concatenations of simpler maps, corresponding to the triple
  decomposition
  \begin{align}
    \label{JumpCurveTripleDecomposition}
    \Triple{d_0^\pm}{d_0^\pm}{0} &\bullet \Triple{0}{k}{k}\\
                                 &\bullet \Triple{k}{\pm k}{0}\nonumber \\
                                 &\bullet \Triple{0}{d^--(d_0^- + d_1^-)}{0}\nonumber\\
                                 &\bullet \Triple{0}{d_1^\pm}{d_1^\pm}\nonumber.
  \end{align}
  If the original maps $H_\pm$ are already normalized along any of the
  fixed point circles and if they have the same degree there, then the
  homotopy to the maps described by the aforementioned triples can be
  chosen relatively to those boundary circles. After this reduction to
  the maps $H'_\pm$, it suffices to prove the existence of a jump
  curve from $H'_-$ to $H'_+$. The restrictions of the maps for the
  triple (\ref{JumpCurveTripleDecomposition}) to the cylinder $Z$ are
  defined as the concatenation of five maps $H^j$, each defined on one
  fifth $Z_j$ of the cylinder $Z$, $j=0,\ldots,4$. We can assume that
  each map $H^j$ is normalized on the boundary circles of its
  subcylinder $Z_j$. Note that $H'_-$ and $H'_+$ agree on
  $Z_1 \cup Z_3$ and differ only on $Z_0 \cup Z_2 \cup Z_4$. This
  allows us to define a curve
  \begin{align*}
    H'\colon [-1,1] \times (Z,C) &\to (\R^3,\{z=0\})\\
    (t,p) &\mapsto
    \begin{cases}
      H^{\Triple{d_0^\pm}{d_0^\pm}{0}}(t,p) & \;\text{if $p \in Z_0$}\\
      H'_-(t,p) & \;\text{if $p \in Z_1$}\\
      H^{\Triple{k}{\pm k}{0}}(t,p) & \;\text{if $p \in Z_2$}\\
      H'_-(t,p) & \;\text{if $p \in Z_3$}\\
      H^{\Triple{0}{d_1^\pm}{d_1^\pm}}(t,p) & \;\text{if $p \in Z_4$}
    \end{cases}
  \end{align*}
  Note that if $d_0^+ = d_0^-$ or $d_1^+ = d_1^-$, then this homotopy
  $H_t$ is relative to the respective boundary circles on which the
  fixed point degrees agree.
  Equivariant extension of $H'$ yields a $G$-map
  \begin{align*}
    \widehat{H'}\colon [-1,1] \times X \to \R^3
  \end{align*}
  satisfying the desired properties (see
  lemma~\ref{Class1EquivariantExtension}). By
  lemma~\ref{JumpsRank2Class1FixpointDegree} we obtain $|d_0^+ -
  d_0^-|$ singular points from $H^{\Triple{d_0^\pm}{d_0^\pm}{0}}$,
  $|d_1^+ - d_1^-|$ singular points from
  $H^{\Triple{0}{d_1^\pm}{d_1^\pm}}$. To this we have to add the
  singular set added by $H^{\Triple{k}{\pm k}{0}}$, which (by
  lemma~\ref{JumpsRank2Class1TotalDegree}) and by
  (\ref{JumpsRank2Class1ProofDiscussionD}) consists of
  \begin{align*}
    2k = d^+ - (d_0^+ + d_1^+) - \left[d^- - (d_0^- + d_1^-)\right]
  \end{align*}
  points. In total these are
  \begin{align*}
    |d_0^+ - d_0^-| + |d_1^+ - d_1^-| + d^+ - (d_0^+ + d_1^+) -
    \left[d^- - (d_0^- + d_1^-)\right]
  \end{align*}
  singular points on the torus.
\end{proof}

The following corollary is a reformulation of
theorem~\ref{JumpsRank2Class1Corollary} using the isomorphism $\R^3
\cong i\mathfrak{su}_2$:
\begin{corollary}
  \label{JumpsRank2Class1Corollary}
  Let the torus $X$ be equipped with the type I involution and let
  \begin{align*}
    H_\pm\colon X \to i\mathfrak{su}_2\smallsetminus\{0\} \subset \mathcal{H}_{(1,1)} 
  \end{align*}
  be two $G$-maps with the triples
  $\Triple{d_0^\pm}{d^\pm}{d_1^\pm}$. Then there exists a jump curve
  \begin{align*}
    H\colon [-1,1] \times X \to i\mathfrak{su}_2 \subset \mathcal{H}_{(1,1)}
  \end{align*}
  from $H_-$ to $H_+$ such that the singular set $S(H_0) \subset X$
  consists of
  \begin{align*}
    |d_0^+ - d_0^-| + |d_1^+ - d_1^-| + d^+ - (d_0^+ + d_1^+) - \left[d^- - (d_0^- + d_1^-)\right]
  \end{align*}
  isolated points.
\end{corollary}

\subsection{Type II}

As in the classification chapter, it is possible to deduce results for
the type II involution from the respective results of the type I
involution. Let $(X,T)$ be the torus equipped with the type II
involution and $(X',T')$ be the torus equipped with the type I
involution. As an immediate corollary to
theorem~\ref{JumpsRank2Class1} note:
\begin{corollary}
  \label{JumpsRank2Class2}
  Let $H_\pm\colon X' \to S^2 \subset \R^3$ be two $G$-maps with
  triples $\Triple{0}{d^\pm}{d_1^\pm}$. Then there exists a jump curve
  $H\colon [-1,1] \times X' \to \R^3$ from $H_-$ to $H_+$ such that
  the singular set $S(H_0) \subset X'$ consists of
  \begin{align*}
    |d_1^+ - d_1^-| + d^+ - d^- + d_1^- - d_1^+
  \end{align*}
  isolated points. In particular, if the maps $H_-$ and $H_+$ are
  normalized along the boundary circle $C_0$ (meaning that
  $H_\pm(C_0) = \{p_0\} \in E \subset S^2$), then the curve $H_t$ can
  be chosen such that $H_t(C_0) = \{p_0\}$ for all $t$.
\end{corollary}

\begin{proof}
  This is just the special case $d_0^\pm = 0$ of
  theorem~\ref{JumpsRank2Class1}.
\end{proof}

For reducing jump curves for the type II involution to jump curves for
the type I involution, recall the equivariant and orientation
preserving diffeomorphism
\begin{align*}
  \Psi\colon X'\setminus C_0 \xrightarrow{\;\sim\;} S^2\setminus \{O_\pm\}.
\end{align*}
from remark~\ref{Class1TorusWithSphereIdentification} which we have
used for identifying type I normalized $G$-maps $X' \to S^2$ with type
II normalized $G$-maps $S^2 \to S^2$. We can use the same method for
jump curves:
\begin{proposition}
  \label{JumpsRank2Class2CompleteSphere}
  Let $H_\pm$ be two type II normalized $G$-maps
  $S^2 \to S^2 \subset \R^3$ with degree pairs
  $\smash{\Pair{d_E^\pm}{d^\pm}}$. Then there exists a jump curve
  $H\colon [-1,1] \times S^2 \to \R^3$ from $H_-$ to $H_+$ such that
  the singular set $S(H_0)$ consists of
  \begin{align*}
    |d_E^+ - d_E^-| + d^+ - d^- + d_E^- - d_E^+
  \end{align*}
  isolated points in $X$.
\end{proposition}

\begin{proof}
  By lemma~\ref{LemmaClass2ReductionToClass1}, the maps $H_\pm$ induce
  type I normalized equivariant maps
  $H'_\pm\colon X' \to S^2 \subset \R^3$ with the degree triples
  $\Triple{0}{d^\pm}{d_E^\pm}$ such that
  \begin{align*}
    \Restr{H'_\pm}{X'\setminus C_0} = \Restr{H_\pm}{S^2\setminus\{O_\pm\}} \circ \Psi.
  \end{align*}
  Using corollary~\ref{JumpsRank2Class2} we obtain a jump curve
  $H'\colon [-1,1] \times X' \to \R^3$ from $H'_-$ to $H'_+$ such that
  the singular set $S(H'_0)$ consists of
  \begin{align}
    \label{SingularPointsJumpsRank2Class2CompleteSphere}
    |d_E^+ - d_E^-| + d^+ - d^- + d_E^- - d_E^+
  \end{align}
  points. In particular this curve has the property that
  $H'_t(C_0) = \{p_0\}$ for all times $t$. The last property allows us
  to define the curve
  \begin{align*}
    H\colon [-1,1] \times S^2 &\to \R^3\\
    (t,p) &\mapsto
    \begin{cases}
      H'\left(t,\Psi^{-1}(p)\right) & \;\text{if $p \not\in \{O_\pm\}$}\\
      p_0 & \;\text{if $p \in \{O_\pm\}$}\\
    \end{cases}
  \end{align*}
  Since $H'_t$ is non-singular for $t \not= 0$, the same applies to
  $H$. By construction
  $H_t$ is a jump curve from $H_-$ to $H_+$. The diffeomorphism $\Psi$
  maps the singular set $S(H'_0)$ onto the singular set $S(H_0)$, thus
  the number of singular points of $S(H_0)$ is also given by
  (\ref{SingularPointsJumpsRank2Class2CompleteSphere}). This finishes
  the proof.
\end{proof}

Now we can deduce the following result for type II jump curves:
\begin{theorem}
  \label{JumpsRank2Class2Complete}
  Let $X$ be the torus equipped with the type II involution and let
  $H_\pm$ be two $G$-maps $X \to S^2 \subset \R^3$ with degree pairs
  $\Pair{d_C^\pm}{d^\pm}$. Then there exists a jump curve
  $H\colon [-1,1] \times X \to \R^3$ from $H_-$ to $H_+$ such that the
  singular set $S(H_0)$ consists of
  \begin{align*}
    |d_C^+ - d_C^-| + d^+ - d^- + d_C^- - d_C^+
  \end{align*}
  isolated points in $X$.
\end{theorem}

\begin{proof}
  After a $G$-homotopy we can assume that both maps $H_\pm$ are type
  II normalized. Therefore they push down to maps $H'_\pm$ on the
  quotient $X/A$, which can be equivariantly identified with
  $S^2$. Thus $H'_\pm$ can be regarded as type II normalized maps
  $S^2 \to S^2$ and we have the identity
  $H_\pm = H'_\pm \circ \smash{\pi_{S^2}}$. By
  remark~\ref{Class2CorrespondenceXS2} the degree pairs of the maps
  $H'_\pm$ are the same as those of the original maps $H_\pm$. Now
  proposition~\ref{JumpsRank2Class2CompleteSphere} guarantees the
  existence of a jump curve $H'\colon I \times S^2 \to \R^3$ from
  $H'_-$ to $H'_+$ with
  \begin{align*}
    |d_C^+ - d_C^-| + d^+ - d^- + d_C^- - d_C^+
  \end{align*}
  singular points on $S^2$. Then the composition
  $H'_t \circ \smash{\pi_{S^2}}$ defines a jump curve from $H_-$ to
  $H_+$ with the same number of singular points.
\end{proof}

As in the case of the type I involution we use the isomorphism $\R^3
\cong i\mathfrak{su}_2$ to note the following reformulation:
\begin{corollary}
  \label{JumpsRank2Class2CompleteCorollary}
  Let $X$ be the torus equipped with the type II involution and let
  $H_\pm$ be two $G$-maps $X \to i\mathfrak{su}_2 \subset
  \mathcal{H}_{(1,1)}$ with degree pairs $\Pair{d_C^\pm}{d^\pm}$. Then
  there exists a jump curve
  \begin{align*}
    H\colon [-1,1] \times X \to i\mathfrak{su}_2 \subset \mathcal{H}_{(1,1)}
  \end{align*}
  from $H_-$ to $H_+$ such that the singular set $S(H_0)$ consists of
  \begin{align*}
    |d_C^+ - d_C^-| + d^+ - d^- + d_C^- - d_C^+
  \end{align*}
  isolated points in $X$.
\end{corollary}

\section{Maps into $\cl(\mathcal{H}_n)$}

In this section we assume $n > 2$. Recall that $\mathcal{H}_{(p,q)}$
contains the orbit $U(n).\I{p}{q}$ as an equivariant strong
deformation retract.  This orbit can be equivariantly identified with
the Grassmann manifold $\Gr_p(\C^n)$. In order to employ the method
described in the previous section in the higher dimensional setting,
we embed the space $i\mathfrak{su}_2$ into $\cl(\mathcal{H}_{(p,q)})$
such that $i\mathfrak{su}_2\setminus\{0\} \subset U(n).\I{p}{q}$ and
the origin in $i\mathfrak{su}_2$ defines a singular matrix in the
boundary of this orbit. As we have seen earlier, the topology on the
mapping space is only non-trivial in the case of a
mixed-signature. Thus we can assume that $0 < p,q < n$. The space
$i\mathfrak{su}_2$ can be equivariantly embedded into
$\cl(\mathcal{H}_{(p,q)})$ as an affine subspace via
\begin{align}
  \label{isu2Embedding}
  \iota\colon i\mathfrak{su}_2 &\hookrightarrow \cl(\mathcal{H}_{(p,q)})\\
  \xi &\mapsto \begin{pmatrix}
    \I{p-1} & 0 & 0\\
    0 & \xi & 0 \\
    0 & 0 & -\I{q-1}
  \end{pmatrix}\nonumber
\end{align}
Similarly, the group $U(2)$ can be embedded into $U(n)$ via
\begin{align*}
  U(2) &\hookrightarrow U(n)\\
  U &\mapsto \begin{pmatrix}
    \I{p-1} & 0 & 0\\
    0 & U & 0 \\
    0 & 0 & -\I{q-1}
  \end{pmatrix}.
\end{align*}
Since $U(n)$ acts on $\mathcal{H}_{(p,q)}$ via conjugation, so does
$U(2)$, by regarding each element of $U(2)$ as being embedded in
$U(n)$. The $U(2)$-orbit of $\I{p}{q}$ is the same as
$\iota(U(2).\I{1}{1})$. It is contained as a subspace in the orbit
$U(n).\I{p}{q}$ and corresponds -- under the isomorphism $\Gr_p(\C^n)
\cong U(n).\I{p}{q}$ -- to the Schubert variety $\mathcal{S} \cong
\Proj_1$ (see the remarks about the Schubert variety $\mathcal{S}$ on
p.~\pageref{SchubertVarietyDiscussion}). This follows from the fact
that $U(2)$ acts transitively on
\begin{align*}
  \mathcal{S} = \left\{E \in \Gr_p(\C^n)\colon \C^{p-1} \subset E \subset \C^{p+1}\right\}.
\end{align*}
Under the above embedding $i\mathfrak{su}_2 \hookrightarrow
\cl(\mathcal{H}_{(p,q)})$, the origin of $i\mathfrak{su}_2$
corresponds to a singular matrix, which is contained in the boundary
$\partial\mathcal{H}_{(p,q)}$.

Since $\mathcal{H}_{(p,q)}$ has $\Gr_p(\C^n)$ as equivariant
deformation retract, these two spaces and also their real parts have
the same fundamental group:
\begin{align*}
  \pi_1\left(\left(\mathcal{H}_{(p,q)}\right)_\R\right) \cong \pi_1\left(\left(\Gr_p\left(\C^n\right)\right)_\R\right) \cong C_2.
\end{align*}
We write $C_2$ additively as the group $\Z_2$ and identify elements of
$\pi_1((\mathcal{H}_{(p,q)})_\R)$ with $\Z_2$. Thus we can regard the
restrictions of maps $X \to \mathcal{H}_{(p,q)}$ to the boundary
circle(s) in $X$ as defining an element in
$\pi_1((\mathcal{H}_{(p,q)})_\R) \cong \{0,1\}$.

\begin{theorem}
  \label{JumpsRankN}
  Assume $n = p + q > 2$. Let $X$ be the torus equipped with the type
  I (resp. type II) involution and let $H_\pm$ be two $G$-maps $X \to
  \mathcal{H}_{(p,q)}$ with the degree triples
  $\smash{\Triple{m_0^\pm}{d^\pm}{m_1^\pm}}$ resp. the degree pairs
  $\smash{\Pair{m_C^\pm}{d^\pm}}$.  Then there exists a jump curve
  \begin{align*}
    H\colon [-1,1] \times X \to \cl(\mathcal{H}_{(p,q)})
  \end{align*}
  from $H_-$ to $H_+$ such that the singular set $S(H_0)$ consists of
  \begin{align*}
    & |m_0^+ - m_0^-| + |m_1^+ - m_1^-| + d^+ - (m_0^+ + m_1^+) - \left[d^- - (m_0^- + m_1^-)\right]\\
    \text{resp.}\; & |m_C^+ - m_C^-| + d^+ - d^- + m_C^- - m_C^+
  \end{align*}
  isolated points.
\end{theorem}

\begin{proof}
  Let $H_\pm$ be $G$-maps $X \to \mathcal{H}_{(p,q)}$ with degree
  triples $\Triple{m_0^\pm}{d^\pm}{m_1^\pm}$ resp. degree pairs
  $\Pair{m_C^\pm}{d^\pm}$. By applying
  proposition~\ref{LemmaReductionOfGrToP1} we can assume that the maps
  $H_\pm$ have their images contained in the Schubert variety
  $\mathcal{S} \subset \Gr_p(\C^n)$. The Schubert variety is the
  $U(2)$-orbit of $\I{p}{q}$, which is $\iota(U(2).\I{1}{1})$. That
  is, it is contained in the image of the embedding
  \begin{align}
    \label{TopJumpsGrIotaEmbedding}
    \iota\colon i\mathfrak{su}_2 \hookrightarrow \cl(\mathcal{H}_{(p,q)})
  \end{align}
  from (\ref{isu2Embedding}). This means that we can assume that the
  maps $H_\pm$ are of the form
  \begin{align*}
    H_\pm = \iota \circ H'_\pm\colon X \to \cl(\mathcal{H}_{(p,q)})
  \end{align*}
  where the $H'_\pm$ are $G$-maps
  \begin{align*}
    H'_\pm\colon X \to U(2).\I{1}{1} \subset i\mathfrak{su}_2
  \end{align*}
  By construction, their degree triples resp. their degree pairs of
  $H'_\pm$ agree with those of $H_\pm$.
  Now corollary~\ref{JumpsRank2Class1Corollary} (type I)
  resp. corollary~\ref{JumpsRank2Class2CompleteCorollary} (type II)
  implies the existence of a jump curve
  \begin{align*}
    H'\colon [-1,1] \times X \to i\mathfrak{su}_2
  \end{align*}
  from $H'_-$ to $H'_+$ such that the singular set $S(H_0)$ consists
  of
  \begin{align}
    \label{TopJumpsGrSingularPoints}
    & |m_0^+ - m_0^-| + |m_1^+ - m_1^-| + d^+ - (m_0^+ + m_1^+) - \left[d^- - (m_0^- + m_1^-)\right]\\
      \text{resp.}\; & |m_C^+ - m_C^-| + d^+ - d^- + m_C^- - m_C^+\nonumber
  \end{align}
  isolated points in $X$. Composing $H'_t$ with $\iota$ yields a jump
  curve
  \begin{align*}
    H = \iota \circ H'\colon X \to \cl(\mathcal{H}_{(p,q)})
  \end{align*}
  from $H_-$ to $H_+$ whose number of isolated points is the same as
  that of $H'_t$, namely that given in
  (\ref{TopJumpsGrSingularPoints}).
\end{proof}

\section{An Example from Physics}

In this section we present an example jump curve coming from physics
(see e.\,g. \cite[p.~11]{PrimerTopIns}). For this we regard the torus
$X$ as $\R^2/\Lambda$, where
\begin{align*}
  \Lambda = \LatGen{
    \begin{pmatrix}
      2\pi \\
      0
    \end{pmatrix}}{
    \begin{pmatrix}
      0 \\
      2\pi
    \end{pmatrix}}.
\end{align*}
The orientation on $X$ is the orientation inherited from
$\R^2$. Employing physics notation for the coordinates on the torus
phase space, we define the jump curve as
\begin{align}
  H_t(q,p) = \begin{pmatrix}
    t - \cos q - \cos p & \sin q - i \sin p \\
    \sin q + i \sin p & -(t - \cos q - \cos p)
  \end{pmatrix}.
\end{align}
Each map $H_t$ is equivariant with respect to the type I
involution. This curve differs from the jump curves we have
constructed earlier in that it depends on a continuous parameter $t
\in \R$ with multiple jumps taking place as $t$ varies. These occur
precisely at $t=-2$, $t=0$ and at $t=2$. The associated curve of scaled
maps into $S^2 \subset \R^3$ is
\begin{align*}
  \widetilde{H}_t(q,p) =
  c_t(q,p)
  \begin{pmatrix}
    t - \cos q - \cos p\\
    \phantom{-}\sin q\\
    -\sin p
  \end{pmatrix},
\end{align*}
where
\begin{align*}
  c_t(q,p) = \frac{1}{\sqrt{(t - \cos q \cos p)^2 + (\sin q)^2 + (\sin p)^2}}.
\end{align*}
In the following we compute the degree triples $\mathcal{T}(H_t)$. For
this it is convenient to decompose $\smash{\widetilde{H}}_t$ as
\begin{align*}
  \widetilde{H}_t(q,p) =
  c_t(q,p)\left[
  \begin{pmatrix}
    t - \cos p\\
    0\\
    -\sin p
  \end{pmatrix} + \begin{pmatrix}
    - \cos q\\
    \sin q\\
    0
  \end{pmatrix}\right].
\end{align*}
Now we can regard $\smash{\widetilde{H}}_t$ as a family of circles in
the $x,y$-plane whose centers in $\R^3$ varies with $t$ and $p$. It
follows that $\widetilde{H}_t$ is not surjective (as a map to $S^2$)
for $|t| \in (2,\infty)$, therefore we have
$\smash{\deg \widetilde{H}_t} = 0$ for these $t$. Furthermore, when
$t$ is in a neighborhood of the origin and transitions from a negative
real number to a positive real number, then the total degree of
$\smash{\widetilde{H}}_t$ flips its sign. Hence it suffices to compute
the total degree of e.\,g. $\smash{\widetilde{H}}_1$. This can be done
by counting preimage points, for example in the fiber
$\smash{\widetilde{H}}_1^{-1}(y_0)$ where $y_0 = \ltrans{(0,1,0)}$. A
computation shows that the $y_0$-fiber consists only of the point
$(q_0,p_0) = (\text{\textonehalf}\pi,0)$.  Using the orientation of
$S^2$ induced by an outer normal vector field one can find that the
orientation sign of the map $\smash{\widetilde{H}}_1$ at $(q_0,p_0)$
is positive. This then implies that $\deg \smash{\widetilde{H}}_t = 1$
for $t \in (0,2)$ and $\deg \smash{\widetilde{H}}_t = -1$ for
$t \in (-2,0)$.

Regarding the fixed point degrees: For $p=0$ resp. $p=\pi$ the map
$\smash{\widetilde{H}}_t(p,\cdot)$ can be written as
\begin{align}
  \label{ExampleFixpointMapP=0}
  & q \mapsto
  c_t(q,0)\left[
  \begin{pmatrix}
    t - 1\\
    0 \\
    0
  \end{pmatrix} +
  \begin{pmatrix}
    -\cos q \\
    \phantom{-}\sin q \\
    0
  \end{pmatrix}\right]\\
\text{resp.}\;\;
& q \mapsto
  c_t(q,\pi)\left[
  \begin{pmatrix}
    t + 1\\
    0 \\
    0
  \end{pmatrix} +
  \begin{pmatrix}
    -\cos q \\
    \phantom{-}\sin q \\
    0
  \end{pmatrix}\right]\nonumber
\end{align}
Using the orientation on the equator $E \subset S^2$ as defined on
p.~\pageref{DegreeOneMapOnEquator} it follows that the fixed point
degree for $p=0$ is $0$ for $t \in \R\smallsetminus[0,2]$, and $-1$
for $t \in (0,2)$. The fixed point degree for $p=\pi$ is $0$ for
$t \in \R\smallsetminus[-2,0]$ and $-1$ for $t \in (-2,0)$. To
summarize the above:
\begin{remark}
  The following table lists the degree triples for $H_t$ depending on
  $t$, where $t$ is assumed to be in one of the $t$-intervals such
  that $H_t$ is non-singular:
  \begin{center}
    \footnotesize
    \begin{tabular}[h]{ll}
      \toprule
      \textbf{$t$-interval} & \textbf{degree triple $\mathcal{T}(H_t)$} \\
      \midrule
      $(-\infty,-2)$ & $\Triple{0}{0}{0}$ \\
      $(-2,0)$ & $\Triple{0}{-1}{-1}$ \\
      $(0,2)$ & $\Triple{-1}{1}{0}$ \\
      $(2,\infty)$ & $\Triple{0}{0}{0}$ \\
      \bottomrule
    \end{tabular}
  \end{center}
\end{remark}

Now we consider a generalization of the above map $H_t$. We define
\begin{align}
  H_t^m(q,p) = c_t(q,mp) \begin{pmatrix}
    t - \cos q - \cos(mp) & \sin q - i \sin(mp) \\
    \sin q + i \sin (mp) & -(t - \cos q - \cos(mp))
  \end{pmatrix},
\end{align}
with $m$ a positive integer. The maps $H_t^m$ are still type I
equivariant. They can be expressed as the composition
$H_t \circ p_m(q,p)$, where $p_m$ is the following $m:1$-cover of $X$:
\begin{align*}
  p_m\colon X &\to X\\
  (q,p) &\mapsto (q,mp).
\end{align*}
In particular, $p_m$ has degree $m$, which implies that $\deg H_t^m =
m\deg H_t$. Let us now compute the fixed point degrees of $H_t^m$. These
depend on the parity of $m$. If $m$ is even, then the
$2\pi$-periodicity of the trigonometric functions implies that the
fixed point degrees for $p=0$ and $p=\pi$ must be the same. They are
given by the map (\ref{ExampleFixpointMapP=0}). Therefore, the
fixed point degrees for even $m$ must be $(0,0)$ for $t \in
\R\smallsetminus[0,2]$ and $(-1,-1)$ otherwise. The periodicity of the
trigonometric functions implies that fixed point degrees for odd $m$ are
the same as those for $m=1$. We obtain the following result:
\begin{remark}
  The following table lists the degree triples for $H_t^m$ for even
  $m$ varying with $t$ such that $H_t^m$ is non-singular:
  \begin{center}
    \footnotesize
    \begin{tabular}[h]{ll}
      \toprule
      \textbf{$t$-interval} & \textbf{degree triple $\mathcal{T}(H^m_t)$} \\
      \midrule
      $(-\infty,-2)$ & $\Triple{0}{0}{0}$ \\
      $(-2,0)$ & $\Triple{0}{-m}{0}$ \\
      $(0,2)$ & $\Triple{-1}{m}{-1}$ \\
      $(2,\infty)$ & $\Triple{0}{0}{0}$ \\
      \bottomrule
    \end{tabular}
  \end{center}
\end{remark}
In particular, this example shows how the degree triples vary with $t$
in such a way that the condition $d \equiv d_0 + d_1 \mod 2$ is always
satisfied.


\appendix
\chapter{Appendix}
\label{ChapterAppendix}

This chapter lists some standard material, which is included for the
convenience of the reader.

\section{Topology}

Taken from \cite{Hatcher}:
\begin{proposition}
  \label{HEP} (Homotopy Extension Property) Let $X$ be a CW complex
  and $A \subset X$ a CW subcomplex. Then the CW pair $(X,A)$ has the
  homotopy extension property (HEP) for all spaces.
\end{proposition}
\begin{proof}
  See \cite[p.\,15]{Hatcher}.
\end{proof}

\begin{lemma}
  \label{LemmaCurvesNullHomotopic}
  Let $X$ be a topological space and $f,g\colon I \to X$ be two curves
  in $X$ with $f(0) = g(0) = p$ and $f(1) = g(1) = q$. Then $g^{-1}
  \ast f$ is null-homotopic if and only if $f$ and $g$ are homotopic
  (with fixed endpoints).
\end{lemma}
\begin{proof}
  The one direction is trivial: When $f$ and $g$ are homotopic, then
  we can define a null-homotopy $H\colon I \times I \to X$ as follows:
  During $0 \leq t \leq \text{\textonehalf}$ use a homotopy from $f$
  to $g$ to build a homotopy from $g^{-1} \ast f$ to $g^{-1} \ast
  g$. Then, during $\text{\textonehalf} \leq t \leq 1$ we shrink $g^{-1} \ast
  g$ the constant curve at $p$.

  Now assume that we are given a homotopy from $\gamma_1^{-1} \ast
  \gamma_0$ to the constant curve at $p$. This is a map $H\colon I
  \times I \to X$. On the left, top and right boundary of the square
  $I$ the map $H$ takes on the value $p$, on the bottom boundary,
  i.\,e. $H_0$, this is the curve $\gamma^{-1}_1 \ast \gamma_0$. We use
  this map to construct a new map $\gamma\colon I \times I \to X$
  which is a homotopy from $\gamma_0$ to $\gamma_1$ as follows:

  We define $\gamma_t\colon I \to X$ to be the map $I
  \xhookrightarrow{\iota_t} I \times I \xrightarrow{H} X$ where the
  embedding $\iota_t$ is defined as follows:
  \begin{align*}
    \iota_t\colon I &\to I \times I\\
    s &\mapsto \frac{1}{\max \{|\cos \pi(1-t)|,|\sin \pi(1-t)|\}} \begin{pmatrix}
      \frac{\cos \pi(1-t)}{2}\\
      i \sin \pi(1-t)
    \end{pmatrix} +
    \begin{pmatrix}
      1\\
      2
    \end{pmatrix}
  \end{align*}
  With this definition, $\iota_0$ is the embedding $s \mapsto
  \text{\textonehalf}(1-s)$ and $\iota_1$ is the embeding $s \mapsto
  \text{\textonehalf}(1+s)$. Hence, $\gamma_t$ is indeed a homotopy
  from $\gamma_0$ to $\gamma_1$. Furthermore, the construction makes
  sure that for every $t$, $\iota_t(0) = \text{\textonehalf}$ and
  $\iota_t(1)$ is always contained in the left, top or right boundary
  of $I \times I$. This translates to the fact that $\gamma_t(0) = p$
  and $\gamma_t(1) = q$ for all $t$.
\end{proof}

\begin{theorem}
  \label{WhitneyApproximationTheorem}
  (Taken from John Lee, Whitney Approximation Theorem) Suppose $N$ is
  a smooth manifold with or without boundary, $M$ is a smooth manifold
  (without boundary), and $F\colon N \to M$ is a continuous map. Then
  $F$ is homotopic to a smooth map. If $F$ is already smooth on a
  closed subset $A \subseteq N$, then the homotopy can be taken to be
  relative to $A$.
\end{theorem}
\begin{proof}
  See \cite[p.\,141]{JohnLee}.
\end{proof}

Here is a theorem by Whitney, taken from \cite{JohnLee}:
\begin{theorem}
  \label{SmoothApproximation0}
  Let $N$ and $M$ be smooth manifolds and let $F\colon N \to M$ be a
  map. Then $F$ is homotopic to a smooth map
  $\smash{\widetilde{F}}\colon N \to M$. If $F$ is smooth on a closed
  subset $A \subset N$, then the homotopy can be taken to be relative
  to $A$.
\end{theorem}
\begin{proof}
  See \cite[p.~142]{JohnLee}
\end{proof}

The following two statements can be found in \cite{Bredon}:
\begin{theorem}
  \label{SmoothApproximation1}
  Let $G$ be a compact Lie group acting smoothly on the manifolds $M$
  and $N$. Let $\varphi\colon M \to N$ be an equivariant map. Then
  $\varphi$ can be approximated by a smooth equivariant map
  $\psi\colon M \to N$ which is equivariantly homotopic to $\varphi$
  by a homotopy approximating the constant homotopy. Moreover, if
  $\varphi$ is already smooth on the closed invariant set
  $A \subset M$, then $\psi$ can be chosen to coincide with $\varphi$
  on $A$, and the homotopy between $\varphi$ and $\psi$ to be constant
  there.
\end{theorem}
\begin{proof}
  See \cite[p.\,317]{Bredon}.
\end{proof}

\begin{corollary}
  \label{SmoothApproximation2}
  Let $G, M, N$ be as in [the previous theorem]. Then any equivariant
  map $M \to N$ is equivariantly homotopic to a smooth equivariant
  map. Moreover, if two smooth equivariant maps $M \to N$ are
  equivariantly homotopic, then they are so by a smooth equivariant
  homotopy.
\end{corollary}
\begin{proof}
  \cite[p.\,317]{Bredon}.
\end{proof}

The following theorem is taken (and translated) from
\cite{StoeckerZieschang}:
\begin{theorem}
  \label{TopologyMapInducedOnQuotient}
  Let $f\colon X \to Y$ be a continuous map, which is compatible with
  given equivalence relations $R$ resp. $S$ on $X$ resp. $Y$ (which
  means: $x \sim_R x'$ implies $f(x) \sim_S f(x')$). Then
  $f'([x]_R) = [f(x)]_S$\footnote{Here $[x]_R$ denotes the equivalence
    class of $x$ under the relation $R$ (likewise for $y$ and $S$).}
  defines a continuous map $f'\colon X/R \to Y/S$; it is called the
  map induced by $f$. If $f$ is a homeomorphism and $f^{-1}$ is also
  compatible with the relations $R$ resp. $S$, then the induced map
  $f'$ is also a homeomorphism.
\end{theorem}
\begin{proof}
  See \cite[p.~9]{StoeckerZieschang}.
\end{proof}

The same applies to the following
\begin{theorem}
  \label{QuotientOfCWIsCW}
  Let $(X,A)$ be a CW pair and $A \neq \emptyset$, then $X/A$ is a
  CW complex with the zero cell $[A] \in X/A$ and the cells of the
  form $p(e)$, where $p\colon X \to X/A$ denotes the identifying map
  and $e$ is a cell in $X\setminus A$.
\end{theorem}
\begin{proof}
  See \cite[p.~93]{StoeckerZieschang}.
\end{proof}

Taken from \cite{May}:
\begin{theorem}
  \label{G-HELP} (G-HELP) Let $A$ be a subcomplex of a $G$-CW
  complex $X$ of dimension $\nu$ and let $e\colon Y \to Y$ be a
  $\nu$-equivalence. Suppose given maps $g\colon A \to Y$, $h\colon A
  \times I \to Z$, and $f\colon X \to Z$ such that $e \circ g = h
  \circ i_1$ and $f \circ i = h \circ i_0$ in the following diagram:
  \[
  \label{HELP-Diagram}
  \begin{tikzcd}
    A \arrow{rr}{i_0} \arrow{dd}{i} & & A \times I \arrow{dd} \arrow{dl}{h} & & A \arrow{ll}{i_1} \arrow{dl}{g} \arrow{dd}{i} \\
    & Z & & Y \arrow[crossing over,near start]{ll}{e} \\
    X \arrow{ru}{f} \arrow{rr}{i_0} & & X \times I \arrow[dashed]{ul}{\tilde{h}} & & X \arrow{ll}{i_1} \arrow[dashed]{lu}{\tilde{g}} \\
  \end{tikzcd}
  \]
  Then there exist maps $\smash{\tilde{g}}$ and $\smash{\tilde{h}}$
  that make the diagram commute.
\end{theorem}
\begin{proof}
  See \cite[p.~17]{May}.
\end{proof}

From this we can deduce a statement about equivariant homotopy
extensions for $G$-CW complexes:
\begin{corollary}
  \label{G-HEP}
  (Equivariant HEP) Let $G$ be a topological group.  Let $X$ and $Y$
  be $G$-CW complexes and $A$ a $G$-CW subcomplex of $X$. Then the
  $G$-CW pair $(X,A)$ has the equivariant homotopy extension
  property. That is, given a $G$-map $f\colon X \to Y$ and a
  $G$-homotopy $h\colon I \times A \to Y$. Then $h$ extends to a
  homotopy $H\colon I \times X \to Y$ such that $H_0 \equiv f$ and
  $\restr{H}{I \times A} \equiv h$.
\end{corollary}
\begin{proof}
  Set $Z = Y$ and let $f\colon Y \to Z$ be the identity. In particular
  this makes $f$ be a $\nu$-equivalence where we set $\nu(H) =
  \dim(X)$ for all subgroups $H \subset G$. Then, the equivariant
  homotopy extension property follows from the LHS square of the
  diagram (\ref{HELP-Diagram}).
\end{proof}

\begin{remark}
  \label{RemarkQuotientDegreeOne}
  Let $M$ be an $n$-dimensional, connected, closed, oriented
  CW manifold and $A \subset M$ such that
  \begin{enumerate}[(i)]
  \item $A$ is contained in the $n-1$ skeleton $M^{n-1}$ and
  \item $M/A$ is topologically an $n$-dimensional, connected, closed,
    orientable manifold.
  \end{enumerate}
  Then the quotient $M/A$ can be equipped with an orientation such
  that the projection map $M \to M/A$ has degree $+1$.
\end{remark}

\begin{proof}
  The manifold $M$ is assumed to be oriented, which corresponds to a
  choice of fundamental class $[M]$ in $H_n(M,\Z) \cong \Z$. Denote
  the quotient map with
  \begin{align*}
    \pi\colon M \to M/A.
  \end{align*}
  Since $A$ is contained in $M^{n-1}$, $\pi$ maps the $n$-cells in $M$
  homeomorphically to the $n$-cells in the quotient $M/A$. It follows
  that there exists an $n$-ball in one of the $n$-cells in $M/A$ such
  that $\pi^{-1}(B)$ is a single $n$-ball in $M$ which gets mapped
  homeomorphically to $B$.  Fix a fundamental class $[M/A]$ of
  $H_n(M/A,\Z)$. Using exercise~8 from \cite[p.~258]{Hatcher} we can
  conclude that the degree of $\pi$ is $\pm 1$, depending on wether it
  is orientation preserving or orientation reversing. In case it is
  orientation reversing we can chose the opposite orientation on $M/A$
  such that $\deg \pi = +1$.
\end{proof}

A direct consequence of the previous remark is:
\begin{remark}
  \label{RemarkMapOnQuotientDegreeUnchanged}
  Let $M$ and $N$ be $n$-dimensional, connected, closed, oriented
  manifolds and $f\colon M \to N$ a map of degree $d$. Furthermore,
  assume that $f$ is constant on $A \subset M$ such that it induces a
  map $f'\colon M/A \to N$. If the pair $(M,A)$ satisfies the
  assumptions of remark~\ref{RemarkQuotientDegreeOne}, then $f'$ also
  has degree $d$.
\end{remark}

\begin{proof}
  Denote the quotient map with $\pi\colon M \to M/A$. We then have the
  following commutative diagram:
  \[
  \xymatrix{
    M \ar[d]_\pi \ar[dr]^f\\
    M/A \ar[r]_{f'} & N
  }
  \]
  In homology we obtain:
  \[
  \xymatrix{
    H_n(M,\Z) \ar[d]_{\pi_*} \ar[dr]^{f_*}\\
    H_n(M/A,\Z) \ar[r]_{f'_*} & H_n(N,\Z)
  }
  \]
  By remark~\ref{RemarkQuotientDegreeOne} the quotient $M/A$ can be
  equipped with an orientation such that, after identifying all
  homology groups with $\Z$, $\pi_*$ is multiplication by $+1$ and
  $f'_*$ is multiplication by $d'$. Thus we have:
  \[
  \xymatrix{
    \Z \ar[d]_{\cdot +1} \ar[dr]^{\cdot d}\\
    \Z \ar[r]_{\cdot d'} & \Z
  }
  \]
  By commutativity of the diagram it follows that the degree of
  $f'$ is $'d = d$.
\end{proof}

\begin{proposition}
  \label{TopologyGrassmannians}
  Assume $0 < p < n$. Then:
  \begin{enumerate}[(i)]
  \item The second homology group $H_2(\Gr_p(\C^n),\Z)$ is infinite
    cyclic for all $p$ and $n$. Fixing a full flag of $\C^n$ induces a
    decomposition of $\Gr_p(\C^n)$ into Schubert cells. With respect
    to a fixed flag there exists a unique (complex) 1-dimensional
    Schubert variety, which can be regarded as the generator of
    $H_2(\Gr_p(\C^n),\Z)$.
  \item The fundamental group $\pi_1(\Gr_p(\R^n))$ is cyclic of order
    two unless $p=1$ and $n=2$, in which case it is infinite cyclic.
  \end{enumerate} \qed
\end{proposition}

\section{Hausdorff Dimension}
\label{AppendixHausdorffDimension}

For completeness we quote a theorem taken from
\cite[p.~515]{SchleicherArticle}:
\begin{theorem}
  \label{HDimProperties}
  Hausdorff dimension has the following properties:
  \begin{enumerate}[(1)]
  \item if $X \subset Y$, then $\hdim X \leq \hdim Y$;
  \item if $X_i$ is a countable collection of sets with $\hdim X_i
    \leq d$, then
    \begin{align*}
      \hdim \bigcup_i X_i \leq d
    \end{align*}
  \item if $X$ is countable, then $\hdim X = 0$;
  \item if $X \subset \R^d$, then $\hdim X \leq d$;
  \item if $f\colon X \to f(X)$ is a Lipschitz map, then $\hdim(f(X))
    \leq \hdim(X)$.
  \item if $\hdim X = d$ and $\hdim Y = d'$, then $\hdim(X \times Y) \geq d + d'$;
  \item if $X$ is connected and contains more than one point, then
    $\hdim X \geq 1$; more generally, the Hausdorff dimension of any
    set is no smaller than its topological dimension;
  \item if a subset $X$ of $\R^n$ has finite positive $d$-dimensional
    Lebesgue measure, then
    \begin{align*}
      \hdim X = d. 
    \end{align*}
  \end{enumerate}
\end{theorem}

Although the Hausdorff dimension is not invariant under homeomorphisms,
it is invariant under diffeomorphisms:
\begin{corollary}
  \label{HdimDiffeomorphismInvariant}
  If $f\colon X \to Y$ is a diffeomorphism (between metric spaces),
  then $\hdim X = \hdim Y$.
\end{corollary}
\begin{proof}
  Let $f\colon X \to Y$ be a diffeomorphism. In particular, $f$ and
  $f^{-1}$ are both Lipschitz continuous. Thus, by
  theorem~\ref{HDimProperties} (5) we obtain
  \begin{align*}
    & \hdim(Y) = \hdim(f(X)) \leq \hdim(X)\\
    \;\text{and}\; & \hdim(X) = \hdim(f^{-1}(Y)) \leq \hdim Y
  \end{align*}
  and the statement follows.
\end{proof}

\begin{corollary}
  \label{HDimViaSetDecomposition}
  If $X$ is the finite union of sets $X_i$, then
  \begin{align*}
    \hdim X = \max_i \hdim X_i.
  \end{align*}
\end{corollary}

\begin{proof}
  Each $X_i$ is contained in $X$, hence by
  theorem~\ref{HDimProperties} (1) we obtain
  \begin{align*}
    \hdim X_i \leq \hdim X \;\text{for all $i$}.
  \end{align*}
  In other words:
  \begin{align*}
    \max_i \hdim X_i \leq \hdim X
  \end{align*}
  On the other hand, noting the trivial fact
  \begin{align*}
    \hdim X_i \leq \max_i \hdim X_i
  \end{align*}
  and using (2) of the same theorem we can conclude
  \begin{align*}
    \hdim X = \hdim\left(\bigcup_i X_i\right) \leq \max_i(\hdim X_i).
  \end{align*}
  Thus we get the desired equation:
  \begin{align*}
    \hdim X = \max_i(\hdim X_i).
  \end{align*}
\end{proof}

\begin{proposition}
  \label{PropositionHausdorffPackingDimension}
  For sets $A, B \subset \R^n$ we have
  \begin{align*}
    \hdim(A) + \hdim(B) \leq \hdim(A \times B) \leq \hdim(A) + \pdim(B)
  \end{align*}
\end{proposition}

\begin{proof}
  See \cite[p.~1]{XiaoArticle}.
\end{proof}

For a submanifold $B \subset \R^n$ the packing dimension $\pdim(B)$,
the Hausdorff dimension $\hdim(B)$ and the usual manifold dimension
$\dim(B)$ agree\footnote{See e.\,g. \cite[p.~48]{Falconer}}. Thus we
obtain for submanifolds $A,B \subset \R^n$:
\begin{align}
  \label{HDimOfProduct}
  \hdim(A \times B) = \hdim(A) + \dim(B).
\end{align}
This allows us to prove the following:
\begin{corollary}
  \label{HDimOfTotalSpace}
  Let $E \xrightarrow{\;\pi\;} X$ be a fiber bundle over the smooth
  manifold $X$ where $E$ is a smooth submanifold of $\R^N$ such that
  the fiber $F$ is also a smooth manifold.
  Let $A$ be a subset in $X$. Then $\hdim(\pi^{-1}(A)) = \hdim(A) +
  \dim(F)$.
\end{corollary}

\begin{proof}
  The statement follows by using using a trivializing covering of $X$
  together with corollary~\ref{HdimDiffeomorphismInvariant},
  corollary~\ref{HDimViaSetDecomposition} and (\ref{HDimOfProduct}).
\end{proof}


\addchap{Notation}

{\footnotesize
\begin{longtable}[ht]{lp{11.6cm}}
  $\emptyset$ & The empty set\\
  $I$ & The closed unit interval $[0,1]$\\
  $\widehat{\R}$ & Compactified real line, $\widehat{\R} = \R \cup \{\infty\}$\\
  $\LatGen{\omega_1}{\omega_2}$ & The lattice in $\mathbb{C}$ generated by $\omega_1$ and $\omega_2$\\
  $C_n$ & The cyclic group of order $n$\\
  $\mathcal{M}(X,Y)$ & The space of maps $X \to Y$\\
  $\mathcal{M}((X,A),(Y,B))$ & The space of maps $(X,A) \to (Y,B)$\\
  $\mathcal{M}_G(X,Y)$ & The space of continuous $G$-maps $X \to Y$ (both $G$-spaces)\\
  $f \simeq g$ & The maps $f$ and $g$ are homotopic\\
  $f \simeq_G g$ & The maps $f$ and $g$ are $G$-equivariantly homotopic\\
  $[X,Y]_G$ & The $G$-homotopy classes of maps $X \to Y$\\
  $\Triple{d_0}{d}{d_1}$ & Degree triple for equivariant homotopy classes (see p.~\pageref{DefinitionTriple})\\
  $\mathcal{L}X$ & The free loop space of $X$\\
  $\Omega X$ & The space of based loops in $X$\\
  $\Omega (X,x_0)$ & The space of based loops in $X$ with basepoint $x_0$\\
  $\mathbb{H}^+$ & The upper half plane in $\mathbb{C}$\\
  $\pi(X,p)$ & The fundamental group of $X$ with basepoint $p$\\
  $\pi(X;p,q)$ & The homotopy classes of curves in $X$ with fixed endpoints $p$ and $q$\\
  $\mathcal{H}_n$ & The set of complex, hermitian $n \times n$ matrices\\
  $\mathcal{H}^*_n$ & The set of complex, hermitian, non-singular $n \times n$ matrices\\
  $\mathcal{H}_{(p,q)}$ & The subset of $\mathcal{H}_{p+q}^*$ consisting of matrices with eigenvalue signature $(p,q)$\\
  $c_p\colon X \to Y$ & The constant map, which sends every $x \in X$ to $p \in Y$\\
  iff & if and only if\\
  LHS, RHS & Left-hand side, right-hand side\\
  $X^G$ & For a $G$-space $X$, $X^G = \{x \in X\colon G(x) = \{x\}\}$\\
  $\ltrans{M}$ & Matrix transpose of the matrix $M$\\
  $M^*$ & For a matrix $M$, $M^*$ denotes its conjugate-transpose, i.\,e. $M^* = \ltrans{\overline{M}}$\\
  $e^n$ & An open $n$-cell\\
  $\gamma_2 \ast \gamma_1$ & Concatenation of curves $\gamma_1$ and $\gamma_2$ where $\gamma_1(1) = \gamma_2(0)$\\
  $\gamma^{-1}$ & For a curve $\gamma$, $\gamma^{-1}$ denotes the same curve with reversed time\\
  $\mathbb{P}_n$ & The $n$-dimensional complex projective space\\
  $\mathbb{RP}_n$ & The $n$-dimensional real projective space\\
  $\I_n$ & The $n \times n$ identity matrix\\
  $\I{p}{q}$ & The block diagonal matrix $\Diag(\I{p},-\I{q})$\\
  $\hdim X$ & The Hausdorff dimension of the topological space $X$\\
  $\pdim X$ & The packing dimension of the topological space $X$\\
  $\Gr_k(\C^n)$ & Grassmannian of $k$-dimensional complex subvectorspaces in $\C^n$\\
  $\Gr_k(\R^n)$ & Grassmannian of $k$-dimensional real subvectorspaces in $\R^n$\\
  $\cl(X)$ & Closure of the topological space $X$\\
  $\partial X$ & Boundary of the topological space $X$
\end{longtable}}

All maps are assumed to be continuous, unless otherwise stated. When
discussing the dimension of complex geometric objects we refer to its
\emph{complex} dimension unless we explicitely use the term \emph{real
  dimension}.


\listoffigures

\bibliography{literature}
\bibliographystyle{plain}

\end{document}